\documentclass[11pt]{amsart}

\newcommand{\jrev}[1]{{{#1}}}
\newcommand{\rev}[1]{{#1}}

\newcommand{\cog}{C_\og}
\newcommand{\Hd}{\hat{H}_\dv}
\newcommand{\Ht}{\hat{H}_\tp}
\newcommand{\Ed}{\hat{E}_\dv}
\newcommand{\Et}{\hat{E}_\tp}
\newcommand{\eps}{\epsilon}
\newcommand{\tp}{{\scriptscriptstyle{\top}}}
\newcommand{\dv}{{\scriptscriptstyle{\dashv}}}
\newcommand{\ii}{\imath}
\newcommand{\og}{\omega}
\DeclareMathOperator{\osc}{osc}
\newcommand{\veps}{\varepsilon}
\newcommand{\ncrl}[2]{\| {#1}\|_{{#2},\mathrm{curl}}}

\newcommand{\gradt}{\ensuremath{\mathop{\mathrm{grad}_{{\tp}}}}}
\newcommand{\divt}{\ensuremath{\mathop{\mathrm{div}_{\tp}}}}

\newcommand{\nul}{\mathrm{nul}}
\newcommand{\rank}{\mathrm{rank}}
\newcommand{\Bg}{B_{p+3}^{\mathrm{grad}}}
\newcommand{\Bc}{B_{p+3}^{\mathrm{curl}}}
\newcommand{\Bd}{B_{p+3}^{\mathrm{div}}}
\newcommand{\vog}{\pigo v} 
\newcommand{\vpi}{\varPi}
\newcommand{\pigo}{\mathop{\varPi_{0}^\mathrm{g}}}
\newcommand{\pig}{\mathop{\varPi_{p+3}^\mathrm{grad}}}
\newcommand{\pid}{\mathop{\varPi_{p+3}^\mathrm{div}}}
\newcommand{\pic}{\mathop{\varPi_{p+3}^\mathrm{curl}}}
\newcommand{\pil}{\mathop{\varPi_{p+2}}}

\newcommand{\pich}{\mathop{\hat{\varPi}_{p+3}^\mathrm{curl}}}

\newcommand{\diam}{\mathop{\mathrm{diam}}}
\newcommand{\Pb}{P^{\bot}}
\newcommand{\Pbo}{P^{0,\bot}}
\newcommand{\Pbd}{P^{\bot}_{p+2}(\d K)}
\newcommand{\Pbod}{P^{0,\bot}_{p+2}(\d K)}
\newcommand{\Babuska}{Babu{\v{s}}ka}       
                 
\newcommand{\Poincare}{Poincar{\'{e}}}     
\newcommand{\Nedelec}{N{\'{e}}d{\'{e}}lec} 
\newcommand{\Pc}{P^\mathrm{c}}
\newcommand{\Yhopt}{Y_h^\mathrm{opt}}
\newcommand{\ip}[1]{\langle {#1} \rangle}
\newcommand{\trg}{\mathop{\mathrm{tr}_\mathrm{grad}}}
\newcommand{\trc}{\mathop{\mathrm{tr}_{\mathrm{curl},\tp}}}
\newcommand{\trp}{\mathop{\mathrm{tr}_{\mathrm{curl},{\dv}}}}
\newcommand{\trd}{\mathop{\mathrm{tr}_\mathrm{div}}}
\newcommand{\trgi}{\mathop{\mathrm{tr}_\mathrm{grad}^{-1}}}
\newcommand{\trci}{\mathop{\mathrm{tr}_{\mathrm{curl},{\dv}}^{-1}}}
\newcommand{\trpi}{\mathop{\mathrm{tr}_{\mathrm{curl},{\tp}}^{-1}}}
\newcommand{\trdi}{\mathop{\mathrm{tr}_\mathrm{div}^{-1}}}
\newcommand{\trgK}{\mathop{\mathrm{tr}_\mathrm{grad}^K}}
\newcommand{\trcK}{\mathop{\mathrm{tr}_{\mathrm{curl},{\dv}}^K}}
\newcommand{\trpK}{\mathop{\mathrm{tr}_{\mathrm{curl},{\tp}}^K}}
\newcommand{\trdK}{\mathop{\mathrm{tr}_\mathrm{div}^K}}
\newcommand{\om}{{\varOmega}}
\newcommand{\oh}{{\varOmega_h}}
\newcommand{\ds}{\displaystyle}
\newcommand{\ptl}{{\partial}}

\newcommand{\dom}{\mathop{\mathrm{dom}}}
\newcommand{\grad}{\mathop{\mathrm{grad}}}
\newcommand{\dive}{\mathop{\mathrm{div}}}
\newcommand{\curl}{\mathop{\mathrm{curl}}}
\newcommand{\Ho}{\ring{H}}
\newcommand{\Hdiv}[1]{H(\mathrm{div},#1)}
\newcommand{\Hdivo}[1]{\ring{H}(\mathrm{div},#1)}
\newcommand{\Hhdiv}[1]{H^{-1/2}(\mathrm{div},#1)}
\newcommand{\Hhdivo}[1]{\Ho^{-1/2}(\mathrm{div},#1)}

\newcommand{\Hcurl}[1]{H(\mathrm{curl},#1)}
\newcommand{\Hcurlo}[1]{\Ho(\mathrm{curl},#1)}
\newcommand{\Hhcurl}[1]{H^{-1/2}(\mathrm{curl},#1)}
\newcommand{\Hhcurlo}[1]{\Ho^{-1/2}(\mathrm{curl},#1)}
\renewcommand{\Omega}{\varOmega}
\newcommand{\D}{\mathcal{D}}
\newcommand{\CCC}{\mathbb{C}}
\newcommand{\RRR}{\mathbb{R}}
\newcommand{\qeg}{/\!/\!/}
\def\d{\partial}

\usepackage{mathabx}
\usepackage{accents}
\usepackage{amssymb}
\usepackage{amscd}
\usepackage{amsmath}
\usepackage{amsfonts}
\usepackage{mathrsfs}
\usepackage{graphicx}
\usepackage{subcaption}
\usepackage{color}
\usepackage{tikz}
\usetikzlibrary{chains}
\usetikzlibrary{arrows}
\usepackage{pgfplots}
\pgfplotsset{width=6.5cm}

\graphicspath{{./figures/}}

\newtheorem{theorem}{Theorem}
\newtheorem{lemma}[theorem]{Lemma}
\newtheorem{corollary}[theorem]{Corollary}

\numberwithin{theorem}{section} 
\numberwithin{equation}{section}
\theoremstyle{remark}
\newtheorem{remark}[theorem]{Remark}
\newtheorem{assumption}{Assumption}
\newtheorem{example}[theorem]{Example}

\begin{document}

\title[]{Breaking spaces and forms for the DPG method and applications including Maxwell equations}

\begin{abstract}
  Discontinuous Petrov Galerkin (DPG) methods are made easily
  implementable using {``broken'' test spaces, i.e., spaces of
    functions with no continuity constraints} across mesh element
  interfaces. Broken spaces derivable from a standard exact sequence
  of first order (unbroken) Sobolev spaces are of particular interest.
  A characterization of interface spaces that connect the broken
  spaces to their unbroken counterparts is provided.  Stability of
  certain formulations using the broken spaces can be derived from the
  stability of analogues that use unbroken spaces.  This technique is
  used to provide a complete error analysis of DPG methods for Maxwell
  equations with perfect electric boundary conditions. The technique
  also permits considerable simplifications of previous analyses of
  DPG methods for other equations. Reliability and efficiency
  estimates for an error indicator also follow. Finally, the
  equivalence of stability for various formulations of the same
  Maxwell problem is proved, including the strong form, the ultraweak
  form, and various forms in between.
\end{abstract}



\thanks{This work was partially supported 
  by the AFOSR under grant FA9550-12-1-0484, 
  by the NSF under grant
  DMS-1318916, and by the DFG
  via SPP~1748} 

\author[C. Carstensen]{C.~Carstensen}
\address{Department of Mathematics,
Humboldt-Universit\"at zu Berlin, Unter den Linden 6,
10099 Berlin, Germany}
\email{cc@math.hu-berlin.de}

\author[L. Demkowicz]{L.~Demkowicz}
\address{Institute for Computational Engineering and Sciences,
  The University of Texas at Austin, Austin, TX 78712, USA}
\email{leszek@ices.utexas.edu}

\author[J. Gopalakrishnan]{J.~Gopalakrishnan}
\address{PO Box 751, Portland State University, Portland, OR 97207-0751, USA}
\email{gjay@pdx.edu}

\maketitle

\section{Introduction}

When a domain $\om$ is partitioned into elements, a function in a
Sobolev space like $\Hcurl\om$ or $\Hdiv\om$ has continuity
constraints across element interfaces, e.g, the former has tangential
continuity, while the latter has continuity of its normal
component. If these continuity constraints are removed from the space,
then we {obtain}  ``broken'' Sobolev spaces.  Discontinuous Petrov Galerkin
(DPG) methods introduced in~\cite{DemkoGopal10a,DemkoGopal11} used
spaces of such discontinuous functions in broken Sobolev spaces to
localize certain computations. The studies in this paper begin by
clarifying this process of breaking Sobolev spaces. This process,
sometimes called hybridization, has been well studied within a
discrete setting. For instance, the hybridized Raviart-Thomas
method~\cite{BrezzForti91,RaviaThoma77} is obtained by
discretizing a variational formulation and then removing the
continuity constraints of the discrete space, i.e., by discretizing
first and then hybridizing. In contrast, in this paper, we identify
methods obtained by hybridizing first and then discretizing, a setting
more natural for DPG methods. We then take this idea {further} by
connecting the stability of formulations with broken spaces and
unbroken spaces, leading to the first convergence proof of a DPG
method for Maxwell equations.

The next section (Section~\ref{sec:break}) is devoted to {a study
  of } the interface spaces that arise when breaking Sobolev
spaces. These infinite-dimensional interface spaces can be used to
connect the broken and the unbroken spaces. The main result of
{Section~\ref{sec:break}, contained in Theorem~\ref{thm:duality}, makes}
this connection precise and provides an elementary
characterization (by duality) of the natural norms on these interface
spaces. {This theorem can be viewed as a generalization 
of a similar result in~\cite{RaviaThoma77a}.}

Having discussed breaking spaces, we proceed to break variational
formulations in Section~\ref{sec:break-forms}. The motivation for the
theory in that section is that some variational formulations set in
broken spaces have another closely related variational formulation set
in their unbroken counterpart. This is the case with all the
formulations on which the DPG method is based. The main observation
{of Section~\ref{sec:break-forms} is} a simple result
(Theorem~\ref{thm:hybrid}) {which in its abstract form seems to be
  already known in other studies~\cite{GargPrudhZee14}. 
In the DPG context,} it provides
sufficient conditions under which {\em stability of broken forms follow
from stability of their unbroken relatives}. As a consequence of this
observation, we are able to simplify many previous
analyses of DPG methods. The content of Sections~\ref{sec:break}
and~\ref{sec:break-forms} can be understood without reference to the
DPG method.

A quick introduction to the DPG method is given in
Section~\ref{sec:dpg}, where known conditions needed for {\it a
  priori} and {\it a posteriori} error analysis are also presented.
One of the conditions is the existence of a Fortin
operator. Anticipating the needs of the Maxwell application, we then
present, in Section~\ref{sec:fortin}, a sequence of Fortin operators
for $H^1(K), \Hcurl K $ and $\Hdiv K $, all on a single tetrahedral
mesh element~$K$. They are constructed to satisfy certain moment
conditions required for analysis of DPG methods. They fit into a
commuting diagram that \jrev{helps us} prove the required norm estimates (see
Theorem~\ref{thm:Fortin}).

The \rev{time-harmonic Maxwell equations} within a cavity are considered
afterward in Section~\ref{sec:maxwell}. Focusing first on a simple DPG
method for Maxwell equation, called the primal DPG method, we provide
a complete analysis using the tools developed in the previous section.
To understand one of the novelties here, recall that the wellposedness
of the Maxwell equations is guaranteed as soon as the excitation
frequency {of the harmonic wave} is different from a cavity
resonance.  However, this wellposedness is not directly inherited by
most standard discretizations, which are often known to be stable
solely in an asymptotic regime~\cite{Monk03b}. The discrete spaces
used must be sufficiently fine before one can even guarantee
solvability of the discrete system, not to mention error guarantees.
Furthermore, the analysis of the standard finite element method does
not clarify how fine the mesh needs to be to ensure that the stable
regime is reached.  In contrast, the DPG schemes, having inherited
their stability from the exact equations, are stable no matter how
coarse the mesh is.  This advantage is striking when attempting robust
adaptive meshing strategies.

Another focus of Section~\ref{sec:maxwell} is the understanding of a
proliferation of formulations for the Maxwell boundary value problem.
One may decide to treat individual equations of the Maxwell system
differently, e.g., one equation may be imposed strongly, while another
may be imposed weakly via integration by parts. Mixed methods make a
particular choice, while primal methods make a different choice. We
will show (see Theorem~\ref{thm:MaxwellCycles}) that the stability of one formulation
implies the stability of five others. The proof 
is an interesting application of
the closed range theorem.  However, when the DPG methodology is
applied to discretize these formulations, the numerical results
reported in Section~\ref{sec:numer}, show that the various methods do
exhibit differences. This is because the functional settings are
different for different formulations, i.e., convergence to the
solution occurs in different norms.  Section~\ref{sec:numer} also
provides results from numerical investigations on issues where the
theory is currently silent.

\section{Breaking Sobolev spaces}  \label{sec:break}

In this section, we discuss precisely what we mean by breaking Sobolev
spaces using a mesh. We will define {\em broken spaces} and {\em
  interface spaces} and prove a duality result that clarifies the
interplay between these spaces.  We work with infinite-dimensional
\jrev{(but mesh-dependent)} spaces on an open bounded domain
$\om \subset \RRR^3$ with Lipschitz boundary.  The mesh, denoted by
$\om_h$, is a disjoint partitioning of $\om$ into open elements $K$
such that the union of their closures is the closure of $\om.$ The
collection of element boundaries $\d K$ for all $K \in \oh$, is
denoted by $\d \om_h$. We assume that each element boundary $\d K$ is
Lipschitz. The shape of the elements is otherwise arbitrary for now.

We focus on the most commonly occurring first order Sobolev spaces of
real or \jrev{complex-valued} functions, namely $H^1(\om)$, $\Hdiv \om$, and
$\Hcurl \om$.  Their {\em broken} versions are defined, respectively,
by
\begin{align*}
H^1(\om_h) 
& = \{ u \in L^2(\om) \, : \, u\vert_K \in H^1(K),\, K \in \oh \}
&& = \prod_{K \in \oh} H^1(K), 
\\
\Hcurl\oh & = \{ E \in \jrev{(L^2(\om))^3} \, : \, E\vert_K \in H(\text{curl},K),\, K \in \oh \}
&& \ds = \prod_{K \in \oh} \Hcurl K,
\\
\Hdiv\oh
& = \{ \sigma \in \jrev{(L^2(\om))^3} \, : \, \sigma\vert_K \in H(\text{div},K),\, K \in \oh \}
&& \ds = \prod_{K \in \oh} H(\text{div},K).
\end{align*}
As these broken spaces contain functions with no continuity
requirements at element interfaces, their discretization is easier
than \jrev{that of } globally conforming spaces.

To recover the original Sobolev spaces from these broken spaces, we
need traces and interface variables. First, let us consider these
traces on each element $K$ in $\oh$.
\begin{align*}
  \trgK
  u 
  & = u\vert_{\ptl K} & u \in H^1(K),
  \\
  \trpK E
    & = (n_K \times E) \times n_K|_{\d K}   & E \in \Hcurl K,
  \\
  \trcK E
    & = n_K \times E\vert_{\ptl K}   & E \in \Hcurl K,
  \\
  \trdK \sigma
    & = \sigma\vert_{\ptl K} \cdot n_K & \sigma \in \Hdiv K.
\end{align*}
Here and throughout $n_K$ denotes the unit outward normal on $\d K$
and is often simply written as $n$.  Both $n_K$ and these traces are
well defined almost everywhere on $\d K$, thanks to our assumption
that $\d K$ is Lipschitz.  The operators $\trg,\trc,$ {$\trp$,} and $\trd$
perform the above trace operation element by element on each of the
broken spaces we defined previously, thus giving rise to linear maps 
\begin{gather*}
  \trg  : H^1(\oh) \to \prod_{K \in \oh} H^{1/2}(\d K), 
  \qquad
    \trc  : \Hcurl \oh \to  \prod_{K\in\oh} \Hhcurl{\d K},
  \\
  \trp  : \Hcurl \oh \to  \prod_{K\in\oh} \Hhdiv{\d K},
  \qquad
  \trd  : \Hdiv \oh \to \prod_{K \in \oh} H^{-1/2}(\d K).
\end{gather*}
It is well known that these maps are continuous and surjective 
\jrev{(for the 
standard definitions of the above codomain Sobolev spaces, see 
e.g.,~\cite{Monk03b}).}
 An
element of $\prod_{K \in \oh} H^{-1/2}(\d K)$ is expressed using
notations like $n \cdot \hat\sigma$ or $\hat{\sigma}_n$ (even when
$\hat \sigma$ itself has not been assigned any separate meaning) that
are evocative of their dependence on the interface normals.  Similarly,
the elements of the other trace map codomains are expressed 
using notations like
\[
n \times \hat E \equiv \hat{E}_{\dv} \in 
\prod_{K\in\oh} \Hhdiv{\d K},
\qquad
(n \times \hat E) \times n \equiv \hat{E}_{\tp}
\in \prod_{K\in\oh} \Hhcurl{\d K}.
\]

Next, we need spaces of {\em interface} functions.  We use the above
trace operators to define them, after cautiously noting two issues
that can arise on an interface piece $f = \d K^+ \cap \d K^-$ shared
by two mesh elements $K^\pm$ in $\oh$.  First, functions in the range
of $\trg$ when restricted to $F$ is generally multivalued, and we
would like our interface functions to be single valued in some
sense. Second, the range of the remaining trace operators \jrev{consists} of
functionals whose restrictions to $f$ are in general undefined. The
following definitions circumvent these issues.
  \begin{align*}
H^{1/2}({\d\oh}) 
& = \trg H^1(\om),
&
\Hhdiv\oh
& = \trc \Hcurl\om,
\\
\Hhcurl{\d\oh}
& = \trp \Hcurl\om,
&
H^{-1/2}({\d\oh}) 
& = \trd \Hdiv\om.
  \end{align*}
\jrev{
If  $\oh$ consists of a single element,  then 
 $H^{1/2}(\d\oh)$ equals  $H^{1/2}(\d\om)$, but in 
general 
$\trg H^1(\om) \subsetneq \trg H^1(\oh)$
(and similar remarks apply for the 
 other spaces). 
 We} norm each of the above
interface spaces by these quotient norms:
\begin{subequations} \label{eq:norms}
\begin{align}
  \label{eq:normsg}
  \| \hat u \|_{H^{1/2}({\d\oh})} 
  & = \inf_{u \in  H^1(\om) \cap\trgi \{ \hat u\}  } \| u \|_{H^1(\om)},
  \\
  \| \Et \|_{\Hhdiv {\d\oh}}
  & = \inf_{E \in \Hcurl\om \cap \trci \{ \Ed\}} \| E \|_{\Hcurl\om},
  \\
  \| \Ed \|_{\Hhcurl {\d\oh}}
  & = \inf_{E \in \Hcurl\om \cap \trpi \{ \Et\}} \| E \|_{\Hcurl\om},
  \\    \label{eq:normsd}
  \|\hat \sigma_n\|_{H^{-1/2}(\d\oh)}
  & = \inf_{\sigma \in \Hdiv\om \cap \trdi\{ \hat\sigma_n\} } \| \sigma \|_{\Hdiv\om}.
\end{align}  
\end{subequations} 
These are indeed quotient norms because the infimums are over cosets
generated by kernels of the trace maps. E.g., if $w$ is any function
in $H^1(\om)$ such that $\trg w = \hat u$, then the set where the
minimization is carried out in~\eqref{eq:normsg}, namely
$H^1(\om) \cap \trgi \{\hat u\},$ equals the coset
$w + \prod_{K\in\oh} \Ho^1(K)$, \jrev{where $\Ho^1(K)=\ker(\trgK)$.} 
 Note that every element of this coset
is an extension of $\hat u$.  For this reason, such norms are also
known as the ``minimum energy extension'' norms.  For an alternate way
to characterize the interface spaces, 
see~\cite{RobertsBui-ThanhDemkowicz14}.

\begin{remark}
  The quotient norm in~\eqref{eq:normsd} appeared in the literature as
  early as~\cite{RaviaThoma77a}. The word ``hybrid'' that appears in
  their title was used to refer to situations where, to
  quote~\cite{RaviaThoma77a}, ``the constraint of interelement
  continuity has been removed at the expense of introducing a Lagrange
  multiplier.''  The quote also summarizes the discussion of this
  section well.  The above definitions of our four interface spaces are
  thus generalizations of a definition in~\cite{RaviaThoma77a} and
  each can be interpreted as an appropriate space of Lagrange
  multipliers.
\end{remark}

We now show by elementary arguments  that the quotient norms on the 
two pairs of trace spaces  
\[
\{ H^{1/2}(\ptl K), H^{-1/2}(\ptl K)\} \quad\text{ and }\quad
\{H^{-1/2}(\text{curl},\ptl K),H^{-1/2}(\text{div},\ptl K)\},
\]
are dual to each other. The duality pairing in any Hilbert space $X$,
namely the action of a linear or conjugate linear (antilinear)
functional $x' \in X'$ on $x \in X$ is denoted by
$\ip{ x', x}_{X'\times X}$ and we omit the subscript in this notation
when no confusion can arise.  We also adopt the convention that when
taking supremum over vector spaces (such as in the next result) the
zero element is omitted tacitly.

\begin{lemma} \label{lem:duality} %
  The following identities hold for any $ n \cdot \hat \sigma $ in
  $ H^{-1/2}(\d K)$ and any $\hat u$ in $ H^{1/2}(\d K).$
\begin{subequations}\label{eq:sup}
  \begin{align}
    \label{eq:sup1a}
    \| \hat{\sigma}_n \|_{H^{-1/2}(\ptl K)} 
    & 
     = \!\sup_{u \in H^1(K)} \!
     \frac{ \vert \langle \hat{\sigma}_n , u \rangle \vert }
      {\hspace{0.5cm}\Vert  u \Vert_{H^1(K)}\hspace{-0.5cm}}  \hspace{0.5cm}
     \quad = 
     \sup_{\hat u \in H^{1/2}(\d K)} \frac{ \vert \langle \hat{\sigma}_n , 
      \hat u \rangle \vert 
      }{\hspace{0.5cm}\Vert \hat u \Vert_{H^{1/2}(\ptl K)}\hspace{-0.5cm}}\quad,
   \\
    \label{eq:sup1b}
    \Vert \hat u \Vert_{H^{1/2}(\ptl K)} 
    & = \!\sup_{\sigma \in \Hdiv K} \!\!\!
    \frac{ \vert \langle n \cdot \sigma,\hat u \rangle \vert }
      {\hspace{0.5cm}\Vert \sigma \Vert_{\Hdiv{ K}}\hspace{-0.5cm}}     
    \quad = \sup_{ \hat{\sigma}_n  \in H^{-1/2}(\ptl K)} \!\!\!
    \frac{ \vert \langle \hat{\sigma}_n,\hat u \rangle \vert
      }
      {\hspace{0.5cm}\Vert \hat{\sigma}_n \Vert_{H^{-1/2}(\ptl K)}\hspace{-0.5cm}}     
      \hspace{0.5cm}.
\end{align}
The next two identities hold for any
$ \hat{E}_{\dv}$ in $ \Hhdiv {\d K} $ and any
$\hat{F}_{\tp}$ in $ \Hhcurl {\d K}$.
\begin{align}
    \label{eq:sup2a}
    \Vert \hat{E}_{\dv} \Vert_{\Hhdiv{\ptl K}}
    & = \sup_{F \in \Hcurl K} 
      \frac{\vert \langle \hat{E}_{\dv}, F \rangle \vert}
      {\hspace{0.5cm} \Vert F \Vert_{\Hcurl{ K}}\hspace{-0.5cm}}
      \\ \nonumber 
    &  = 
      \!\!\!\!\sup_{\hat{F}_{\tp}  \in \Hhcurl{\d K}} 
      \!\!\!\!
      \frac{\vert \langle \hat{E}_{\dv}, \hat{F}_{\tp} 
      \rangle \vert  
      }
      {\hspace{1cm} \Vert  \hat{F}_{\tp} \Vert_{\Hhcurl{\ptl K}}\hspace{-1cm}}\;, 
    \\ \label{eq:sup2b}
    \Vert \hat{F}_{\tp} \Vert_{\Hhcurl{\d K}}
    & 
      = \sup_{ E \in \Hcurl K} \!\!
      \frac{\vert \langle n \times E, \hat{F}_{\tp}  \rangle \vert }
      {\hspace{0.5cm} \| E \|_{\Hcurl{ K}} \hspace{-0.5cm}}
    \\ \nonumber 
    &  =  \!\!\sup_{ \hat{E}_{\dv} \in \Hhdiv{\d K}}\!\!\!
    \frac{\vert \langle \hat{E}_{\dv},  \hat{F}_{\tp} \rangle \vert  
      }
      {\hspace{1cm} \| \hat{E}_{\dv} \|_{\Hhdiv{\ptl K}} \hspace{-1cm}}\hspace{1cm}.
  \end{align}
\end{subequations}
\end{lemma}
\begin{proof}
  The first identity is proved using an equivalence between a
  Dirichlet and a Neumann problem. The Dirichlet problem is the
  problem of finding $\sigma \in \Hdiv K$, given
  $\hat{\sigma}_n \in H^{-1/2}(\d K)$, such that
  \begin{equation}
    \label{eq:3}
\left\{
\begin{array}{ll}
n \cdot \sigma  = n \cdot \hat\sigma, & \quad\text{on } \ptl K, \\[8pt]
-\grad (\text{div } \sigma ) + \sigma = 0, & \quad\text{in } K.
\end{array}
\right.    
\end{equation}
The Neumann problem
finds $w \in H^1(K)$  satisfying 
\begin{equation}
  \label{eq:2}
\left\{
\begin{array}{lll}
\ds \frac{\ptl w}{\ptl n} = \hat{\sigma}_n,   &\quad \text{on } \ptl K, \\[8pt]
- \text{div} (\grad w) + w = 0, &\quad\text{in } K.
\end{array}
\right.  
\end{equation}
It is immediate that problems~\eqref{eq:2} and~\eqref{eq:3} are
equivalent in the sense that {$w$} solves~\eqref{eq:2} if and only if
$\sigma = \grad w$ solves~\eqref{eq:3} and moreover
$\| w \|_{H^1(K)} = \| \sigma \|_{\Hdiv K}$. It is also obvious from
the calculus of variations that among all $\Hdiv K$-extensions of
$\hat{\sigma}_n$, the solution of~\eqref{eq:3} has the minimal
$\Hdiv K$ norm (i.e., $\sigma$ is the ``minimum energy extension''
referred to earlier), so
\begin{align*}
  \| \hat{\sigma}_n \|_{H^{-1/2}(\d K)}
  & = \| \sigma \|_{\Hdiv K} = \| w \|_{H^1(K)}
    = \sup_{v \in H^1(K) } 
    \frac
    {| (\grad w, \grad v)_K + (w,v)_K |}
    { \| v \|_{H^1(K)} }
  \\
  & = \sup_{v \in H^1(K) } 
    \frac
    {| \ip{ \hat{\sigma}_n, v} | }
    { \| v \|_{H^1(K)}},
\end{align*}
where we used the variational form of~\eqref{eq:2} in the last
step. (Here and throughout, we use $(\cdot,\cdot)_K$ to denote the
inner product in $L^2(K)$ or its Cartesian products.) This proves the first
equality of~\eqref{eq:sup1a}.

Next, analogous to~\eqref{eq:3} and~\eqref{eq:2}, we set up another
pair of Dirichlet and Neumann problems. The first problem is to find
$u$ in $H^1(K)$, given any $\hat u \in H^{1/2}(\d K)$, such that 
\begin{equation}
  \label{eq:2a}
\left\{
\begin{array}{ll}
u = \hat u,        &\quad \text{on } \ptl K, \\
-\text{div} (\grad u) + u  = 0, & \quad\text{in } K.
\end{array}
\right.  
\end{equation}
The second is to find $\tau$  in $\Hdiv K$ such that 
\begin{equation}
  \label{eq:2b}
\left\{
\begin{array}{ll}
\text{div }\tau = \hat u,  & \quad\text{on } \ptl K, \\
-\grad (\text{div } \tau ) + \tau = 0, &\quad \text{in } K.
\end{array}
\right.
\end{equation}
The solution $u$ of~\eqref{eq:2a} has the minimal $H^1(K)$ norm among
all extensions of $\hat u$ into $H^1(K)$, i.e.,
$\| \hat u \|_{H^{1/2}(\d K) } = \| u \|_{H^1(K)}$. Thus
${\ip{ \hat{\sigma}_n, \hat u} } / \| \hat u \|_{H^{1/2}(\d K)} =
\ip{ \hat{\sigma}_n, u}/ \| u \|_{H^1(K)},$
so taking the supremum over all $\hat u$ in $H^{1/2}(\d K)$, we obtain 
\begin{align*}
  \sup_{\hat u \in H^{1/2}(\d K)} 
  \frac
  {| \ip{ \hat{\sigma}_n, \hat u} | }
  { \hspace{0.5cm}\| \hat u \|_{H^{1/2}(\d K)}\hspace{-0.5cm}}\hspace{0.5cm}
  \le 
  \sup_{u \in H^1(K)} 
  \frac
  {| \ip{ \hat{\sigma}_n,  u} | }
  { \hspace{0.5cm}\| u \|_{H^1(K)}\hspace{-0.5cm}}\hspace{0.2cm}.
\end{align*}
Since the reverse inequality is obvious from the definition of the
quotient norm in the denominator, we have established the second
identity of~\eqref{eq:sup1a}.  To \jrev{prove~\eqref{eq:sup1b}}, we begin, as
above, by observing that $\tau$ is the solution to the Neumann
problem~\eqref{eq:2b} if and only if $u = \dive \tau$ solves the
Dirichlet problem~\eqref{eq:2a}. Moreover,
$\| \tau \|_{\Hdiv K} = \| u \|_{H^1(K)}$. Hence 
\begin{gather*}
  \| \hat u \|_{H^{1/2}(\d K)}
   = \| u \|_{H^1(K)} = \| \tau \|_{\Hdiv K}
    = \sup_{\rho \in  \Hdiv K}
    \frac
    {| (\dive \tau,\dive \rho)_K + (\tau,\rho)_K | }
    { \| \rho \|_{\Hdiv K}}
  \\
   = \sup_{\rho \in \Hdiv K}
    \frac
    {| \ip{ n \cdot \rho, \hat u} | }
    { \| \rho \|_{\Hdiv K}},
\end{gather*}
where we have used the variational form of~\eqref{eq:2b} in the last
step. The proof of~\eqref{eq:sup1b} can now be completed as before.

We follow exactly the same reasoning for the $H(\text{curl})$ case,
summarized as follows: On one hand, the norm of an
  interface function equals the norm of a minimum energy extension,
  while on the other hand, it equals the norm of the inverse of a
  Riesz map applied to a functional generated by the interface
  function. The minimum energy extension that yields the interface
norm $\| \hat{E}_{\dv}\|_{\Hhdiv {\d K}}$ is now the solution of the
Dirichlet problem of finding $E \in \Hcurl K$ satisfying
\begin{equation}
  \label{eq:HcurlD}
\left\{
\begin{array}{ll}
n \times E = \hat{E}_{\dv}, & \quad\text{on } \ptl K, \\
\curl (\curl E) + E = 0, &\quad \text{in } K,
\end{array}
\right.
\end{equation}
while the inverse of the Riesz map applied to the functional
generated by $\hat{E}_{\dv}$ is obtained by solving the Neumann
problem
\begin{equation}
  \label{eq:HcurlN}
  \left\{
    \begin{array}{ll}
      n \times (\curl F) = \hat{E}_{\dv}, &\quad \text{on } \ptl K \\
      \curl ( \curl F) + F = 0, &\quad \text{in } K.
    \end{array}
  \right.
\end{equation}
Again, the two problems are equivalent in the sense that $F$
solves~\eqref{eq:HcurlN} if and only if $E = \curl F$
solves~\eqref{eq:HcurlD}. Moreover,
$\| E\|_{\Hcurl K} = \| F \|_{\Hcurl K}$. Hence
\begin{align*}
  \| \hat{E}_{\dv} \|_{\Hhdiv {\d K}}
  & = \| E \|_{\Hcurl K}
    = \| F \|_{\Hcurl K}
    \\
  & = \sup_{G \in \Hcurl K} 
    \frac
    {| (\curl F, \curl G)_K + (F, G)_K |}
    { \| G \|_{\Hcurl K} }
    \\
  & = \sup_{G \in \Hcurl K}
    \frac{| \ip{ \hat{E}_{\dv}, G}|} 
    { \hspace{0.5cm} \| G \|_{\Hcurl K}\hspace{-0.5cm} }\hspace{0.5cm}.
\end{align*}
The proof of~\eqref{eq:sup2a} follows from this. The  proof 
of~\eqref{eq:sup2b} is similar and is left to the reader.
\end{proof}

Let us return to the product spaces like $H^1(\oh), \Hcurl \oh,$ and
$\Hdiv \oh$. Any Hilbert space $V$ that is the Cartesian product of
various Hilbert spaces $V(K)$ is normed in the standard fashion, 
\[
V = \prod_{K \in \oh} V(K), \qquad \| v \|_V^2 = \sum_{K \in \oh} \| v_K\|_{V(K)}^2,
\]
where $v_K$ denotes the $K$-component of any $v$ in $V$. The dual space $V'$ is the Cartesian product of component duals $V(K)'$. Writing an $\ell
\in V'$  as 
$
\ell (v) = \sum_{K \in \oh} \ell_K(v_K),
$
where
$\ell_K \in V(K)'$, 
it is elementary
to prove that
$
  \| \ell \|_{V'}^2 = \sum_{K \in \oh} \| \ell_K \|_{V(K)'}^2,
$
i.e., 
\begin{equation}
  \label{eq:4}
  \left( \sup_{v \in V} \frac{\vert \ell(v) \vert} 
  {\hspace{0.2cm}\Vert v \Vert_V \hspace{-0.2cm}}\hspace{0.2cm}\right)^2 
   = 
   \sum_{K \in \oh} 
   \left( \sup_{v_K \in V(K)} \frac{\vert \ell_K(v_K) \vert}
     {\hspace{0.5cm}\Vert v_K \Vert_{V(K)}\hspace{-0.5cm}}\hspace{0.5cm} \right)^2. 
\end{equation}
Some of our interface spaces have such functionals, e.g., 
the function $\hat{\sigma}_n$ in $H^{-1/2}(\d\oh)$ gives rise to 
$\ell(v) = \ip{ \hat{\sigma}_n, v}_h$ where
\[
 \ip{ \hat{\sigma}_n, v}_h = 
\sum_{K \in \oh} \ip{ \hat{\sigma}_n, v}_{H^{-1/2}(\d K)\times 
  H^{1/2}(\d K)},
\]
is a functional acting on $ v \in H^1(\oh)$ which is the sum of
component functionals
$\ell_K(v) =\ip{ \hat{\sigma}_n, v}_{H^{-1/2}(\d K)\times
  H^{1/2}(\d K)}$
acting on $v_K = v|_K$ over every $K \in \oh$. Other functionals like
$\ip{ \hat{E}_{\dv}, F}_h$ are defined similarly. We are now ready
to state a few basic relationships between the interface and broken spaces.
\jrev{As usual, we define $\Ho^1(\om)=\{ v \in H^1(\om): v|_{\d\om}=0\},$
$\Hdivo\om = \{ \tau \in \Hdiv\om: \tau\cdot n|_{\d\om} =0\}$, and 
$\Hcurlo\om=\{ F\in \Hcurl\om: n \times F|_{\d \om} = 0\}.$}
\begin{theorem}
\label{thm:duality} %
The following identities hold for any interface space function
$ \hat{\sigma}_n $ in $ H^{-1/2}(\d \oh),$ $\hat u$ in
$ H^{1/2}(\d \oh),$ $ \hat{E}_{\dv}$ in $ \Hhdiv {\d K},$ and
$\hat{F}_{\tp}$ in $ \Hhcurl {\d K}$.
\begin{subequations}
  \label{eq:duality}
\begin{align}
  \label{eq:duality1}
  \| \hat{\sigma}_n \|_{H^{-1/2}(\d \oh)}
  & 
    = \!\sup_{u \in H^1(\oh)} \!
    \frac{ \vert \langle \hat{\sigma}_n , u \rangle_h \vert }
    {\hspace{0.5cm}\Vert  u \Vert_{H^1(\oh)}\hspace{-0.5cm}}\;,\hspace{0.5cm}
  \\ \label{eq:duality2}
  \Vert \hat u \Vert_{H^{1/2}(\d\oh)} 
  & = \!\sup_{\sigma \in \Hdiv \oh} \!\!\!
    \frac{ \vert \langle n \cdot \sigma,\hat u \rangle_h \vert }
    {\hspace{0.5cm}\Vert \sigma \Vert_{\Hdiv{ \oh}}\hspace{-0.5cm}}\;,
  \\ \label{eq:duality3}
  \Vert \hat{E}_{\dv} \Vert_{\Hhdiv{\ptl \oh}}
    & = \sup_{F \in \Hcurl \oh} 
      \frac{\vert \langle \hat{E}_{\dv} , F \rangle_h \vert}
      {\hspace{0.5cm} \Vert F \Vert_{\Hcurl{ \oh}}\hspace{-0.5cm}}\;,
  \\ \label{eq:duality4}
    \Vert  \hat{F}_{\tp} \Vert_{\Hhcurl{\d \oh}}
    & 
      = \sup_{ E \in \Hcurl \oh} \!\!
      \frac{\vert \langle n \times E, \hat{F}_{\tp} \rangle_h \vert }
      {\hspace{0.5cm} \| E \|_{\Hcurl{ \oh}} \hspace{-0.5cm}}\;.
\end{align}
\end{subequations}
 For any broken space function
$v \in H^{1}(\om_h), \tau \in \Hdiv \oh,$ and
$ F \in \Hcurl \oh$, 
\begin{subequations} \label{eq:conformequiv}
  \begin{align}
    \label{eq:conformequiv1}
    v \in \Ho^1(\om) 
    &
    \iff \langle \hat{\sigma}_n,v \rangle_h = 0
      && \forall \, \hat{\sigma}_n\in H^{-1/2}({\d\oh}),
    \\     \label{eq:conformequiv2}
    \tau \in \Hdivo\om
    & \iff \langle \tau \cdot n, \hat u  \rangle_h = 0
      && \forall \, \hat u  \in H^{1/2}({\d\oh}),
    \\     \label{eq:conformequiv3}
    F \in\Hcurlo\om 
    & \iff \langle \hat{E}_{\dv} ,F \rangle_h = 0
      && \forall \, \hat{E}_{\dv} \in \Hhdiv\oh.
  \end{align}
\end{subequations}
\end{theorem}
\begin{proof}
  The identities immediately follow from Lemma~\ref{lem:duality}
  and~\eqref{eq:4}. The proofs of the three equivalences
  in~\eqref{eq:conformequiv} are similar, so we will only detail the
  last one.  If $F$ is in $\Hcurlo\om$, then choosing any
  $E \in \Hcurl\om$ such that $\trc E = \hat{E}_{\dv}$ and 
  integrating by parts over entire $\om$, 
  \begin{align*}
    (\curl E, F)_\om + (E, \curl F)_\om = 0,
  \end{align*}
  because of the boundary conditions on $F$ on $\d\om$. Now, if the
  left hand side is integrated by parts again, this time element by
  element, then we find that $\ip{ \hat{E}_{\dv}, F}_h=0.$

  Conversely, given that $\ip{ \hat{E}_{\dv}, F}_h=0$ for any $F$ in
  $\Hcurl \oh$, consider $\curl F \in (\D(\om)^3)'$. As a
  distribution, $\curl F$ acts on $\phi \in \D(\om)^3$, and satisfies
  \[
  (\curl F)(\phi) = (F, \curl \phi)_\om = 
  ( \curl F, \phi)_h - \ip{ n \times \phi, F}_h = ( \curl F, \phi)_h,
  \]
  where we have integrated by parts element by element and denoted
  \[
  (\cdot, \cdot)_h = \sum_{K \in \oh} (\cdot, \cdot)_K.
  \]
  This notation also serves to emphasize that the term $\curl F$
  appearing on the right-hand side above is a derivative taken
  piecewise, element by element. Clearly $\curl F|_K$ is in $L^2(K)^3$
  for all $K \in\oh$ since $F \in \Hcurl \oh$, so the distribution
  $\curl F$ is in $L^2(\om)$. Having established that
  $F \in \Hcurl\om$, we may now integrate by parts to get
  $\ip{ n \times E, F}_{\Hhdiv{\d\om} \times \Hhcurl{\d\om}} = (\curl
  E, F)_\om + (E, \curl F)_\om = \ip{ n \times E, F}_h=0$
  for all $E \in \Hcurl\om$. This shows that the trace
  $(n \times F)\times n|_{\d \om}=0$, i.e, $F \in \Hcurlo\om$.
\end{proof}

\begin{remark}
  While $\prod_{K \in \oh}H^{-1/2}(\d K)$ and
  $\prod_{K\in\oh} H^{1/2}(\d K)$ are dual to each other, our
  interface spaces $H^{-1/2}(\d \oh)$ and $H^{1/2}(\d \oh)$ are {\em
    not} dual to each other in general.
\end{remark}

\begin{remark}
  Equivalences analogous to~\eqref{eq:conformequiv} hold with
  interface subspaces  
  \begin{align*}
    \Ho^{1/2}({\d\oh}) 
    & = \trg \Ho^1(\om),
    &
      \Hhdivo\oh
    & = \trp \Hcurlo\om,
    \\
    \Hhcurlo{\d\oh}
    & = \trc \Hcurlo\om,
    &
      \Ho^{-1/2}({\d\oh}) 
    & = \trd \Hdivo\om.
  \end{align*}
  By a minor modification of the arguments in the proof in
  Theorem~\ref{thm:duality}, we can prove that for any
  $v \in H^{1}(\om_h), \tau \in \Hdiv \oh,$ and $ F \in \Hcurl \oh$,
\begin{subequations} \label{eq:conformequivo}
  \begin{align}
    \label{eq:conformequiv4}
    v \in H^1(\om) 
    &
    \iff \langle \hat{\sigma}_n,v \rangle_h = 0
      && \forall \, \hat{\sigma}_n\in \Ho^{-1/2}({\d\oh}),
    \\     \label{eq:conformequiv5}
    \tau \in \Hdiv\om
    & \iff \langle \tau \cdot n, \hat u  \rangle_h = 0
      && \forall \, \hat u  \in \Ho^{1/2}({\d\oh}),
    \\     \label{eq:conformequiv6}
    F \in\Hcurl\om 
    & \iff \langle \hat{E}_{\dv} ,F \rangle_h = 0
      && \forall \, \hat{E}_{\dv} \in \Hhdivo\oh.
  \end{align}
\end{subequations}
\end{remark}

\section{Breaking variational forms}  \label{sec:break-forms}

The goal in this section is to investigate in what sense a variational
formulation can be reformulated using broken spaces without losing
stability. We will describe the main result in an abstract
setting first and close  the section with simple examples that use the
results of the previous section.

Let $X_0$ and $Y$ denote two Hilbert spaces and let $Y_0$ be a closed
subspace of $Y$.  For definiteness, we assume that all our spaces in
this section are over $\CCC$ (but our results hold also for spaces
over $\RRR$). In the examples we have in mind, $Y$ will be a broken
space, while $Y_0$ will be its unbroken analogue (but no such
assumption is needed to understand the upcoming results
abstractly). The abstract setting involves a continuous sesquilinear
form $b_0 : X_0 \times Y \to \CCC$ satisfying the following
assumption.

\begin{assumption}
  \label{asm:A0}
  There is a positive constant $c_0$ such that 
  \[
  c_0 \| x \|_{X_0} \le \sup_{y \in Y_0} \frac{ | b_0(x,y) |}{\| y\|_Y}
  \qquad \forall  x\in X_0.
  \]
\end{assumption}

It is a well-known result of \jrev{\Babuska\ and}
Ne\v{c}as~\cite{Babus70a,Necas67} that Assumption~\ref{asm:A0} together
with triviality of
\begin{equation}
  \label{eq:1}
  Z_0 = \{ y \in Y_0: \; b_0(x,y)=0 \text{ for all } x \in X_0 \}
\end{equation}
guarantees wellposedness of the following variational problem: Given
$\ell \in Y_0'$ (the space of conjugate linear functionals on $Y_0$),
find $x \in X_0$ satisfying
\begin{equation}
  \label{eq:5}
  b_0(x,y) = \ell(y) \quad \forall y \in Y_0.  
\end{equation}
When $Z_0$ is non-trivial, we can still obtain existence of a solution
$x$ provided the load functional $\ell$ satisfies the compatibility
condition $\ell(z)=0$ for all $z \in Z_0$. In~\eqref{eq:5},
the {\em trial space} $X_0$ need
not be the same as the {\em test space}~$Y_0$.

To describe a ``broken'' version of~\eqref{eq:5}, we need another 
Hilbert space $\hat X$, together with a continuous sesquilinear form
$\hat b : \hat X \times Y \to \CCC$. In applications $Y$ and $\hat X$
will usually be set to a broken Sobolev space and an interface space,
respectively.  Define 
\[
  b( \, (x,\hat x), \, y) = b_0(x,y) + \hat b( \hat x, y).
\]
Clearly $b : X \times Y \to \CCC$ is continuous, where 
\begin{equation}
  \label{eq:X}
X = X_0 \times \hat X  
\end{equation}
is a Hilbert space under the Cartesian product norm.  Now consider the
following new broken variational formulation: Given $\ell \in Y'$,
find $x \in X_0$ and $\hat x \in \hat X$ satisfying
\begin{equation}
  \label{eq:8}
  b( \, (x,\hat x), \, y) = \ell(y) \quad \forall y \in Y.
\end{equation}
The close relationship between problems~\eqref{eq:8} and~\eqref{eq:5} is 
readily revealed under the following assumption.

\begin{assumption} \label{asm:hybrid}
  The spaces $Y_0, Y,$ and $\hat X$ satisfy 
  \begin{equation}
    \label{eq:Yo}
    Y_0 = \{ y \in Y : \; \hat b( \hat x, y)=0 \text{ for all } \hat x \in \hat X\}
  \end{equation}
  and there is a positive constant $\hat c$ such that 
  \begin{equation}
    \label{eq:infsupinterface}
  \hat c \,\| \hat x \|_{\hat X} 
  \le \sup_{y \in Y} \frac{ | \hat b( \hat x, y) |}{ \| y\|_Y}
  \quad\forall \hat x \in \hat X.    
  \end{equation}
\end{assumption}

Under this assumption, we present a {simple result} which shows that
the broken form~\eqref{eq:8} inherits stability from the original
unbroken form~\eqref{eq:5}. {A very similar such abstract result was
  formulated and proved in~\cite[Appendix~A]{GargPrudhZee14} and used
  for other applications. Our proof is simple, unsurprising, and uses
  the same type of arguments from the early days of mixed
  methods~\cite[p.~40]{BrezzForti91}: stability of a larger system can
  be obtained in a triangular fashion by first restricting to a
  smaller subspace and obtaining stability there, followed by a
  backsubstitution-like step.  \jrev{Below, $\| b_0\|$ denotes the smallest
  number $C$ for which the inequality
  $|b_0(x,y)| \le C \| x \|_{X_0} \|y\|_{Y}$ holds for all $x\in X_0$
  and all $y \in Y$.} }

\begin{theorem}   \label{thm:hybrid}
  Assumptions~\ref{asm:A0} and~\ref{asm:hybrid} imply 
  \[
  c_1 \| (x, \hat x)\|_X
  \le \sup_{ y \in Y} \frac{ | b( \, (x,\hat x), \, y) | }
  {\| y \|_Y },
  \]
  where $c_1$ is defined by 
  \[
  \frac{1}{c_1^2} = 
  \frac{1}{c_0^2} + 
   \frac{1}{\hat c^2}\left( \frac{\| b_0 \|}{c_0} + 1   \right)^2. 
  \] 
  Moreover, if $Z = \{ y \in Y: \; b(\,(x,\hat x),y)=0$ for all
  $x \in X_0$ and $\hat x$ in $\hat X \}$, then 
  \[
  Z = Z_0.
  \]
  Consequently, if $Z_0=\{0\}$, then~\eqref{eq:8} is uniquely
  solvable and moreover the solution component $x$ from~\eqref{eq:8}
  coincides with the solution of~\eqref{eq:5}.
\end{theorem}
\begin{proof}
  We need to bound $\| x \|_{X_0}$ and $\| \hat x \|_{\hat X}$. First,
  \begin{align*}
    c_0 \| x \|_{X_0}
    & \le \sup_{y_0 \in Y_0} \frac{ | b_0(x,y) |}{\| y\|_Y}
    && \text{ by Assumption~\ref{asm:A0},}
    \\
    & \le \sup_{y_0 \in Y_0} \frac{ | b_0(x,y) + \hat b(\hat x,y) |}{\| y\|_Y}
    && \text{ by Assumption~\ref{asm:hybrid}, \eqref{eq:Yo}}
    \\
    &\le 
      \sup_{y \in Y} \frac{ |  b( \, (x,\hat x), \, y) |}{\| y\|_Y}
    && \text{ as } Y_0 \subseteq Y.
  \end{align*}
  Next, to bound $\| \hat x \|_{\hat X}$, using \eqref{eq:infsupinterface} of 
  Assumption~\ref{asm:hybrid},  
  \begin{align*}
    \hat c\, \| \hat x \|_{\hat X} 
    & 
      \le \sup_{y \in Y} \frac{ | \hat b( \hat x, y) |}{ \| y\|_Y}
      =       
      \sup_{y \in Y} \frac{ |  b( \, (x,\hat x), \, y) -  b_0( x, y) |}{ \| y\|_Y}
      \\
    & \le 
      \,\| b_0\| \,\| x \|_{X_0}
      + \sup_{y \in Y} \frac{ |  b( \, (x,\hat x), \, y)| }{ \| y\|_Y}.
  \end{align*}
  \jrev{Using} the already proved bound for
  $\| x \|_{X_0}$ in the last inequality and combining,
  \begin{align*}
    \| (x,\hat x) \|_X^2 
    & = \| x\|_{X_0}^2 + \| \hat x \|_{\hat X}^2\\
    & \le \left( \frac{1}{c_0}  
      \sup_{y \in Y} \frac{ |  b( \, (x,\hat x), \, y)| }{ \| y\|_Y}\right)^2
      + 
      \left(
      \frac{1}{ {\hat c}}
      \left(
      \frac{\| b_0\|}{c_0 } + 1
      \right)
      \sup_{y \in Y} \frac{ |  b( \, (x,\hat x), \, y) |}{\| y\|_Y}
      \right)^2
  \end{align*}
  from which the inequality of the theorem follows.

  Finally, to prove that $Z=Z_0$, using~\eqref{eq:Yo},
  \begin{align*}
    y \in Z & \iff b_0(x, y)=0 \text{ for all } x \in X_0\text{ and } \hat b(\hat x,y) =0 \text{ for all } \hat x \in \hat X
    \\
    & 
      \iff  b_0(x, y)=0 \text{ for all } x \in X_0\text{ and } 
      y \in Y_0,
  \end{align*}
  which holds if and only if $y \in Z_0$. 
\end{proof}

\begin{remark}
  Note that in the proof of the inf-sup condition, we did not fully
  use~\eqref{eq:Yo}. We only needed
  $ Y_0 \subseteq \{ y \in Y : \; \hat b( \hat x, y)=0 \text{ for all
  } \hat x \in \hat X\}$.
  The reverse inclusion was needed to conclude that $Z=Z_0$.
\end{remark}

\begin{remark}
  It is natural to ask, in the same spirit as
  Theorem~\ref{thm:hybrid}, if the numerical solutions of DPG methods
  using discretizations of the broken formulations coincide with those
  of discretizations of the original unbroken formulation. A result
  addressing this question is given
  in~\cite[Theorem~2.6]{BoumaGopalHarb14}.
\end{remark}

In the remainder of this section, we illustrate how to apply this
theorem on some examples.

\begin{example}[Primal DPG formulation] \label{eg:primal} %
  Suppose $f \in L^2(\om)$ and $u$ satisfies
  \begin{subequations}
    \label{eq:Poisson}
    \begin{align}
      \label{eq:Poisson1}
      -\Delta u & = f &&\text{ in } \om, 
      \\
      \label{eq:Poisson2}
      u & =0  &&\text{ on } \d\om.
    \end{align}
  \end{subequations}
  The standard variational formulation for this problem, finds
  $u$ { in }  $\Ho^1(\om)$  such that 
  \begin{equation}
    \label{eq:7}
    (\grad u, \grad v)_\om = (f, v)_\om\qquad \forall v \in \Ho^1(\om).
  \end{equation}
  This form is obtained by multiplying~\eqref{eq:Poisson1} by
  $v\in \Ho^1(\om)$ and integrating by parts over the entire domain
  $\om$. If on the other hand, we multiply~\eqref{eq:Poisson1} by a
  $v \in H^1(\oh)$ and integrate by parts element by element,
  then we obtain another variational formulation proposed
  in~\cite{DemkoGopal13a}: Solve for $u$ { in }$\Ho^1(\om)$ as well as a
  separate unknown $\hat{\sigma}_n \in H^{-1/2}(\d\oh)$ (representing
  the fluxes \jrev{$-n \cdot \grad u$} along mesh interfaces) satisfying
  \begin{equation}
    \label{eq:6}
  (\grad u, \grad y)_h + \ip{ \hat{\sigma}_n, y}_h 
  = (f, y)_\om
  \qquad \forall y \in H^1(\oh).
  \end{equation}
  We can view this as the broken version of~\eqref{eq:7} by  setting
  \begin{alignat*}{4}
    &X_0   = \Ho^1(\om),        &\quad&   Y_0  = \Ho^1(\om),
    \\
    &\hat X  = H^{-1/2}(\d \oh), &\quad& Y  = H^1(\oh),
    \\
    &b_0(u,y) = (\grad u, \grad y)_h, &\qquad& \hat b( \hat{\sigma}_n, y) 
    = \ip{ \hat{\sigma}_n, y}_h.
  \end{alignat*}
  For these settings, the conditions required to apply Theorem~\ref{thm:hybrid} are
  verified as {follows.}
  \begin{align*}
    \text{Coercivity of $b_0(\cdot,\cdot)$ on $Y_0$}
    & \implies \text{Assumption~\ref{asm:A0} holds.}
    \\
    \text{Theorem~\ref{thm:duality}, \eqref{eq:duality1}} 
    & \implies 
      \text{\eqref{eq:infsupinterface} of Assumption~\ref{asm:hybrid} holds with
      $\hat c =1$.}  
      \\
    \text{Theorem~\ref{thm:duality}, \eqref{eq:conformequiv1}}
    & \implies 
      \text{\eqref{eq:Yo} of Assumption~\ref{asm:hybrid} holds.}
  \end{align*}
  Noting that $Z_0=\{0\}$, an application of Theorem~\ref{thm:hybrid}
  implies that problem~\eqref{eq:6} is wellposed. This wellposedness
  result also shows that~\eqref{eq:6} is uniquely solvable with a more
  general {right-hand} side $f$ in $H^1(\oh)'$.

  An alternate (and longer) proof of this wellposedness result
  can be found in~\cite{DemkoGopal13a}.  The classical work
  of~\cite{RaviaThoma77a} also uses the spaces $H^1(\oh)$ and
  $H^{-1/2}(\d \oh)$, but proceeds to develop a Bubnov-Galerkin hybrid
  formulation different from the Petrov-Galerkin
  formulation~\eqref{eq:6}.~\hfill\qeg
\end{example}

\begin{example}[Many formulations of an elliptic
  problem] \label{eg:diffus} %
  Considering a model problem involving diffusion, convection, and
  reaction terms, we now show how to analyze, all at once, its various
  variational formulations.  The diffusion coefficient
  $a=\alpha^{-1}: \om \to \RRR^{3 \times 3}$ is a symmetric matrix
  function which is uniformly bounded and positive definite on $\om$,
  the convection coefficient is $\beta \in L^\infty(\om)^3$ which
  satisfies $\dive (a\beta)=0,$ and reaction is incorporated through a
  non-negative $\gamma \in L^\infty (\om)$.  The classical form of the
  equations on $\om$ are $\sigma = a \grad u + a \beta u + f_1 $ and
  $-\dive \sigma + \gamma u = f_2$ (for some {given
    $f_1\in L^2(\om)^3$ and $f_2\in L^2(\om)$}) together with the
  boundary condition $u|_{\d \om}=0$. {This can be written in operator form
    using}
  \begin{equation}
    \label{eq:Adiffus}    
  A \begin{bmatrix}
      \sigma \\ u 
    \end{bmatrix}
    = 
    \begin{bmatrix}
      \alpha \sigma - \grad u - \beta u  
      \\
      \dive \sigma - \gamma u 
    \end{bmatrix}, 
    \qquad 
    A^*
    \begin{bmatrix}
      \sigma \\ u 
    \end{bmatrix}
    = 
    \begin{bmatrix}
      \alpha \sigma - \grad u
      \\
      \dive \sigma - \beta \cdot \sigma - \gamma u 
    \end{bmatrix}.
  \end{equation}
  We begin with the formulation closest to the classical form.

  \begin{figure}
    \centering
    \begin{center}
      \begin{tikzpicture}[
        prob/.style={
          minimum size=6mm,
          inner sep=0pt, 
          very thick,  draw=black!50,
          top color=white, bottom color=blue!50!black!20,
          font=\tiny
        },
        to/.style={->, >=latex', semithick, 
          double,double distance=2pt,    
        },
        every node/.style={align=center}
        ]

        %
        %

        \matrix[row sep=2mm,column sep=3mm] {
          & &\node[prob] (S) {
            \begin{tabular}{l}
              Assumption~\ref{asm:A0} holds for the \\ 
              {\bf{Strong form}} with $b_0 = b_0^S$
            \end{tabular}
          }; 
          & &\node[prob] (U) {
            \begin{tabular}{l}
              Assumption~\ref{asm:A0} holds for the\\
              {\bf{Ultraweak form}} with $b_0 = b_0^U$
            \end{tabular}
          }; & & & 
          \\
          \node[prob] (P) {
            \begin{tabular}{l}
              Assumption~\ref{asm:A0} holds for the\\ 
              {\bf{Primal form}} with $b_0 = b_0^P$
            \end{tabular}
          }; 
          \\ 
          & & 
          \node[prob] (D) {
            \begin{tabular}{l}
              Assumption~\ref{asm:A0} holds for the\\ 
              {\bf{Mixed form}} with $b_0 = b_0^M$
            \end{tabular}
          };
          & & 
          \node[prob] (M) {
            \begin{tabular}{l}
              Assumption~\ref{asm:A0} holds for the\\ 
              {\bf{Dual Mixed form}} with $b_0 = b_0^D$
            \end{tabular}
          };
          \\
        };

        \draw[to] (P) -- (S) ;
        \draw[to] (S) -- (U) ;
        \draw[to] (U) -- (D);
        \draw[to] (D)-- (P) ;

        \draw[to] (U)-- (M) ;
        \draw[to] (M)-- (S) ;
      \end{tikzpicture}
    \end{center}
    \caption{Chains of implications of inf-sup conditions}
    \label{fig:impl}
  \end{figure}
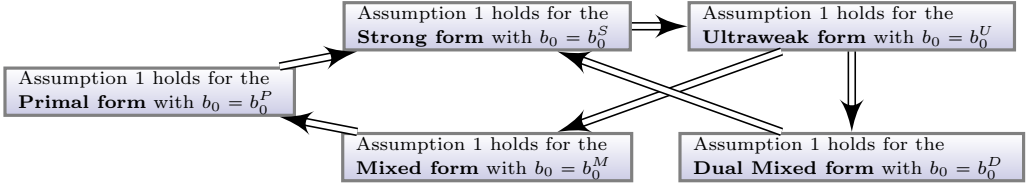

  \begin{description}
  \item[Strong form] 
    Let $x = (\sigma, u)$ be a group variable. Set spaces by 
    \begin{alignat*}{4}
      &X_0   = \Hdiv\om \times \Ho^1(\om),        &\quad& Y =Y_0=L^2(\om)^3 \times L^2(\om),
    \end{alignat*}
    and consider the problem of finding $x \in X_0,$ given $f \in Y$,  satisfying 
    $
    A   x = f.
    $
    We can trivially fit this into our variational
    framework~\eqref{eq:5} by setting
    $b_0$ to 
  \begin{gather*}
    \quad
    b_0^S( x, y) = (Ax, y)_\om.
  \end{gather*}
  \end{description}
  Unlike the remaining formulations below, there is no need to discuss
  a broken version of the above strong form as the test space already admits
  discontinuous functions.  The next formulation is often derived
  directly from a second order equation obtained by eliminating
  $\sigma$ from the strong form.
  \begin{description}
  \item [Primal form] %
    First, set spaces by 
    \begin{alignat*}{4}
    &X_0   = \Ho^1(\om),        &\quad&   Y_0  = \Ho^1(\om),
    \\
    &\hat X  = H^{-1/2}(\d \oh), &\quad& Y  = H^1(\oh).
    \end{alignat*}
    Then, with 
    $\hat b$ set to $\hat b^P( \hat{\sigma}_n, v)
    = \ip{ n\cdot\hat \sigma, \jrev{v}}_h$ and 
    $b_0$ set to 
    \begin{gather*}
      b_0^P(u,v) 
       = (a \grad u, \grad v)_h \jrev{+} ( a    \beta u, \grad v)_h + (\gamma u,v)_\om,
    \end{gather*}
    the standard primal formulation is~\eqref{eq:5} and its broken version is~\eqref{eq:8}.
  \end{description}
  Next, consider the formulation derived by multiplying each equation
  in the strong form by a test function and integrating {\em both}
  equations by parts, i.e., both equations are imposed weakly. It was
  previously studied in~\cite{CausiSacco05,DemkoGopal11a}, but we can
  now simplify its analysis considerably  using Theorem~\ref{thm:hybrid}.
  \begin{description}
  \item[Ultraweak form] %
    Set group variables $x = (\sigma, u)$, $y = (\tau, v)$, $\hat x =
    (\hat{\sigma}_n, \hat u)$, and 
    \begin{alignat*}{4}
    &X_0   = L^2(\om)^3 \times L^2(\om),        &\quad&   Y_0  =
    \Hdiv\om \times \Ho^1(\om),
    \\
    &\hat X  = \Ho^{1/2}(\d\oh) \times H^{-1/2}(\d \oh), &\quad& Y  =
    \Hdiv\oh \times H^1(\oh),
    \\
    & b_0^U( x,y) = (x, A^* y)_\oh,     & \quad & \hat b^U(\hat x, y) =
    \ip{\hat{\sigma}_n, v}_h \jrev{-} \ip{ \hat u, n \cdot \tau}_h.
    \end{alignat*}
    Formulations~\eqref{eq:5} and~\eqref{eq:8} with
    $b_0 = b_0^U$ and $ \hat b = \hat b^U$ are of the ultraweak type.
  \end{description}
  The fourth formulation, well-known as the mixed
  form~\cite{BrezzForti91}, is derived by weakly imposing (via
  integration by parts) the first equation of the strong form, but
  strongly imposing the second equation.
  \begin{description}
  \item[Dual Mixed form]  %
    Set the spaces by 
    \begin{alignat*}{4}
    &X_0   = \Hdiv \om \times L^2(\om),        &\qquad&   Y_0  = \Hdiv \om \times L^2(\om),
    \\
    &\hat X  = \Ho^{1/2}(\d\oh), &\qquad& Y  =  \Hdiv\oh \times L^2(\om).
    \end{alignat*}
    The well-known mixed formulation is then~\eqref{eq:5} with $b_0 =
    b_0^D$, 
    \begin{align*}
            b_0^D(\, (\sigma, u),\, (\tau, v)\,)
      & = (\alpha \sigma, \tau)_h + (u,\dive \tau)_h - (\beta        u,\tau)_h
        \\
      & + ( \dive \sigma, v)_h - (\gamma u,v)_h.
    \end{align*}
    Its broken version is~\eqref{eq:8} with $\hat b$ set to
    \jrev{$\hat b^D( \hat u, (\tau,v)) = -\ip{ \hat u, n \cdot \tau }_h$}.
  \end{description}
  Note that the well-known discrete hybrid mixed
  method~\cite{BrezzForti91,CockbGopal04} is also derived
  from~$b_0^D$. That method however works with a Bubnov-Galerkin
  formulation obtained by breaking both the trial and the test
  $\Hdiv\om$ components, while above we have broken only the test
  space.  The last formulation in this example reverses the roles by
  weakly imposing the second equation of the strong form and strongly
  imposing the first equation:
  \begin{description}
  \item[Mixed form]  %
    Set
    \begin{alignat*}{4}
    &X_0   = L^2(\om)^3\times \Ho^1(\om),        &\qquad&   Y_0  = L^2(\om)^3\times \Ho^1(\om),
    \\
    &\hat X  = H^{-1/2}(\d\oh), &\qquad& Y  =  L^2(\om)^3 \times H^1(\oh).
    \end{alignat*}
    The dual mixed formulation
    is~\eqref{eq:5} with $b_0 =
    b_0^M$, 
    \begin{align*}
      b_0^M(\, (\sigma, u),\, (\tau, v)\,)
      & = (\alpha \sigma, \tau)_\om - (\grad u, \tau)_\om - (\beta
        u,\tau)_\om 
        \\
      & - ( \sigma, \grad v)_\om -  (\gamma u,v)_\om.
    \end{align*}
    and its broken version is~\eqref{eq:8} with $\hat b$ set to $
    \hat b^M(  \hat{\sigma}_n, \jrev{(\tau,v)}) = \ip{ \hat{\sigma}_n, v}_h$.
  \end{description}
  The variational problem~\eqref{eq:5} with $b_0^M$ is sometimes
  called~\cite{Braes07} the {\em primal mixed} form to differentiate
  it with the {\em dual mixed} form given by $b_0^D$.  The broken
  formulation~\eqref{eq:8} with $b_0^M$ and
  $\hat b^M$ was called the {\em mild weak DPG formulation}
  in~\cite{BroerSteve14}. Their analysis can also be simplified now
  using Theorem~\ref{thm:hybrid}.

  In order to apply Theorem~\ref{thm:hybrid} to all these
  formulations, we need to verify Assumption~\ref{asm:A0}.  This can
  be done for all the formulations at once, because the six
  implications displayed in Figure~\ref{fig:impl} are proved
  in~\cite{Demko15} for the model problem of this example \jrev{(thus
    making the five statements in Figure~\ref{fig:impl} equivalent)}. We
    will not detail this proof here because we provide full proofs of
    similar implications for Maxwell equations in
    Section~\ref{sec:maxwell} (and this example is simpler than the
    Maxwell case). To apply these implications for the current
    example, we pick a formulation for which Assumption~\ref{asm:A0}
    is easy to prove: That the primal form is coercive
    $b_0^P(u,u) \ge C \| u \|_{H^1(\om)}^2$ follows immediately by
    integration by parts and the \Poincare\ inequality (under the
    simplifying assumptions we placed on the coefficients).  This
    verifies Assumption~\ref{asm:A0} for the primal form, which in
    turn verifies it for all the formulations by the above chain of
    equivalences. Assumption~\ref{asm:hybrid} can be immediately
    verified for all the formulations using
    either~\eqref{eq:conformequiv}
    or~\eqref{eq:conformequivo}. Together with the easily verified
    triviality of $Z_0$ in each case, we have proven the wellposedness
    of all the formulations above, including the broken
    ones.\hfill\qeg\end{example}

\section{The DPG method}  \label{sec:dpg}

In this section, we quickly introduce the DPG method, indicate why
the broken spaces are needed for practical reasons within the DPG
method, and recall known abstract conditions under which an error analysis
can be conducted.

Let $X$ and $Y$ be Hilbert spaces and let $b : X \times Y \to \CCC$ be
a continuous sesquilinear form.  In the applications we have in mind,
$X$ will always be of the form~\eqref{eq:X} (but we need not assume it
for the theory in this section). The variational problem is to find
{$x$ in $  X$}, given $\ell \in Y'$, satisfying
\begin{equation}
  \label{eq:11}
b(x,y) = \ell(y) \qquad \forall y \in Y.  
\end{equation}

The DPG method uses finite-dimensional subspaces $X_h \subset X$ and
$Y_h \subset Y$.  The test space used in the method is a subspace
$\Yhopt \subseteq Y_h$ of {\em approximately optimal test functions}
computed for any arbitrarily given trial space $X_h$. It is defined by
$ \Yhopt = T_h(X_h)$ where $T_h : X_h \to Y_h$ is given by
\begin{equation}
  \label{eq:14}
  (T_h z, y)_Y = b(z, y) \qquad \forall y \in Y_h.
\end{equation}
Here $(\cdot,\cdot)_Y$ is the inner product in $Y$, hence by Riesz
representation theorem on $Y_h$, the operator $T_h$ is well defined.
The discrete problem posed by the DPG discretization is to find
$x_h \in X_h$ satisfying
\begin{equation}
  \label{eq:pdpg}
b(x_h, y) = \ell(y) \qquad \forall y \in \Yhopt.
\end{equation}
For practical implementation purposes, it is important to note that
$T_h$ can be easily and inexpensively computed via~\eqref{eq:14} provided
the space $Y_h$ is a subspace of a broken space. Then~\eqref{eq:14}
becomes a series of small decoupled problems on each element.

For {\it a posteriori} error estimation, we use an estimator
$\tilde \eta$ that actually works for any $\tilde{x}_h$ in $X_h$,
computed as follows.  (Note that $\tilde{x}_h$ need not equal the
solution $x_h$ of~\eqref{eq:pdpg}.)  First we solve for
$\tilde{\veps}_h$ in $Y_h$ by
\begin{equation}
  \label{eq:23}
(\tilde{\veps}_h, y)_Y = \ell(y) - b(\tilde{x}_h, y), \qquad \forall y \in Y_h. 
\end{equation}
Again, this {amounts to} a local computation if $Y_h$ is a
subspace of a broken space. Then, set
\[
\tilde{\eta} = \| \tilde{\veps}_h \|_Y.
\] 
When $Y_h$ is a broken space, the element-wise norms of
$\tilde{\veps}_h$ serve as good error
estimators~\cite{DemkoGopalNiemi12}.  {The notations $\eta$ and
  $\veps_h$ (without tilde) refer to similarly computed quantities
  with $x_h$ in place of  $\tilde{x}_h$.}
An analysis of errors and error estimators of
the DPG method can be conducted using the following assumption
introduced in~\cite{GopalQiu14}. In accordance with the traditions in
the theory \jrev{of} mixed methods~\cite{BrezzForti91}, we will call the
operator $\varPi$ in the assumption a {\em Fortin operator}.

\begin{assumption} \label{asm:pi} There is a  continuous linear operator
  $\varPi: Y \to Y_h$ such that for all
  $ {w_h}\in {X_h}$ and all $ v \in Y$,
  \[
    b(  w_h , v - {\varPi}  v) =0.
  \]
\end{assumption}
\begin{theorem}  \label{thm:dpg}
  Suppose Assumption~\ref{asm:pi} holds. Assume also that there is a
  positive constant $c_1$ such that
  \begin{equation}
    \label{eq:infsupbroken}
    c_1 \| x \|_{X} \le \sup_{y \in Y} \frac{ | b(x,y) |}{\| y\|_Y}
    \qquad \forall  x\in X,
  \end{equation}
  and the set
  $Z = \{ y \in Y: \; b(x,y)=0 \text{ for all } x \in X \}$ equals
  $\{0\}$. Then
    the DPG method~\eqref{eq:pdpg} is uniquely solvable for $x_h$ and
    the {\it a priori} error estimate 
    \begin{align}
      \label{eq:Apriori}
      \| x- x_h \|_X 
      & \le \frac{\| b \| \| \varPi\|}{c_1}
      \inf_{z_h\in X_h } \| x - z_h \|_X            
      && \text{(quasi-optimality)}
    \end{align}
    holds, where $x$ is the unique exact solution of~\eqref{eq:11}.
    Moreover, {we have the following inequalities for any $\tilde{x}_h$
    in $X_h$ and its corresponding error estimator $\tilde{\eta}$,
    with the data-approximation error
    $ \osc(\ell) = \| \ell \circ(1-\varPi) \|_{Y'}.$}
  \begin{subequations}\label{eq:estimator}
    \begin{align}
      \label{reliability}
      c_1 \| x- \tilde x_h  \|_X
      & \le 
         \|\varPi\|\,  \tilde\eta + \osc(\ell),
      && \text{(reliability)}
      \\
      \label{efficiency} 
      \tilde\eta 
      &  \le \| b \| \, \| x- \tilde x_h \|_X, 
      && \text{(efficiency)}
      \\
      \label{EfficiencyofDataApproximation}
      \osc(\ell)
      & \le \|b\|\, \|1- \varPi\|\,\min_{z_h\in X_h} \|x- z_h\|_X.
    \end{align}
\end{subequations}
\end{theorem}
Here $\| \varPi\|$ and $\|b\|$ are any constants that satisfy
$ \| \varPi y \|_Y \le \| \varPi \| \| y\|_Y$ and
$| b(w,y) | \le \| b \| \| w \|_X \|y \|_Y$, respectively, for all
$w\in X$ and $y \in Y$.  To apply the theorem to specific examples of
DPG methods, we must verify~\eqref{eq:infsupbroken}. This will usually
be done by appealing to Theorem~\ref{thm:hybrid} and verifying
Assumptions~\ref{asm:A0} and~\ref{asm:hybrid}. The previous sections
provided tools for verifying Assumptions~\ref{asm:A0}
and~\ref{asm:hybrid}.  In the next section, we will provide some tools
to verify the remaining major condition in the theorem, namely
Assumption~\ref{asm:pi}. 

\begin{remark}
  A proof of Theorem~\ref{thm:dpg} is available in existing
  literature. The {\it a priori} error bound~\eqref{eq:Apriori} was
  proved in~\cite{GopalQiu14}. The inequalities
  of~\eqref{eq:estimator}, useful for {\it a posteriori} error
  estimation, were proved in~\cite{CarstDemkoGopal14}.  In particular,
  a reliability estimate slightly different from~\eqref{reliability}
  (with worse constants) was proved in~\cite{CarstDemkoGopal14}, but
  the same ideas yield~\eqref{reliability} easily (for example, cf.~\cite[proof of
  Lemma~3.6]{CarstGalliHellw14}).
\end{remark}

\begin{remark}
  The operator $T_h$ is an approximation to an idealized trial-to-test
  operator $T: X \to Y$ given by
\[
(Tx, y)_Y = b(x,y)
\qquad \forall y \in Y.
\]
If $B: X \to Y'$ is the operator defined by the form satisfying
$(Bx)(y) = b(x,y)$ for all $x \in X$ and $y \in Y$, then clearly
$T = R_Y^{-1}B$, where $R_Y : Y \to Y'$ is the Riesz map defined by
$(R_Y y)(v) = (y,v)_Y$. In some examples~\cite{DemkoGopal11a}, 
it is possible to
analytically compute $T$ and then one may substitute $\Yhopt$ with the
{\em exactly optimal test space}
$Y^\mathrm{opt} = T(X_h)$.
\end{remark}

\begin{remark}
  The above-mentioned trial-to-test operator $T=R_Y^{-1} B$ should not be
  confused with another trial-to-test operator $S = (B')^{-1} R_X$
  of~\cite{BarreMorto84} (also cf.~\cite{DemkoOden86a}):
  \[
  \begin{CD}
    X        @>B>>   Y'
    \\
    @V{R_X}VV   @AA{R_Y}A 
    \\
    X'        @<B'< <   Y
  \end{CD}
  \]
  Application of $S$ requires the inversion of the dual operator $B'$.
\end{remark}

\section{Fortin operators}  \label{sec:fortin}

The Fortin operator $\varPi$ appearing in Assumption~\ref{asm:pi} is
problem specific since it depends on the form~$b$ and the
spaces. However, there are a few Fortin operators that have proved
widely useful for analyzing DPG methods, including one for
$Y = H^1(\oh)$ and another for $Y = \Hdiv\oh$, both given
in~\cite{GopalQiu14}. In this section, we complete this collection by
adding another operator for $Y = \Hcurl \oh$ intimately connected to
the other two operators. Its utility will be clear in a subsequent
section.

Since the Fortin operators for DPG methods are to be defined on broken
Sobolev spaces, their construction can be done focusing solely on one
element. We will now assume that the mesh $\oh$ is a geometrically
conforming finite element mesh of tetrahedral elements.  Let $P_p(D)$
denote the set of polynomials of degree at most $p$ on a domain $D$
and let $ N_p(D) = P_{p-1}(D)^3 + {x} \times P_{p-1}(D)^3 $ denote the
\Nedelec~\cite{Nedel80} space.  For domains $D \subset \RRR^n$,
$n=2,3$, let $ R_p(D) = P_{p-1}(D)^n + {x} P_{p-1}(D)$ denote the
Raviart-Thomas~\cite{RaviaThoma77} space. We use
$\vpi_p$ to denote the $L^2$ orthogonal projection onto $P_p(K)$.
From now on, let us use $C$ to denote a generic constant independent
of $h_K = \diam K$. Its value at different occurrences may differ and
may possibly depend on the shape regularity of $K$ and the
polynomial degree~$p$.

\begin{theorem} \label{thm:Fortin} 
  On any tetrahedron $K$, there are operators
  \begin{align*}
    \pig & : H^1(K)    \to  P_{p+3}(K),   \\
    \pic & : \Hcurl K  \to   N_{p+3}(K), \\
    \pid & : \Hdiv K   \to  R_{p+3}(K), 
  \end{align*}
  such that the norm estimates 
  \begin{subequations} \label{eq:bdd}
  \begin{align}
    \label{eq:bdd-a}
    \|    \pig v \|_{H^1(K)} & \le C \| v \|_{H^1(K)}, 
    \\ \label{eq:bdd-b}
    \|    \pic F \|_{\Hcurl K} & \le C \| F \|_{\Hcurl K},
    \\ \label{eq:bdd-c}
    \|    \pid q \|_{\Hdiv K} & \le C \| q \|_{\Hdiv K}
  \end{align}    
  \end{subequations}
  hold, the diagram   
\[
\begin{CD}
0\! @> >> \!    H^1(K)/\RRR 
    @>\grad >> \Hcurl K
    @>\curl >> \Hdiv K
    @>\dive >>  L^2(K)
    @>>> 0 
\\
@. @VV{\pig}V	@ VV{\pic}V	@ VV{\pid}V	@ VV{\pil}V
\\
0\! @> >> \! P_{p+3}(K)/\RRR	
    @> \grad >> N_{p+3}(K)	
    @> \curl >> R_{p+3}(K)
    @> \dive >>P_{p+2}(K)
    @>>> 0
\end{CD}
\]
commutes, and these identities hold for any $v \in H^1(K)$,
$E \in \Hcurl K$, and $\tau \in \Hdiv K$:
\begin{subequations}
\begin{align}
  \label{eq:PiH1a}
  (q, \pig v - v)_{K}
  & 
    =0  
  && \forall\; q \in P_{p-1}(K),
  \\ \label{eq:PiH1b}
  \langle n \cdot \sigma, \pig v - v\rangle
  & = 0
  &&\forall\; n \cdot \sigma \in \trdK R_{p+1}( K),  
  \\ \label{eq:PiHcurla}
    (q, \pic E - E)_{K}
    & 
    =0  
    && \forall\; q \in P_{p}(K)^3,
    \\ \label{eq:PiHcurlb}
    \langle  (n \times F)\times n, n \times (\pic E - E)\rangle
    & = 0
    && \forall\; (n \times F)\times n \in \trpK P_{p+1}(K)^3,
  \\ \label{eq:PiHdiva}
    (q,\pid  \tau - \tau)_{K}
    & 
    =0  
    && \forall\; q\in P_{p+1}(K)^3,
  \\ \label{eq:PiHdivb}
    \langle n \cdot (\pid \tau - \tau), \mu\rangle
    & = 0
    && \forall\; \mu \in \trgK P_{p+2}(K).
\end{align}
\end{subequations}
\end{theorem}

\jrev{Note that the duality pairings above must be taken in the
  appropriate spaces, as in~\eqref{eq:sup}.}
To provide a constructive proof of Theorem~\ref{thm:Fortin}, we will
exhibit Fortin operators. We will use the exact sequence properties of
the finite element spaces appearing as codomains of the operators in
the theorem. We cannot use the canonical interpolation operators in
these finite element spaces because they do not
satisfy~\eqref{eq:bdd}. Hence we will restrict the codomains of our
operators to the following subspaces whose construction is motivated
by zeroing out the unbounded degrees of freedom.
\begin{align*}
  \Bg(K)
  & =\{v\in P_{p+3}(K)
   :   \; v \text{ vanishes on all edges and vertices of } K  
  \},
  \\
  \Bc(K)
  & =\bigg\{
    E\in N_{p+3}(K)
   :   \;
  t \cdot E =0
  \text{ on all edges of } K,
  \\
  & \hspace{3.5cm}   \int_{\d K}\!\! \phi \,n \cdot \curl E =0
    \quad\forall \phi \in \Pbod \text{ and }
  \\
  & \hspace{3.5cm} 
    \int_{\d K} ( (n \times E) \times n )\cdot r = 0 
    \quad \forall r \in R_1(\d K)
  \bigg\},
  \\
  \Bd(K)
  & =\bigg\{ \tau \in R_{p+3}(K)
   :   \; \int_{\d K} \phi\,
  n \cdot\tau = 0
  \quad
  \forall \phi \in \Pbd\bigg\}.
\end{align*}
Here $t$ denotes a tangent vector along the underlying edge,
$R_1(\d K) = \{ r : r|_f \in R_1(f)$ for all faces $f$ of $\d K\}$,
and $\Pbod$ and $\Pbd$ are defined as follows. To simplify notation,
let $P_p(\d K) = \trdK R_{p+1}(K)$ (the space of functions on $\d K$
that are polynomials of degree at most $p$ on each face of $\d K$) and
let $\Pc_p(\d K) = \trgK P_p(K)$.  Let $\Pbo_p(\d K)$ denote the
$L^2(\d K)$-orthogonal complement of $\Pc_p(\d K) +P_0(\d K)$ in
$P_p(\d K)$ and let $\Pb_p(\d K)$ denote the $L^2(\d K)$-orthogonal
complement of $\Pc_p(\d K)$ in $P_p(\d K)$.  The following result is
proved in~\cite[Lemma~3.2]{GopalQiu14}.

\begin{lemma}
  \label{lem:grad}
  For any $v \in H^1(K)$, there is a unique $\vog$ in $\Bg(K)$
  satisfying
  \begin{subequations} \label{eq:pi0}
  \begin{gather}
    \label{eq:pi0a}
    (q, \vog - v)_{K}
    =0  
    \quad\text{ for all } q \in P_{p-1}(K),
    \\  \label{eq:pi0b}
    \langle n \cdot \sigma, \vog - v\rangle
    = 0
    \quad\text{ for all } n \cdot \sigma \in \trdK R_{p+1}(K), 
    \text{ and }
    \\ \label{eq:pi0c}
    \| \vog \|_{L^2(K)} + h_K \| \grad \vog\|_{L^2(K)}
   \le C\left( \| v \|_{L^2(K)} + h_K \| \grad v\|_{L^2(K)}\right).
  \end{gather}
\end{subequations}
\end{lemma}

We define $\pig$ as a minor modification of the analogous operator
in~\cite{GopalQiu14}. Given any $v \in H^1(K)$, first compute its
mean value on the boundary
\[
m_{\d K}(v) = \frac{1}{| \d K|} \int_{\d K} v,
\]
then split $v = v_0 + m_{\d K}(v)$ where $v_0 = v - m_{\d K}( v)$ has
zero mean trace, and finally define
\begin{equation}
  \label{eq:pigradbdr}
  \pig v =  \pigo v_0 +  m_{\d K}(v).
\end{equation}

\begin{lemma} \label{lem:PiGrad} %
  $\pig v$ satisfies~\eqref{eq:PiH1a}--\eqref{eq:PiH1b}
  and~\eqref{eq:bdd-a} for any $v \in H^1(K)$
\end{lemma}
\begin{proof}
  Since $\pig v -v = \pigo v_0 - v_0$, equations~\eqref{eq:pi0a} and
  \eqref{eq:pi0b} immediately yield~\eqref{eq:PiH1a} and
  \eqref{eq:PiH1b}. To prove the norm estimate~\eqref{eq:bdd-a}, note
  that standard scaling arguments imply 
  \begin{gather}
    \label{eq:141}
    \| m_{\d K} (v)\|_{L^2(K)}^2
     \le \| v \|^2 _{L^2(\d K)}\frac{ |K|}{  | \d  K| }
     \le C ( \| v \|_{L^2(K)}^2 + h_K^2 \| \grad v \|_{L^2(K)}^2),
    \\ \label{eq:15}
    \| v - m_{\d K}(v) \|_{L^2(K)}
     \le C h_K \| \grad v \|_{L^2(K)}.
  \end{gather}
  for all $v$ in $H^1(K)$.
  Combining~\eqref{eq:141} and~\eqref{eq:pi0c}, we get 
  \begin{align*}
    \| \pig v \|_{L^2(K)}
    & \le
      C\left( \| v \|_{L^2(K)} + h_K \| \grad v\|_{L^2(K)}\right),
  \end{align*}
  while combining~\eqref{eq:15} and \eqref{eq:pi0c}, 
  \begin{align*}
    h_K & \| \grad (\pig v) \|_{L^2(K)}
          = h_K \| \grad \pigo (v - m_{\d K} (v))  \|_{L^2(K)}, 
    && \text{by~\eqref{eq:pigradbdr}}
    \\
    & \le  C\left( \| v - m_{\d K} (v)\|_{L^2(K)} 
      + h_K \| \grad (v - m_{\d K} (v))\|_{L^2(K)}\right)
    && \text{by~\eqref{eq:pi0c}}
      \\
    & \le C h_K \| \grad v \|_{L^2(K)}
    && \text{by~\eqref{eq:15}}.
  \end{align*}
  These estimates together prove~\eqref{eq:bdd-a}.
\end{proof}

The next lemma is proved in~\cite[Lemma~3.3]{GopalQiu14}. It defines
$\pid$ exactly as in~\cite{GopalQiu14}.

\begin{lemma} \label{lem:div} %
  Any $\sigma \in \Bd(K)$ satisfying
    \begin{subequations}
      \begin{align}
        \label{eq:Hdiva}
        (q, \sigma)_{K}
        & 
          =0  
        && \forall\; q\in P_{p+1}(K)^3,
        \\ \label{eq:Hdivb}
        \langle n \cdot \sigma, \mu\rangle
        & = 0
        && \forall\; \mu \in \trgK P_{p+2}(K),
    \end{align}
  \end{subequations}
  vanishes. Moreover, for any $\tau \in \Hdiv K$, there is a unique
  function $\pid \tau$ in $\Bd(K)$
  satisfying~\eqref{eq:PiHdiva}--\eqref{eq:PiHdivb}. It also
  satisfies~\eqref{eq:bdd-c}.
\end{lemma}

The remaining operator~$\pic$ will be defined after the next result.
It is modeled after the previous two lemmas, but requires considerably
more work.

\begin{lemma}  \label{lem:uniq}
  Any $E \in \Bc(K)$ satisfying 
\begin{subequations}     \label{eq:PiHcurltmp}
  \begin{align}
    \label{eq:PiHcurltmp1}
    (\phi,E)_K
    & 
    =0  
    && \forall \phi \in P_{p}(K)^3
    \\ \label{eq:PiHcurltmp2}
    \langle \mu, n \times  E\rangle
    & = 0
    && \forall \mu\in \trpK P_{p+1}(K)^3,
  \end{align}
\end{subequations}
vanishes.
\end{lemma}
\begin{proof}
  Integrating by parts twice and using~\eqref{eq:PiHcurltmp2}, we have 
  \begin{equation}
    \label{eq:66}
    \int_{\d K} \psi \, n \cdot \curl E =
  (\curl E, \grad \psi)_K = \ip{n \times E, \grad \psi} =0
  \qquad \forall \psi \in P_{p+2}(K).
  \end{equation}
  In addition, by Stokes theorem applied to one face $f$ of $K$, we
  have
  \begin{equation}
    \label{eq:77}
    \int_f \kappa\, n \cdot \curl E = 
    \int_{ \d f} \kappa  \,E \cdot t = 0
    \qquad \forall \kappa \in P_0(\d K),
  \end{equation}
  since $E\cdot t =0$ on all edges by the definition of $\Bc(K)$. The
  definition of $\Bc(K)$ also gives 
  \begin{equation}
    \label{eq:88}
    \int_{\d K} \phi \, n \cdot \curl E =0 \qquad \forall \phi \in \Pbod.
  \end{equation}
  Since $ P_{p+2}(\d K) = \Pbod + P^c_{p+2}(\d K) + P_0(\d K), $
  equations~\eqref{eq:66},\eqref{eq:77} and \eqref{eq:88} together imply
  \begin{equation}
    \label{eq:ncurl}
    \int_{\d K} \psi \,n\cdot \curl E = 0 \qquad
    \forall \psi \in P_{p+2}( \d K).
  \end{equation}
  Since $n \cdot \curl E \in P_{p+2}(\d K),$ we thus find that
  $n \cdot \curl E =0$ on $\d K.$ This implies that the tangential
  component of $E$ on $\d K$, namely
  $E_{\tp} = (n \times E) \times n$, 
  has vanishing surface curl, so it must equal a surface gradient,
  i.e., $E_{\tp}= \gradt v$ for some $v \in \Pc_{p+3}(\d K)$.  Moreover,
  since $E_{\tp}$ vanishes on all edges, $v$ may be chosen to be of the
  form $v = b_f v_p$ for some $v_p \in P_p(\d K)$, where $b_f$ is the
  product of all barycentric coordinates of $K$ that do not vanish
  a.e.\ on $f$.

  To use the remaining (as yet unused) condition in the definition of
  $\Bc(K)$, note that the tangential component of the coordinate
  vector $x$, namely $x_{\tp}$ is in $R_1(\d K)$.  Combining this 
  with~\eqref{eq:PiHcurltmp2}, we find that for all
  $\mu \in \trpK P_{p+1}(K)^3$ and any $\kappa \in \RRR$, 
  \begin{align}
    \nonumber
    0  &= \ip{ n \times E, \mu} + \kappa\int_{\d K} E_{\tp}\cdot x
   = \int_{\d K} E_{\tp} \cdot ( \mu \times n + \kappa x_{\tp}) 
    \\        \nonumber  
  & = 
      \sum_f \int_f \gradt (b_f v_p) \cdot ( \mu \times n + \kappa x_{\tp})
    \\
  &  \label{eq:bfvp}
     = \sum_f\int_f b_f v_p \divt( \mu \times n + \kappa x_{\tp}),
  \end{align}
  where the sums run over all faces $f$ of $\d K$.  For any
  $\mu \in \trpK P_{p+1}(K)^3$, the function $\mu \times n$ is in the
  Raviart-Thomas space on the closed manifold $\d K$ denoted by
  $R_{p+1}^c(\d K)$. (Note that unlike $R_1(\d K)$, this space \jrev{consists}
  of functions with the appropriate compatibility conditions across
  edges of $\d K$.) The surface divergence map
  \[
  \divt : R_{p+1}^c( \d K) \to 
  \bigg\{ w \in P_p(\d K): \int_{\d K} w =0\bigg\} 
  \]
  is surjective. Hence the term $\divt( \mu \times n + \kappa x_{\tp})$
  appearing in~\eqref{eq:bfvp} spans all of $P_p(\d K)$ as $\mu$ and
  $\kappa$ are varied. Choosing $\mu$ and $\kappa$ so that
  $\divt( \mu \times n + \kappa x_{\tp}) = v_p$, we conclude that $v_p$
  vanishes and hence $E_{\tp} = \gradt ( b_f v_p) = 0$, i.e.,
  \begin{equation}
    \label{eq:nxq0}
    n \times E = 0 \qquad \text{ on } \d K.
  \end{equation}

  Next, setting $ \phi = \curl r$ in~\eqref{eq:PiHcurltmp1} and
  integrating by parts, we obtain 
  \begin{equation}
    \label{eq:nxq}
    \int_K r \cdot \curl E  = 0 \qquad \forall r \in P_{p+1}(K)^3.     
  \end{equation}
  From~\eqref{eq:ncurl} and \eqref{eq:nxq}, it follows that
  $\tau = \curl E$ is in $\Bd(K)$, and furthermore, $\tau$
  satisfies~\eqref{eq:Hdiva} and \eqref{eq:Hdivb}. Hence, by
  Lemma~\ref{lem:div}, $\tau$ vanishes. Thus $\curl E = 0$ and
  consequently $E = \grad v$ for some $v \in P_{p+3}(K).$ Furthermore,
  by~\eqref{eq:nxq0}, we may choose $v = b_K v_{p-1}$ for some
  $v_{p-1} \in P_{p-1}(K)$, where $b_K$ is the product of all
  barycentric coordinates of $K$.  Then \eqref{eq:PiHcurltmp1} implies
  \[
  \int_K \grad( b_K v_{p-1}) \cdot \phi = 
  \int_K b_K v_{p-1}\cdot \dive \phi = 0
  \qquad \forall \phi \in P_p(K)^3.
  \]
  It now follows from the surjectivity of
  $\dive : P_p(K)^3 \to P_{p-1}(K)$ that $v_{p-1}$, and in turn
  $E = \grad( b_K v_{p-1})$, vanishes on $K$.
\end{proof}

The next lemma defines the operator $\pic$. It will be useful to
observe now that for any $E \in \Bc(K)$, 
\begin{equation}
  \label{eq:17}
\int_{\d K} ( (n \times E) \times n )\cdot r = 0
    \quad \forall r \in R_1(\d K)
    \iff \int_{\d K} ( (n \times E) \times n )\cdot x_{\tp} = 0.
\end{equation}
Indeed, while the forward implication is obvious, the converse follows
from~\eqref{eq:77}.  This shows that the condition that appears both
in the definition of $\Bc(K)$ and in~\eqref{eq:17} above, actually
amounts to just one constraint.

\begin{lemma}   \label{lem:piCurl}
  Given any $E \in \Hcurl K$, there is a unique $\pic E$ in $\Bc(K)$ 
  satisfying~\eqref{eq:PiHcurla}--\eqref{eq:PiHcurlb}.
\end{lemma}
\begin{proof}
  We need to estimate $n = \dim \Bc(K)$.  First, note that since
  $P_0(\d K) \cap \Pc_{p+2}(\d K)$ is a one-dimensional space of
  constant functions on $\d K$,
  \begin{align*}
    \dim \Pbod & = \dim P_{p+2}(\d K) - \dim( P^c_{p+2}(\d K) + P_0(\d K) )
                 \\
    & = \dim P_{p+2}(\d K) - \dim P^c_{p+2}(\d K) - \dim P_0(\d K)  + 1
    \\
    & = 6p + 11.
  \end{align*}
  The tangential component $E \cdot t$ of any $E \in N_{p+3}(K)$ is a
  polynomial of degree at most $p+2$ on each edge, so $ E \cdot t$
  represents $p+3$ constraints per edge. Hence, counting the number of
  constraints in the definition of $\Bc(K)$,
  \begin{align*}
    n = \dim \Bc(K) 
    & \ge \dim N_{p+3}(K)  - 6(p+3) - \dim(\Pbod) - 1
    \\
    & = 
  \dim N_{p+3}(K)  - 6(p+3) - (6 p+11) - 1
  \end{align*}
  where we have used~\eqref{eq:17}. Thus, 
  \begin{equation}
    \label{eq:nge}
    n \ge \dim  N_{p+3}(K)  -  12 p - 30.
  \end{equation}

  Next, we count the number of equations
  in~\eqref{eq:PiHcurla}--\eqref{eq:PiHcurlb}, namely
  \begin{align*}
    m
    &  = \dim( \trpK P_{p+1}(K)^3 ) + \dim P_p(K)^3 
    \\
    & = 2\dim P_{p+1}(\d K) - 6 (p+2) + \dim P_p(K)^3
      \\
    & = \dim N_{p+3}(K) - 6 (p+2) - 6(p+3).
  \end{align*}
  This together with~\eqref{eq:nge} implies that
  $m = \dim N_{p+3}(K) - 12 p - 30 \le n$ Thus, the
  system~\eqref{eq:PiHcurla}--\eqref{eq:PiHcurlb}, after using a
  basis, is an $m \times n$ matrix system of the form $A x =d$, where
  $x \in \RRR^n$ is the vector of coefficients in a basis expansion of
  $\pic E$ and $d$ is the right-hand side vector made using the given
  $E$.  By Lemma~\ref{lem:uniq}, $\nul (A) = \jrev{\{0\}}$.  Hence
  $ m \le n = \rank (A) + \nul (A) = \rank (A) \le \min(m,n)$ shows
  that $m = n$. The system determining $\pic E$ is therefore a square
  invertible system.
\end{proof}

\begin{lemma} \label{lem:contains}
  For all $v \in H^1(K)/\RRR$ and $E \in \Hcurl K$, 
  \[
  \grad (\pig v) \in   \Bc(K), \qquad 
  \curl (\pic E) \in \Bd(K).
  \]
\end{lemma}
\begin{proof}
  Let $e =\grad (\pig v) = \grad (\pigo v_0)$ where
  $\pigo v_0 \in \Bg(K)$ is as in~\eqref{eq:pigradbdr}. Since
  $\pigo v_0$ is constant along edges of $K$, $e$ must satisfy
  $t \cdot e=0$ along the edges. Moreover,
  \begin{equation}
    \label{eq:10}
    m_{\d K} (\pigo v_0) =0
  \end{equation}
  due to~\eqref{eq:pi0b}.  Hence, integrating by parts \jrev{on} any face
  $f$ of $\d K$,
  \begin{align*}
    \int_{f} e_{\tp} \cdot x_{\tp} 
    & = \int_{f} x_{\tp} \cdot \gradt (\pigo v_0)
      = -\int_f \pigo v_0 \divt( x_{\tp}) = -2 \int_f \pigo v_0.
  \end{align*}
  Summing over all faces $f$ of $\d K$ and using~\eqref{eq:10},  we conclude that 
  \[
  \int_{\d K} e_{\tp} \cdot x =0.
  \]
  Therefore, to finish proving that $e \in \Bc(K)$, it only remains to
  show that $\int_{\d K} \phi \,n \cdot \curl e =0$ for all
  $\phi \in \Pbod$. But this is obvious from the fact that $e$ is a
  gradient.

  Next, we need to show that $\sigma = \curl (\pic E)$ is in
  $\Bd(K)$. Since it is obvious that $\sigma \in R_{p+3}(K)$, it
  suffices to prove that 
  \begin{equation}
    \label{eq:13}
    \int_{\d K} \phi \,n \cdot \curl (\pic E) =0
  \end{equation}
  for all $\phi $ in $\Pbd$.  Note that $\Pbd$ can be orthogonally
  decomposed into its subspace $P_0(\d K) \cap \Pbd $ and its
  $L^2(\d K)$-orthogonal complement. The latter is a subspace of
  $\Pbod$ where~\eqref{eq:13} holds (since $\pic E \in \Bc(K)$).
  Hence it only remains to prove that~\eqref{eq:13} holds for $\phi$
  in $P_0(\d K) \cap \Pbd$. But Stokes theorem shows
  that~\eqref{eq:13} actually holds for all $\phi \in P_0(\d K)$ --
  cf.~\eqref{eq:77}.
\end{proof}

\begin{lemma}   \label{lem:commute}
  For all $v \in H^1(K), E \in \Hcurl K,$ and $\sigma \in \Hdiv K$,
  \begin{subequations}
    \begin{align}
      \label{eq:commut-a}
      \grad \pig v & = \pic \grad v,
      \\
      \label{eq:commut-b}
      \curl \pic E & = \pid \curl E,
      \\
      \label{eq:commut-c}
      \dive \pid \sigma & = \vpi_{p+2} \dive \sigma.
    \end{align}
  \end{subequations}
\end{lemma}
\begin{proof}
  By Lemma~\ref{lem:contains}, $\delta_1 = \grad( \pig v) - \pic \grad v$
  is in $\Bc(K)$. We will now show that $\delta_1$ satisfies~\eqref{eq:PiHcurltmp}.
  Let $\phi \in P_p(K)^3$ and consider
  \begin{align*}
    (\phi, \delta_1)_K
    & = (\phi, \grad( \pig v -v ))_K - (\phi, \pic (\grad v) - \grad v)_K.
  \end{align*}
  By Lemma~\ref{lem:piCurl}, $\pic \grad v$
  satisfies~\eqref{eq:PiHcurla}--\eqref{eq:PiHcurlb}, so the last term
  above vanishes. Integrating the remaining term on the right-hand
  side by parts, and using~\eqref{eq:PiH1a}--\eqref{eq:PiH1b}, 
  we find
  that
  \begin{equation}
    \label{eq:12}
    (\phi, \grad( \pig v -v ))_K =0 \qquad \forall \phi \in P_p(K)^3.
  \end{equation}
  This proves that $(\phi,\delta_1)_K=0$, i.e.,~\eqref{eq:PiHcurltmp1}
  holds.  Next, for any $\mu = \trpK(F),$ $F \in P_{p+1}(K)^3,$ we
  have
  \begin{align*}
    \ip{ \mu, n \times \delta_1}
    & = \ip{F \times n, \grad( \pig v -v )}
      - \ip{ \mu, n \times (\pic (\grad v) - \grad v) }.
  \end{align*}
  The last term vanishes due to~\eqref{eq:PiHcurlb}. Moreover,
  \[
  \ip{F \times n, \grad( \pig v -v )} = -(\curl F, \grad (\pig v -
  v))_K=0
  \]
  due to~\eqref{eq:12}, so we have proven that~\eqref{eq:PiHcurltmp2}
  holds as well. Hence by Lemma~\ref{lem:uniq}, $\delta_1=0$. This
  proves~\eqref{eq:commut-a}.

  To prove~\eqref{eq:commut-b}, we proceed similarly
  and show that $\delta_2 = \curl \pic E -\pid \curl E$ is
  zero. By Lemma~\ref{lem:contains}, we know that $\delta_2 \in \Bd(K)$, so
  if we prove that 
\begin{subequations} \label{eq:PiHdivtmp}
  \begin{align}
    \label{eq:PiHdivtmp1}
    (\delta_2, \phi)_{K}
    & 
    =0  
    && \forall \phi \in P_{p+1}(K)^3,
    \\  \label{eq:PiHdivtmp2}
    \langle n \cdot \delta_2, \mu\rangle_{\partial K}
    & = 0
    && \forall \mu \in \trgK P_{p+2}(K).
  \end{align}
\end{subequations}
then Lemma~\ref{lem:div} would yield $\delta_2=0$. To
prove~\eqref{eq:PiHdivtmp2},
\begin{align*}
  \langle n \cdot \delta_2, \mu\rangle_{\partial K}
  &= 
  \langle n \cdot (\curl (\pic E - E) + (I-\pid )\curl E), \mu\rangle_{\partial K}
  \\
  & = 
  \langle n \cdot \curl (\pic E - E), \mu\rangle_{\partial K}
  && \text{ by~\eqref{eq:Hdivb}}
  \\
  & = 
  \langle (\pic E - E) \times n, \gradt \mu\rangle_{\partial K} = 0
  && \text{ by~\eqref{eq:PiHcurlb}}.
\end{align*}
To prove~\eqref{eq:PiHdivtmp1}, 
\begin{align*}
  (\delta_2, \phi)_K
  & = (\curl (\pic E - E) + (I-\pid )\curl E, \phi)_K
  \\
  & = (\curl (\pic E-E), \phi)_K
    && \text{ by~\eqref{eq:PiHdiva}}
  \\
  & = ( \pic E-E, \curl\phi)_K + \ip{ (\pic E-E) \times n , \phi } =0
  && \text{ by~\eqref{eq:PiHcurla}--\eqref{eq:PiHcurlb}.}
\end{align*}
This finishes the proof of~\eqref{eq:PiHdivtmp} and
hence~\eqref{eq:commut-b} follows.

Finally, to prove~\eqref{eq:commut-c}, let
$\delta_3 = \dive \pid \sigma - \vpi_{p+2} \dive\sigma$ in
$P_{p+2}(K)$. For any $w \in P_{p+2}(K)$, integrating by parts, 
\begin{align*}
  (\delta_3, w)_K
  & = ( \dive (\pid \sigma - \sigma), w)_K
  \\
  & = -( \pid \sigma - \sigma, \grad w)_K
    + \ip{ n \cdot (\pid \sigma - \sigma), w},
\end{align*}
which vanishes by~\eqref{eq:PiHdiva}--\eqref{eq:PiHdivb}.  Hence
$\delta_3=0$ and~\eqref{eq:commut-c} is proved.
\end{proof}

\bigskip

\begin{proof}[Proof of Theorem~\ref{thm:Fortin}]
  The lemmas of this section prove all statements of
  Theorem~\ref{thm:Fortin} except~\eqref{eq:bdd-b}.  To
  prove~\eqref{eq:bdd-b}, we use a scaling argument and the
  commutativity properties of Lemma~\ref{lem:commute}.  Let $\hat K$
  denote the unit tetrahedron and let the $H(\mathrm{curl})$-Fortin
  operator on $\hat K$, defined as above, be denoted by $\pich.$ By
  the unisolvency result of Lemma~\ref{lem:piCurl}, \jrev{the
$\Hcurl {\hat K}$-boundedness
    of the sesquilinear forms $(q, \pich E)$ and
    $\ip{ F_\tp, n \times \pich E}$ (a consequence 
      of Lemma~\ref{lem:duality}),} and by finite dimensionality,
  there is a $C_0>0$ such that
  \[
  \| \pich \hat E \|_{\Hcurl {\hat K}}^2 \le C_0 \| \hat E \|_{\Hcurl {\hat K}}^2
  \]
  for all $\hat E \in \Hcurl {\hat K}$.
  Let
  $S_K : \hat K \to K$ be the one-to-one affine map that maps $\hat K$
  onto a general tetrahedron~$K$. For 
  $E: K \to \RRR^3$, define
  $
  \Phi(E) = (S_K')^t ( E \circ S_K),
  $
  and
  $\ncrl E K^2 = h_K^{-2} \| E \|_{L^2(K)}^2 + \| \curl E
  \|_{L^2(K)}^2$.
  {The elementary proofs of the following assertions
    (i)--(iii) are left to the reader.}
  \begin{enumerate}
  \item[(i)] $ E \in \Bc(K)$ if and only if
    $ \Phi(E) \in \Bc(\hat K).$ 

  \item[(ii)] There are constants $C_1,C_2$ depending only on the shape
    regularity of $K$ (but not on $h_K$) such that
    $ C_1 \ncrl { E }{K }^2 \le h_K \ncrl { \Phi(E) }{\hat K }^2 \le
    C_2 \ncrl { E }{K}^2.$

  \item[(iii)] $ \Phi(\pic E) = \pich \Phi(E).$  
  \end{enumerate}
  These three statements imply that 
  \[
  C_1  \ncrl{ \pic E }{K}^2 
  \le 
  h_K \ncrl { \pich\Phi(E) }{\hat K }^2
  \le 
  h_K C_0 \ncrl { \Phi(E)  }{\hat K }^2
  \le C_0C_2 \ncrl {  E }{K }^2.
  \]
  While this immediately gives the needed estimate for the $L^2$-part, namely
  \begin{equation}
    \label{eq:18}
    \|\pic E \|_{L^2(K)} \le C \| E \|_{\Hcurl K},
  \end{equation}
  we need to improve the estimate on the curl to finish the proof:
  For
  this, we use the commutativity property
  \begin{equation}
    \label{eq:16}
    \begin{aligned}
      \| \curl \pic E \|_{L^2(K)}  
      & = \| \pid \curl E \|_{L^2(K)} 
      && \text{ by~Lemma~\ref{lem:commute},}
      \\
      & \le C \| \curl E \|_{\Hdiv K}
      && \text{ by~Lemma~\ref{lem:div}}.
    \end{aligned}    
  \end{equation}
  The required estimate~\eqref{eq:bdd-b} follows from~\eqref{eq:18}
  and~\eqref{eq:16}.
\end{proof}

\bigskip

Before concluding this section, let us illustrate how to use
Theorem~\ref{thm:Fortin} for error analysis of DPG methods by an 
example.

\begin{example}[Primal DPG method for the Dirichlet problem]
  Consider the broken variational problem of Example~\ref{eg:primal}: Find
  $(u, \hat{\sigma}_n) \in X = \Ho^1(\om) \times H^{-1/2}(\d\oh)$
  such that~\eqref{eq:11} holds with 
  \[
  b( \, (u, \hat{\sigma}_n), y \,) 
  = (\grad u, \grad y)_h + \ip{ \hat{\sigma}_n, y}_h, \qquad
  Y = H^1(\oh).
  \]
  We want to analyze the DPG method given by~\eqref{eq:pdpg} with 
  \begin{align*}
  X_h 
    & = \{ (w_h, n \cdot \hat \tau_h) \in X: \;
      w_h|_K \in P_{p+1}(K) \text{ and } n \cdot \hat \tau_h|_{\d K} \in \trdK R_{p+1}(K) 
      \\
    & \hspace{4cm}
      \text{ for all (tetrahedral mesh elements) } K \in \oh \},
    \\
    Y_h & = \{ y_h \in H^1(\oh): y_h|_K \in P_{p+3}(K) \text{ for all } 
          K \in \oh \}.
  \end{align*}
  We have already shown in Example~\ref{eg:primal} that the inf-sup
  condition required for application of Theorem~\ref{thm:dpg} holds
  (and $Z= \{0\}$).  Hence to obtain optimal error estimates from
  Theorem~\ref{thm:dpg}, it suffices to verify
  Assumption~\ref{asm:pi}. We claim that Assumption~\ref{asm:pi} holds
  with $\vpi = \pig$. Indeed, 
  \begin{align*}
    b( \,(w_h, n \cdot \hat \tau_h),  y - \pig y)
    & = 
    (\grad w_h, \grad (y-\pig y)\, )_h 
      +
      \ip{ n \cdot \hat \tau_h, y-\pig y}_h \qquad
    \\
    & = -(\Delta w_h, y-\pig y )_h 
      + \ip{ n \cdot \hat \tau_h - 
      \frac{ \d w_h}{ \d n}, y-\pig y}_h       =0
  \end{align*}
  by applying~\eqref{eq:PiH1a}--\eqref{eq:PiH1b} element by element.
  Note that here we have used the fact that the discrete spaces have
  been set so that $-\Delta w_h|_K \in P_{p-1}(K)$ and
  $ n \cdot \hat \tau_h - n \cdot \grad w_h$ is a polynomial of degree
  at most $p$ on each face of $\d K$ (i.e., it is in
  $\trdK R_{p+1}(K)$), allowing us to
  apply~\eqref{eq:PiH1a}--\eqref{eq:PiH1b}. Applying
  Theorem~\ref{thm:dpg}, we recover the error estimates for this
  method, originally proved in~\cite{DemkoGopal13a}.\hfill\qeg
\end{example}

\section{Maxwell equations}   \label{sec:maxwell}

In this section, we combine the various tools developed in the
previous sections to analyze the DPG method for a model problem in
time-harmonic electromagnetic wave propagation.

\subsection{The cavity problem}

Consider a cavity $\om$, an open bounded connected and contractible
domain in $\RRR^3$, shielded from its complement by a perfect electric
conductor throughout its boundary $\d \om$.  If all time variations
are harmonic of frequency $\og>0$, then Maxwell equations within the
cavity reduce to these:
\begin{subequations} \label{eq:Maxwell}
  \begin{align}
    \label{eq:Maxwell1}
    -\ii \og \mu H + \curl E &  =0   && \text{ in } \om, 
    \\
    \label{eq:Maxwell2}
    -\ii \og \epsilon E - \curl H & = -J && \text{ in } \om, 
    \\
    \label{eq:Maxwell-PEC}
    n \times E & = 0 && \text{ on } \d \om.
  \end{align}
\end{subequations}
The functions $E, H, J : \om \to \CCC^3$ represent electric field,
magnetic field, and imposed current, respectively, and $\ii$ denotes
the imaginary unit. For simplicity we assume that the electromagnetic
properties $\epsilon$ and $\mu$ are {positive} and constant on each element
of the tetrahedral mesh~$\oh$. The number $\og>0$ denotes a fixed
wavenumber.  In this section we develop and analyze a DPG method
for~\eqref{eq:Maxwell}.

Eliminating $H$ from~\eqref{eq:Maxwell1} and~\eqref{eq:Maxwell2}, we
obtain the following second order (non-elliptic) equation
\begin{equation}
  \label{eq:9}
  \curl \mu^{-1} \curl E - \og^2 \epsilon E = f,
\end{equation}
where $f = \ii \og J$. The standard variational formulation for this
problem is obtained by multiplying~\eqref{eq:9} by a test function
$F \in \Hcurlo \om$, integrating by parts and using the boundary
condition~\eqref{eq:Maxwell-PEC}: Find $E \in \Hcurlo\om$ satisfying 
\begin{equation}
  \label{eq:20}
  (\mu^{-1} \curl E, \curl F)_\om - \og^2 (\epsilon E, F)_\om = \ip{ f, F}
\end{equation}
for any given $f \in \Hcurlo\om'$.  It is well-known~\cite{Monk03b}
that~\eqref{eq:20} has a unique solution for every $f \in \Hcurlo\om'$
whenever $\og$ is not in the countably infinite set $\varSigma$ of
resonances of the cavity~$\om$. Throughout this section, we assume
$\og \not\in \varSigma$. This wellposedness result provides an
accompanying stability estimate, namely there is a constant $\cog>0$
such that
\begin{equation}
  \label{eq:19}
  \| E \|_{\Hcurl\om} \le \cog \| f \|_{\Hcurlo\om'}
\end{equation}
for any $f \in \Hcurlo\om'$ and $E\in \Hcurlo\om$
satisfying~\eqref{eq:20}.  Note that the stability constant $\cog$
may blow up as $\og$ approaches a resonance. We continue to use $C$ to
denote a generic mesh-independent constant, which in this section may
depend on $\og, \mu,$ and $\eps$ as well.

\subsection{Primal DPG method for the cavity problem} 

The primal DPG method for the cavity problem is obtained by
breaking~\eqref{eq:20}.  Multiply~\eqref{eq:9} by a (broken) test
function $F \in \Hcurl \oh$ and integrate by parts, element by
element, to get
\[
( \mu^{-1} \curl E, \curl F)_h \jrev{+} \ip{ n \times \mu^{-1} \curl E, F}_h 
- \og^2 (\veps E, F)_h = (f, F)_h.
\]
Now set $ \Hd \equiv n \times \hat H = (\ii \og)^{-1}n \times \mu^{-1} \curl E$ to
be an independent interface unknown which is to be found in
$\Hhdiv\oh$. This leads to the variational problem~\eqref{eq:8} with
the following spaces and forms:
\begin{subequations} \label{eq:dpgmaxwellcavity}
\begin{gather}
  X_0   = \Hcurlo \om,        \qquad   Y=\Hcurl\oh,
  \\
  \hat X  = \Hhdiv{\d\oh},     \qquad Y_0 = \Hcurlo\om,
  \\
  b_0( E,F) =   ( \mu^{-1} \curl E, \curl F)_h  - \og^2 (\veps E, F)_h, 
   \\
   \hat b( \Hd, F) = \jrev{\ii} \og\ip{\Hd, F}_h.
\end{gather}  
\end{subequations}
This is the primal DPG formulation for the Maxwell cavity problem.

The numerical method discretizes the above variational problem using
subspaces $X_h \subset X = X_0 \times \hat X$ and $Y_h \subset Y$
defined by
\begin{subequations}
  \label{eq:MaxwellXYh}
\begin{align}
  X_h
  & = \{ ( E_h, n \times \hat H_h) \in 
    {    \Hcurlo\om \times \Hhdiv{\d\oh} :}\;
        n \times \hat H_h|_{\d K}  \in \trcK P_{p+1}(K)^3, 
  \\ \nonumber 
  & \hspace{5.3cm} \text{ and }   
    E_h|_K \in P_{p}(K)^3 \text{ for all } K \in \oh \}, 
  \\ 
  Y_h & = \{ F_h \in \Hcurl\oh:  
        F_h|_K \in N_{p+3}(K) \text{ for all } K \in \oh \}.
\end{align}
\end{subequations} 
We have the following error bound for the numerical solution in terms
of the mesh size $h = \max_{K \in \oh} h_K$ and polynomial degree $p\ge 1$.

\begin{corollary}   \label{cor:primalMaxwell}
  Suppose $(E_h, n \times \hat H_h) \in X_h$ is the DPG solution given
  by~\eqref{eq:pdpg} with forms and spaces set
  by~\eqref{eq:dpgmaxwellcavity} and let $(E, n \times \hat H) \in X$
  be the exact solution of~\eqref{eq:11}. Then, there exists a $C$
  depending only on $\og$, $p$, and the shape regularity of the mesh
  such that
  \begin{align*}
  \| E - E_h \|_{\Hcurl\om} 
    & + 
  \| n \times ( \hat H - \hat H_h)\|_{\Hhdiv{\d\oh}}
    \\
    & \le C h^p 
  \left( 
    | E |_{H^{p+1}(\om)}
    +
    | \curl E |_{H^{p+1}(\om)}
    +
    | \curl H |_{H^{p+1}(\om)}
  \right).
  \end{align*}
\end{corollary}
\begin{proof}
  To apply Theorem~\ref{thm:dpg}, we must verify the inf-sup
  condition~\eqref{eq:infsupbroken} for the broken form. As in the
  previous examples, as a first step, we verify the inf-sup condition
  for the unbroken form stated in Assumption~\ref{asm:A0}.  Given any
  $E \in \Hcurlo\om$, let $f_E \in \Hcurlo\om'$ be defined by
  $\ip{ f_E, F} = (\mu^{-1} \curl E, \curl F)_\om - \og^2 (\epsilon E,
  F)_\om$
  for all $F \in \Hcurlo\om$. Then, \eqref{eq:9} and~\eqref{eq:19}
  \jrev{imply}
  \[
  \| E \|_{\Hcurl \om} \le \cog \| f_E \|_{\Hcurlo\om'} = 
  \cog \sup_{F \in \Hcurl\om} \frac{ |b_0( E,F) |}
  { \hspace{0.5cm} \| F \|_{\Hcurl \om} \hspace{-0.5cm}},
  \]
  i.e., Assumption~\ref{asm:A0} holds with $c_0 = \cog^{-1}.$
  Assumption~\ref{asm:hybrid}, with $\hat c = \og^{-1}$ is immediately verified
  by~\eqref{eq:conformequiv3} and~\eqref{eq:duality3} of
  Theorem~\ref{thm:duality}. Hence Theorem~\ref{thm:hybrid}
  verifies~\eqref{eq:infsupbroken} and also shows that $Z=\{0\}$.  The
  only remaining condition to verify before applying
  Theorem~\ref{thm:dpg} is Assumption~\ref{asm:pi}, which immediately
  follows by the choice of spaces and Theorem~\ref{thm:Fortin}. 

  Applying Theorem~\ref{thm:dpg}, we find that
  \begin{align*}
    \| E - E_h \|_{\Hcurl\om}^2 
    & + 
    \| n \times (\hat H   - \hat H_h)\|_{\Hhdiv{\d\oh}}^2
    \\
    \le &
         \; C \inf_{(G_h, n \times \hat R_h) \in X_h} 
      \left[ 
      \| E - G_h \|_{\Hcurl\om}^2 +
      \| n \times \hat H - n \times \hat R_h\|_{\Hhdiv{\d\oh}}^2 \right].
  \end{align*}
  Now, $H = (i\og\mu)^{-1} \curl E$ is an extension to $\om$ of the
  exact interface solution $n \times \hat H$. Moreover, the interface
  function $n \times \hat R_h$ appearing above can be extended into
  $X_{0,h}^p = \{ r \in \Hcurlo\om: \; r_h|_K \in P_{p+1}(K)^3\}$.
  Since the interface norm is the minimum over all extensions, by
  standard approximation estimates (see e.g.,
  \cite[Theorem~8.1]{DemkoGopalSchob12}),
  \begin{align*}
    \| E & - E_h \|_{\Hcurl\om}^2 
     + 
    \| n \times (\hat H - \hat H_h) \|_{\Hhdiv{\d\oh}}^2
    \\
    & \le C
      \left[
      \inf_{G_h \in X_{0,h}^p} \| E - G_h \|_{\Hcurl\om}^2 
      +
      \inf_{R_h \in X_{0,h}^{p+1}} \| H - R_h \|_{\Hcurl\om}^2
      \right]
    \\
    &
      \le 
      C  \sum_{K\in \oh}  
      \bigg[ h_K^{2(s_1+1)} |E|_{H^{s_1+1}(K)}^2
      + 
      h_K^{2s_1} |\curl E|_{H^{s_1+1}(K)}^2 + 
    \\
    & \hspace{2.1cm}
       h_K^{2 (s_2+1)} |H|_{H^{s_2+1}(K)}^2
      +
      h_K^{2s_2} | \curl H |_{H^{s_2+1}(K)}^2
      \bigg],
  \end{align*}
  where $1/2<s_1 \le p$ and
  $1/2<s_2 \le p+1$. Hence the corollary follows.
\end{proof}

\begin{remark}
  Unlike the standard finite element method, for the DPG method, there
  is no need for~$h$ to be ``sufficiently small'' to assert \rev{the
    convergence estimate of Corollary~\ref{cor:primalMaxwell}.}
  \jrev{This property has been called {\em absolute stability} by some
    authors and other methods possessing this property are
    known~\cite{FengWu14}.}
\end{remark}

\subsection{{Alternative} formulations of the same problem}
\label{ssec:altMax}

In Example~\ref{eg:diffus}, we saw that a single
diffusion-convection-reaction equation admits various different
formulations. The situation is similar with Maxwell equations.  First,
let us write~\eqref{eq:Maxwell} in operator form using an operator $A$
(analogous to the one in~\eqref{eq:Adiffus}, but now) defined by
\[
A 
\begin{bmatrix}
  H \\ E
\end{bmatrix}
= 
\begin{bmatrix}
  \ii \og \mu & -\curl \\
  \curl & \ii \og\eps 
\end{bmatrix}
\begin{bmatrix}
  H \\ E
\end{bmatrix}
= 
\begin{bmatrix}
  \ii \og \mu H - \curl E \\
  \ii \og \eps E + \curl H 
\end{bmatrix}
\]
as $A (H,E) = (0,J)$
for some given $J$ in $L^2(\om)^3$.  However, we will not restrict to
right-hand sides of this form as we will need to allow the most
general data possible in the ensuing wellposedness studies.

We view $A$ as an unbounded closed operator on $L^2(\om)^6$ whose
domain is
\[
\dom(A)  = \{ (H, E) \in \Hcurl\om^2 : n \times E = 0 \text{ on } \d \om \}.
\]
It is easy to show that its adjoint (in the sense of closed operators)
is the closed operator $A^*$ given by
\[
A^* 
\begin{bmatrix}
  H \\ E 
\end{bmatrix}
= 
\begin{bmatrix}
  -\ii\og \mu & \curl \\
  -\curl  & -\ii \og\eps
\end{bmatrix}
\begin{bmatrix}
  H \\ E 
\end{bmatrix}
= 
\begin{bmatrix}
  -\ii \og \mu H + \curl E \\
  -\ii \og \eps E - \curl H 
\end{bmatrix},
\]
whose domain is the following subspace of $L^2(\om)^6$:
\[
\dom(A^*) 
 = \{ (H, E) \in \Hcurl\om^2 : E_{\tp} =0  \text{ on } \d \om \}.
\]
Classical arguments show that both $A$ and $A^*$ are injective. 
To facilitate comparison, we list all our formulations at once,
including the already studied primal form.
\begin{description}

  \item[Strong form] 
    Let $x = (H,E)$ be a group variable. Set
    \begin{alignat*}{4}
      &X_0   = \Hcurl \om \times \Hcurlo \om,        
      &\quad& Y= Y_0=L^2(\om)^6.
    \end{alignat*}
    Note that $X_0 = \{ \dom(A), \|\cdot \|_{\Hcurl\om}\},$ i.e.,
    $\dom(A) $ considered as a subspace of $\Hcurl\om^2$ (rather than
    as a subspace of $L^2(\om)^6$).  The Maxwell problem is
    to find $x \in X_0,$ given $f \in Y_0,$ such that $ A x = f.  $
    This fits into our variational framework~\eqref{eq:5} by setting
    $b_0$ to
  \begin{gather*}
    \quad
    b_0^S( x, y) = (Ax, y)_\om.
  \end{gather*}

  \item [Primal form for $E$] %
    This is the same as in~\eqref{eq:dpgmaxwellcavity}, i.e., 
    with the  spaces as set there, 
    with $\hat b$ set to
    $\hat b^E( \hat{H}_\dv, F) = -\ii \og \ip{ \hat{H}_\dv, F}_h$
    and $b_0$ set to
    \begin{gather*}
      b_0^E(E,F)
       = (\mu^{-1}\curl E, \curl F)_\om - \og^2 (\eps E,F)_\om, 
    \end{gather*}
    the electric primal formulation is~\eqref{eq:5}
    and its broken version is~\eqref{eq:8}.

  \item [Primal form for $H$] %
    Eliminating $E$ from~\eqref{eq:Maxwell}, we obtain
    $\curl \eps^{-1} \curl H - \og^2 \mu H = \curl \eps^{-1} J$ and a
    (possibly nonhomogeneous) boundary condition on
    $n \times\eps^{-1}\curl H$.  With this in place of~\eqref{eq:20} as the
    starting point and repeating the derivation that led
    to~\eqref{eq:dpgmaxwellcavity}, we obtain the following magnetic
    primal form. Set 
    \begin{alignat*}{4}
      &X_0   = \Hcurl \om, 
      &\quad& Y_0=X_0,
      \\
      & \hat X = \Hhdivo\oh,
      &\quad & Y = \Hcurl\oh,
      \\
      & b_0^H( H,F)
      =   ( \eps^{-1} \curl H, \curl F)_h  - \og^2 (\mu H, F)_h, 
      &\qquad&
      \hat b^H( \Ed, F) =  \ii \og\ip{ \Ed, F}_h.
    \end{alignat*}
    With $\hat b$ set to $\hat b^H$ and $b_0$ set to $b_0^H$, the
    magnetic primal formulation is~\eqref{eq:5} and its broken version
    is~\eqref{eq:8}.

  \item[Ultraweak form] %
    This form is obtained by integrating by parts all equations of the
    strong form.  Using group variables $x = (H,E)$, $y = (R,S)$ and
    $\hat x = ( \Ht, \Et)$, set
    \begin{subequations}      \label{eq:MaxUW}
    \begin{alignat}{4}
    &X_0   = L^2(\om)^6,        
    &\quad&  Y_0  = \Hcurl\om \times \Hcurlo\om,
    \\
    & Y  = \Hcurl\oh^2,
    &\quad & 
    \hat X = \Hhcurl{\d\oh} \times \Hhcurlo{\d\oh},
    \\
    & b_0^U( x,y) = (x, A^* y)_h,        
    &\quad&  
    \hat{b}^U (\hat x, y)
    = \ip{ \Ht, n \times S}_h - \ip{ \Et, n \times R}_h
    \end{alignat}
    \end{subequations}
    and consider formulations~\eqref{eq:5} and~\eqref{eq:8} with
    $b_0 = b_0^U$ and $\hat b = \hat{b}^U$.  Note that in the
    definition of $b_0^U$, the operator $A^*$ is applied element by
    element, per our tacit conventions when using the
    $(\cdot,\cdot)_h$-notation.

  \item[Dual Mixed form] %
    Among the two equations in the strong form, if one weakly imposes
    (by integrating by parts) the first equation and strongly imposes
    the second, then we get the following dual mixed form. Set
    \begin{alignat*}{4}
    &X_0   = \Hcurl\om \times L^2(\om)^3,
    & \qquad & Y_0 = X_0,
    \\
    &\hat X  = \Hhcurlo{\d\oh}, 
    & \quad &  Y = \Hcurl \oh \times L^2(\om)^3
    \end{alignat*}
    and consider~\eqref{eq:5} with $b_0 =
    b_0^D$, 
    \begin{align*}
      b_0^D(\, (H,E), (R,S)\,)
      & = (\ii \og \mu H, R)_\om
        - (E, \curl R)_h        \\
      & +  (\ii \og \eps E + \curl H, S)_h.
    \end{align*}
    Its broken version is~\eqref{eq:8} with $\hat b$ set to
    $\hat b^D( \Et, (R,\jrev{S})\, ) = \ip{ \Et, n \times R}_h$.

  \item[Mixed form]  %
    Reversing the roles above and weakly imposing the second equation while
    strongly imposing the first, we get another mixed formulation. Set 
    \begin{alignat*}{4}
    &X_0   = L^2(\om)^3\times \Hcurlo\om,
    &\qquad&   Y_0  = X_0,
    \\
    &\hat X  = \Hhcurl{\d\oh},
    &\qquad& Y  =  L^2(\om)^3 \times \Hcurl\oh
    \end{alignat*}
    and consider~\eqref{eq:5} with $b_0 =
    b_0^M$, 
    \begin{align*}
      b_0^M(\, (H,E), (R,S)\,)
      & = (\ii \og \mu H - \curl E, R)_h \\
      & + (\ii \og \eps E, S)_\om + (H, \curl S)_h. 
    \end{align*}
    Its broken version is~\eqref{eq:8} with $\hat b$ set to
    $ \hat b^M( \Ht, (R,S)\,) = \ip{ \Ht, n \times S}_h.$
  \end{description}
  These form a total of six unbroken and five broken formulations,
  counting the already discussed broken and unbroken electric primal
  formulation. To analyze the remaining formulations, let us begin by
  verifying Assumption~\ref{asm:A0} for all the unbroken formulations.
  To this end, label the statement of Assumption~\ref{asm:A0} with
  $b_0$ set to the above-defined $b_0^I$ as ``$(I)$'' for all
  $I \in \{E,H,S,U,D,M\}$.  Then (analogous to the equivalences in
  Figure~\ref{fig:impl} for the elliptic example) we now have
  equivalence of statements $(D), (E), \ldots$ as proved next.

  \begin{theorem}  \label{thm:MaxwellCycles}
    The following implications hold:
    \[
    \begin{tikzpicture}[
        to/.style={->, >=latex', semithick, 
          double,double distance=2pt,    
        },
        every node/.style={align=center}
        ]
        

        \matrix[row sep=6mm,column sep=8mm] {
          & \node (H) { $(H)$ }; & \node (D) { $(D)$ }; 
          \\
          \node (S) { $(S)$ }; & & & \node (U) {$(U)$};
          \\ 
          & \node (E) {$(E)$};  &  \node (M) {$(M)$};
          \\
      };

        \draw[to] (E) -- (S) ;
        \draw[to] (S) -- (U) ;
        \draw[to] (U) -- (D) ;
        \draw[to] (D) -- (H) ;
        \draw[to] (H) -- (S) ;
        \draw[to] (U) -- (M) ;
        \draw[to] (M) -- (E) ;

      \end{tikzpicture}
      \]
  \end{theorem}

  \begin{proof}
    We begin with the most substantial of all the implications, which
    allows us to go from the strongest to the weakest formulation.
    When there can be no confusion, let us abbreviate Cartesian
    products of $L^2(\om)$ as simply $L$ and write $W$ for
    $\Hcurl \om \times \Hcurlo\om$. Clearly, $W$ is complete in the
    $\Hcurl\om^2$-norm. It is easy to see that the graph norms
    $(\| x \|_L^2 + \| A x \|_L^2)^{1/2}$ and
    $(\| x\|_L^2 + \| A^* x \|_L^2)^{1/2}$ are both equivalent to the
    $\Hcurl\om^2$-norm, so $W$ is a Hilbert space in any of these
    norms. These norm equivalences show that the inf-sup
    condition~$(S)$ holds if and only if
    \begin{equation}
      \label{eq:Abddbelow}
      C \| x \|_L \le \| A x\|_L
      \qquad \forall  x \in \dom(A).
    \end{equation}

    {\underline{$(S) \implies (U)$:}} The bound~\eqref{eq:Abddbelow} implied
    by $(S)$ shows that the range of $A$
    is closed.  By the closed range theorem for closed operators,
    range of $A^*$ is closed.  Since $A^*$ is also injective, it follows
    that
    \begin{equation}
      \label{eq:A'bddbelow}
      C \| y \|_L \le \| A^* y\|_L
      \qquad \forall  y \in \dom(A^*)      
    \end{equation}
    holds with the same constant as in~\eqref{eq:Abddbelow}.  This in
    turn implies that the following inf-sup condition holds:
    \[
    C \| y \|_W \le \sup_{x \in L} \frac{ |(x, A^*y)_\om| }{ \| x \|_L}
    \qquad \forall 
    y \in W.
    \]
    Thus, to complete the proof of $(U)$, it suffices to show that 
    \begin{equation}
      \label{eq:US}
      \inf_{x \in L} \sup_{y \in W} 
      \frac{ | (x, A^*y)_\om| }{  \| x \|_{L} \| y \|_W}
      =
      \inf_{y \in W} \sup_{x \in L}
      \frac{ | (x, A^*y)_\om| }{  \| x \|_{L} \| y \|_W}.
    \end{equation}
    For completeness, we now describe the standard argument that shows
    that one may reverse the order of inf and sup to
    prove~\eqref{eq:US}.  Viewing $A^*: W \to L$ as a bounded linear
    operator, we know that it is a bijection because
    of~\eqref{eq:Abddbelow} and~\eqref{eq:A'bddbelow}. Hence
    $(A^*)^{-1}: L \to W$ is bounded.  The right-hand side
    of~\eqref{eq:US} equals its operator norm $\| (A^*)^{-1} \|$. The
    left hand side of~\eqref{eq:US} equals the operator norm of the
    dual of $(A^*)^{-1}$ (considered as the dual operator of a
    continuous linear operator with $L$ as the pivot space identified
    to be the same as its dual space).  The norms of a continuous linear
    operator and its dual are equal, so~\eqref{eq:US} follows.

    \underline{$(U) \implies (D)$:} Let $x = (H,E)$ and $y = (R,S)$ be in
    $W=\Hcurl\om \times \Hcurlo\om$. Clearly, $W$ is contained
    in $Y_0^D = \Hcurl \om \times L^2(\om)^3$.  Because of the extra
    regularity of~$S$, we may integrate by parts the last term in the
    definition of $b_0^D( (H,E),(R,S))$ to get that
    $ b_0^D (x,y) = b_0^U(x,y).$ Hence using $(U)$, 
    \begin{align}\label{eq:24}
      \sup_{y \in Y_0^D}      
      \frac{ |b_0^D(x,y)|}{ \| y \|_{Y_0^D}}
      & \ge 
      \sup_{y \in W}
      \frac{ |b_0^D(x,y)|}{ \| y \|_{W}}
        = 
      \sup_{y \in W}
        \frac{ |b_0^U(x,y)|}{ \| y \|_{W}}
        \ge C \| x \|_{L}.
    \end{align}
    Thus, to finish the proof of $(D)$, we only need to control
    $\curl H$ using the last term of $b_0^D$.
    \begin{align*}
      \| \curl H \|_{L} 
      & 
      = 
        \sup_{S \in L} 
        \frac{|(\curl H, S)_\om|}{\| S \|_{L}}
      = 
        \sup_{S \in L} 
        \frac{|  b_0^D( (H,E), (0,S)) - (\ii\og\veps E, S) |}
        {\| S \|_{L}}
        \\
      & \le 
      \sup_{y \in Y_0^D}      
      \frac{ |b_0^D(x,y)|}{ \| y \|_{Y_0^D}}
        + \| \ii \og \veps E \|_{L}.
    \end{align*}
    Using~\eqref{eq:24} to bound the last term, the proof of $(D)$ is
    finished.

    \underline{$(D) \implies (H)$:} For any $H \in \Hcurl\om,$  set 
    $\ell_H(R) = (\eps^{-1} \curl H, \curl R)_\om - \og^2 (\mu H,
    R)_\om.$ We need to prove $(H),$ which  is equivalent to 
    \begin{equation}
      \label{eq:25}
      \| H \|_{\Hcurl\om} \le C \| \ell_H \|_{\Hcurl\om'}.
    \end{equation}
    Introducing a new variable $E = -(\ii \og \eps)^{-1} \curl H$, we
    find that
    $(\ii\og)^{-1} \ell_H(R) = -(E, \curl R)_\om + (\ii \og \mu H,
    R)_\om.$
    Hence 
    \[
    b_0^D ((H,E), (R,S)) = (\ii\og)^{-1} \ell_H(R)
    \qquad \forall 
    (R,S) \in \Hcurl \om \times L^2(\om)^3.
    \] 
    Hence~\eqref{eq:25} immediately follows from~$(D)$.

    \underline{$(H) \implies (S)$:} To prove the inf-sup condition
    $(S)$, it is enough to prove~\eqref{eq:Abddbelow} for all
    $ x = (H,E) \in W.$ Given any $F, G \in L^2(\om)^3,$ the equation
    $Ax = (F,G)$ is the same as the system
    \begin{subequations}
    \begin{align}
      \label{eq:27}
      \ii \og \mu  H - \curl E & = F,
      \\ \label{eq:28}
      \ii \og \eps E + \curl H & = G.
    \end{align}      
    \end{subequations}
    We multiply~\eqref{eq:28} by the conjugate of $\eps^{-1} \curl R$,
    for some $R \in \Hcurl\om$ and integrate by parts, while we
    multiply~\eqref{eq:27} by $-\ii \og R$ and solely
    integrate. The result is
    \begin{align*}
      (\eps^{-1} \curl H, \curl R)_\om + (\ii \og \curl E, R)_\om 
      & = (G, \veps^{-1} \curl R)_\om,
      \\
      -(\og^2 \mu H,  R)_\om 
      - (\ii \og\curl E, R)_\om & = -(F, \ii \og R)_\om.
    \end{align*}
    Adding the above two equations together, 
    we get the primal form $b_0^H(H,R) = \ell(R)$ where
    $\ell(R) = (G, \veps^{-1} \curl R)_\om  -(F, \ii \og R)_\om.$
    Hence the given inf-sup condition $(H)$ implies
    \[
    C \| H \|_{\Hcurl\om}
    \le \sup_{R \in \Hcurl\om} 
    \frac{|b_0^H(H,R)|}
    {\hspace{0.5cm}\|R \|_{\Hcurl\om} \hspace{-0.5cm}} \hspace{0.5cm}
    =  \sup_{R \in \Hcurl\om} 
    \frac{| \ell(R)| }
    {\hspace{0.5cm}\|R \|_{\Hcurl\om} \hspace{-0.5cm}}.
    \]
    Since
    $|\ell(R)| \le C ( \| F \|_{L} + \| G \|_{L}) \| R \|_{\Hcurl
      \om}$,
    this provides the required bound for $\| H \|_{L}$.  Since
    $\curl H$ is also bounded, equation~\eqref{eq:28} yields a bound
    for $\|E \|_{L}$. Combining these bounds,~\eqref{eq:Abddbelow} follows.

    To conclude the proof of the theorem, we note that the proofs
    of the implications $(U) \implies (M)$, $(M) \implies (S)$,
    $(E) \implies (S)$ are similar to the proofs of
    $(U) \implies (D)$, $(D) \implies (H)$, and $(H) \implies (S)$,
    respectively.
  \end{proof}

  Theorem~\ref{thm:MaxwellCycles} verifies Assumption~\ref{asm:A0} for
  all the formulations because we know from~\eqref{eq:19} that~$(E)$
  holds. Assumption~\ref{asm:hybrid} can be easily verified for all
  the broken formulations using
  Theorem~\ref{thm:duality}. Assumption~\ref{asm:pi} can be verified
  using Theorem~\ref{thm:Fortin}. Hence convergence rate estimates
  like in Corollary~\ref{cor:primalMaxwell} can be derived for each of
  the broken formulations. We omit the repetitive details.

\section{{Numerical studies}}  \label{sec:numer}

In this section, we present some numerical studies focusing on the
Maxwell example. Numerical results for other examples, including the
diffusion-convection-reaction example, can be found
elsewhere~\cite{ChanHeuerBui-T14,DemkoHeuer13}. The numerical studies
are not aimed at verifying the already proved convergence results, but
rather at investigations of the performance of the DPG method beyond
the limited range of applicability permitted by the theorems.  All
numerical examples presented in this section have been obtained with
$hp3d$, a 3D finite element code supporting anisotropic~$h$ and $p$
refinements and solution of multi-physics problems involving variables
discretized compatibly with the $H^1(\om)$-$\Hcurl\om$-$\Hdiv\om$
exact sequence of spaces.  The code has recently been equipped with a
complete family of orientation embedded shape functions for elements
of many shapes~\cite{FuentKeithDemko15}. The remainder of this section
is divided into results from two numerical examples.

\begin{example}[Smooth solution]
  We numerically solve the time-harmonic Maxwell equations setting
  material data to
  $$
  \epsilon = \mu = 1,
  \quad \omega = 1,
  $$
  and $\om$ to the unit cube.  To obtain $\oh$, the unit cube was
  partitioned first into five tetrahedra: four similar ones adjacent
  to the faces of the cube, and a fifth inside of the cube. We have
  used the refinement strategy of~\cite{KrizeStrou97} to generate a
  sequence of successive uniform refinements.  On these meshes,
  consider the primal DPG method for $E$, described
  by~\eqref{eq:dpgmaxwellcavity}, with data set so that the exact
  solution is the following smooth function.
$$
E_1 = \sin \pi x_1 \, \sin \pi x_2\, \sin \pi x_3,\quad E_2=E_3 = 0 \, .
$$
Instead of the pair of discrete spaces~\eqref{eq:MaxwellXYh} that we
know is guaranteed to work by our theoretical results, we experiment
with these discrete spaces:
\begin{subequations}
  \label{eq:XYexpt}
\begin{align}
  X_h
  & = \{ ( E_h, n \times \hat H_h) \in 
    \Hcurlo\om \times \Hhdiv{\d\oh} :\;
        n \times \hat H_h|_{\d K}  \in \trcK N_{p}(K), 
  \\ \nonumber 
  & \hspace{5.3cm} \text{ and }   
    E_h|_K \in N_{p}(K) \text{ for all } K \in \oh \}, 
  \\ 
  Y_h & = \{ F_h \in \Hcurl\oh:  
        F_h|_K \in N_{p+2}(K) \text{ for all } K \in \oh \}.
\end{align}  
\end{subequations}

\begin{figure}
\centering
\begin{subfigure}{\textwidth}
  \includegraphics[width=\textwidth]{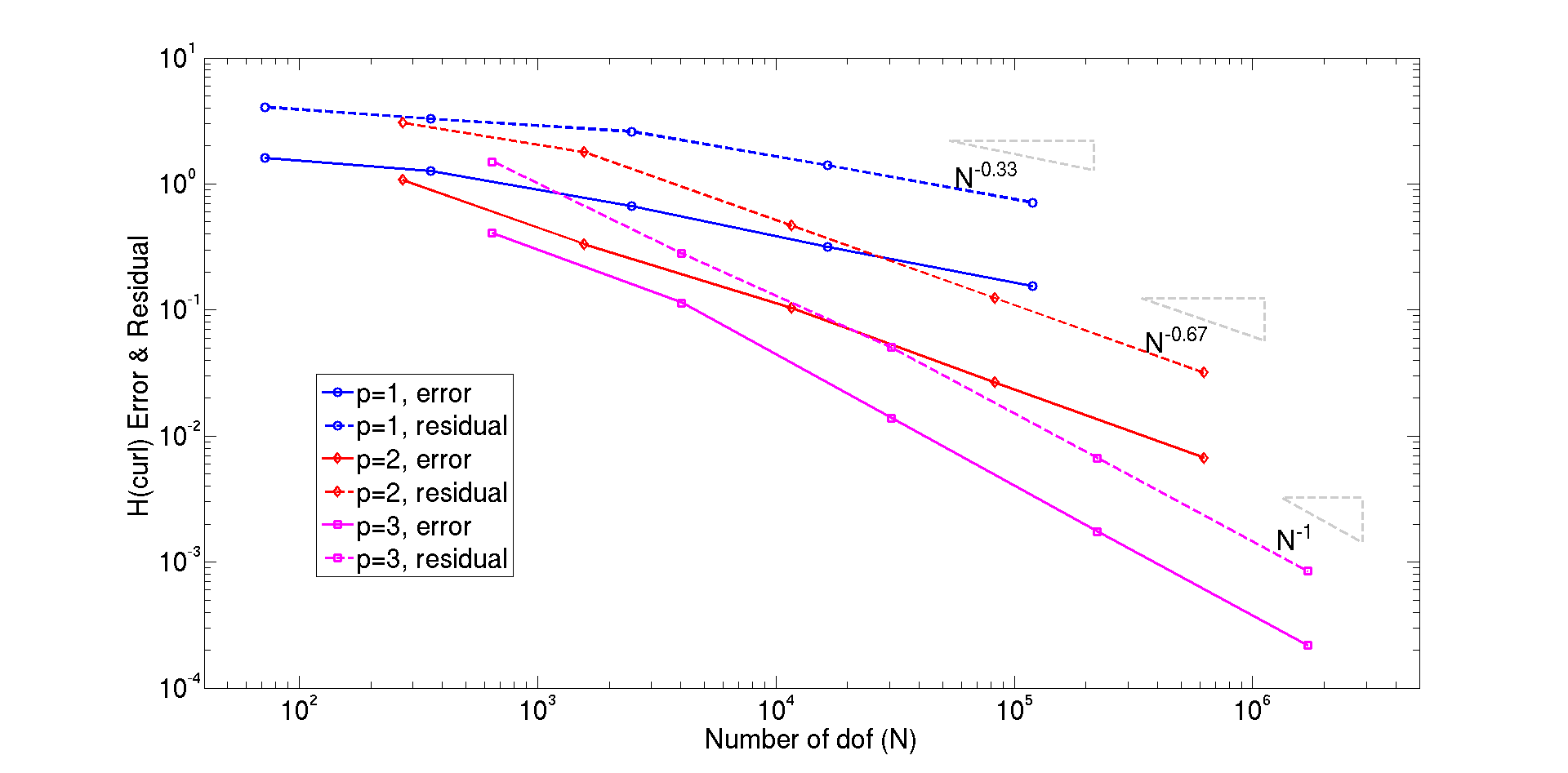}
  \caption{Tetrahedral meshes.}
  \label{fig:smooth_primal_rates_tet}
\end{subfigure}
\begin{subfigure}{\textwidth}
  \includegraphics[width=\textwidth]{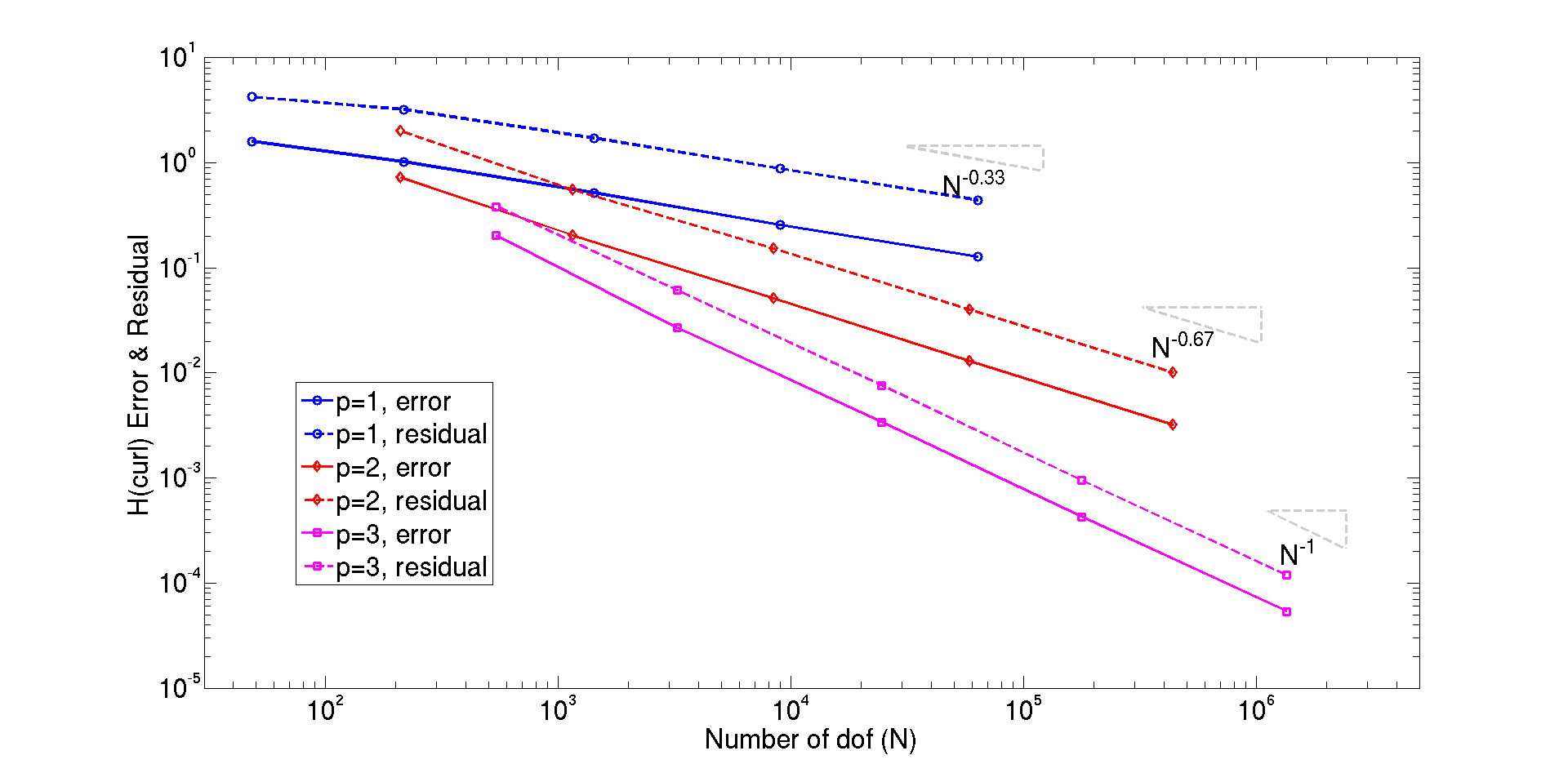}
  \caption{Hexahedral meshes.}
  \label{fig:smooth_primal_rates_hex}
\end{subfigure}
\caption{Rates for the case of smooth solution and 
  primal formulation.}
\label{fig:smooth_primal_rates}
\end{figure}

\begin{figure}
\centering
\begin{subfigure}{\textwidth}
  \includegraphics[width=\textwidth]{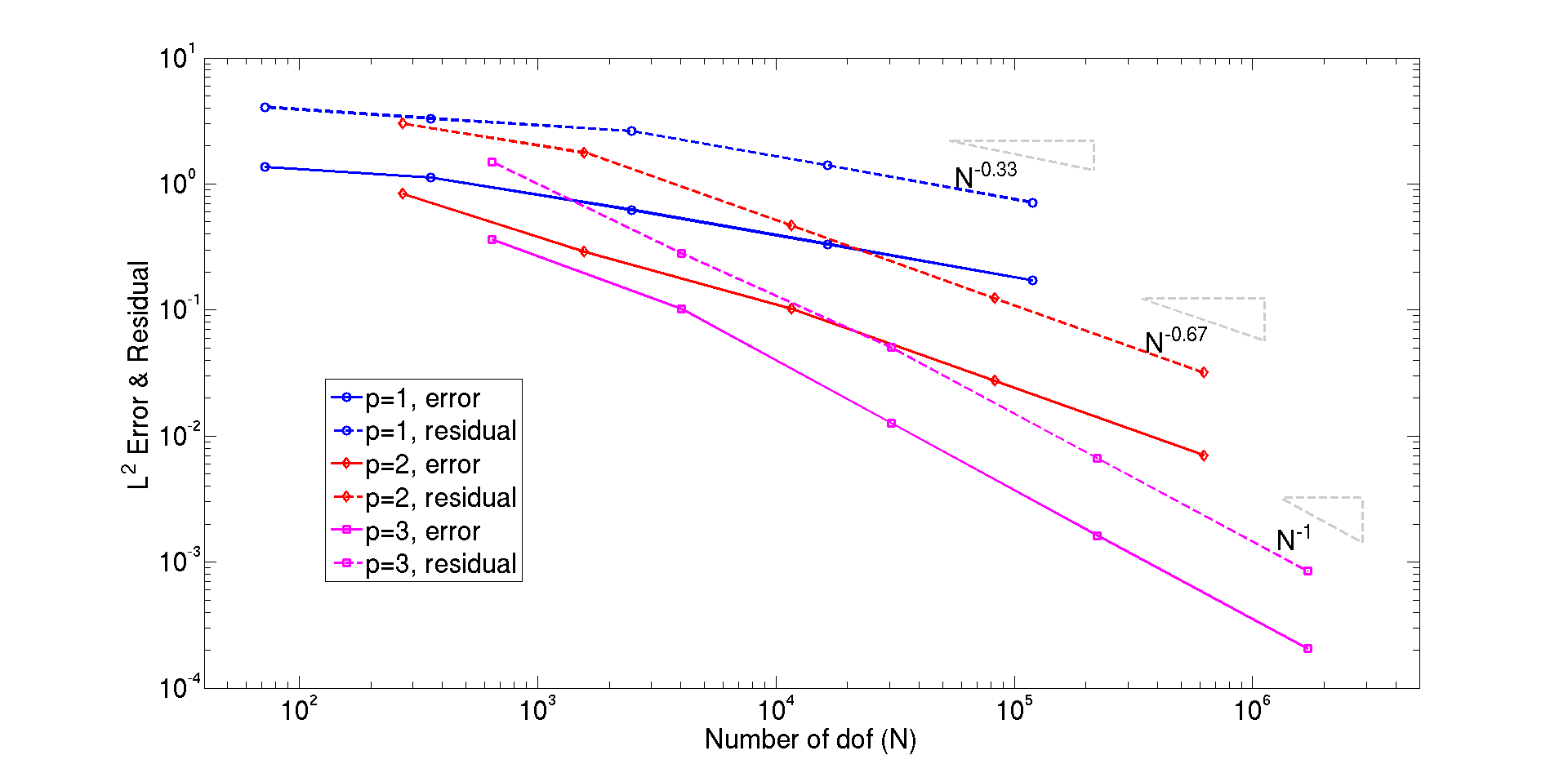}
  \caption{Tetrahedral meshes.}
  \label{fig:smooth_uw_rates_tet}
\end{subfigure}
\begin{subfigure}{\textwidth}
  \includegraphics[width=\textwidth]{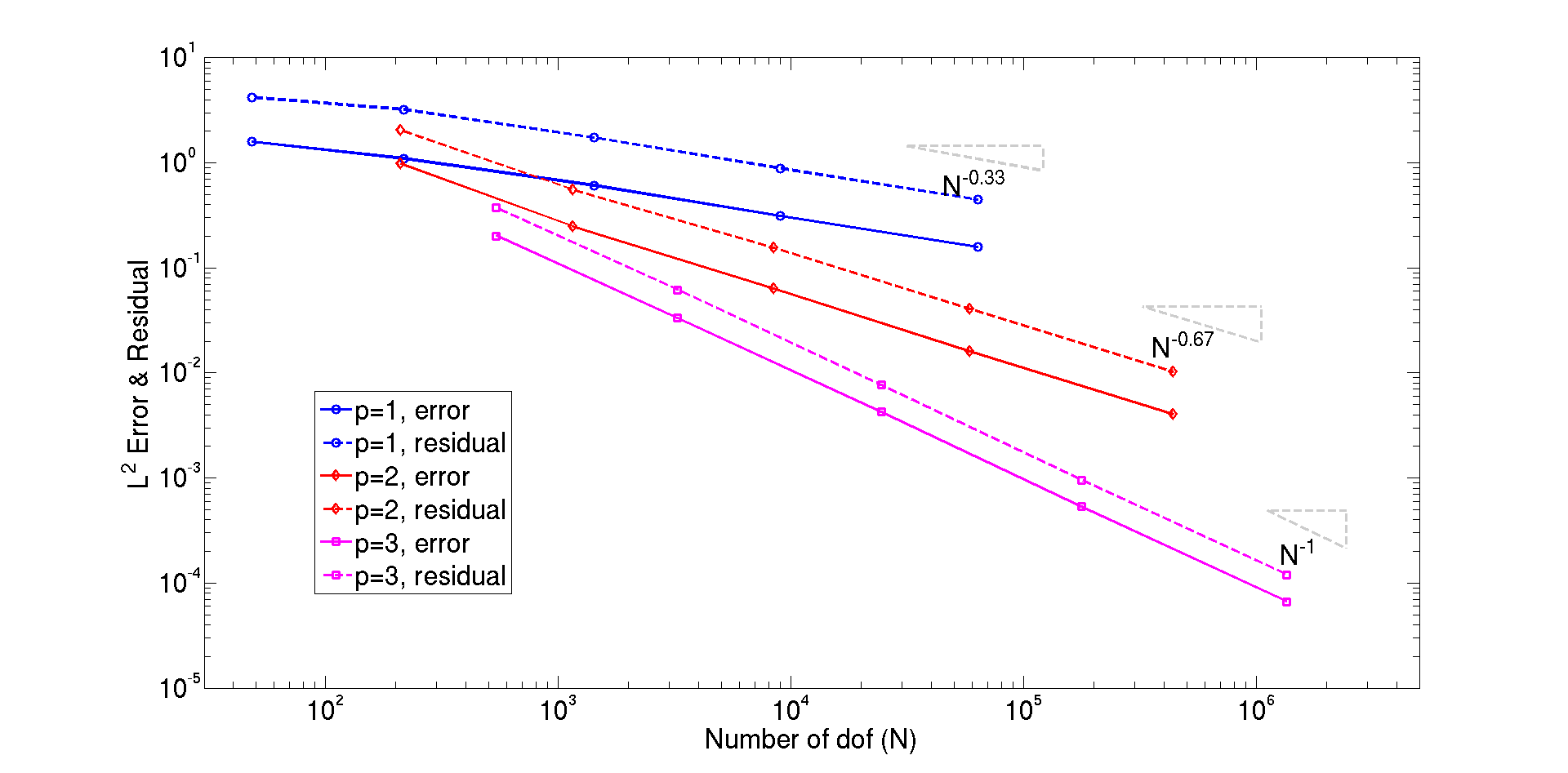}
  \caption{Hexahedral meshes.}
  \label{fig:smooth_uw_rates_hex}
\end{subfigure}
\caption{Rates for the case of smooth solution and 
  ultraweak formulation.}
\label{fig:smooth_uw_rates}
\end{figure}

The observed rates of convergence of the error
$\| E - E_h\|_{\Hcurl\oh}$ and the residual norm $\eta$ are shown \jrev{in
Figure~\ref{fig:smooth_primal_rates_tet}}. The rates are optimal. This
suggests that the results of Corollary~\ref{cor:primalMaxwell} may
hold with other choices of spaces. 
\jrev{Results  analogous to those in~\cite{BoumaGopalHarb14,CarstGalliHellw14}
that allow one to reduce the degree of the 
test space $Y_h$ while maintaining optimal convergence rates are
currently not known for the 
the Maxwell problem.}

We also present similar results obtained using cubic meshes using
$\Hcurl\om$-conforming \Nedelec\ hexahedron of the first type. Namely,
$X_h$ and $Y_h$ are set by~\eqref{eq:XYexpt} after revising $N_p(K)$
to $Q_{p-1,p,p} (K) \times Q_{p,p-1,p}(K) \times Q_{p,p,p-1}(K)$ where
$Q_{l,m,n}(K)$ denotes the set of polynomials of degree at most $l,m,$
and $n$ in the $x_1,x_2$ and $x_3$ directions, respectively. The
convergence rates reported in \jrev{Figure~\ref{fig:smooth_primal_rates_hex}}
are again optimal.

Before concluding this example, we also report convergence rates
obtained from the ultraweak formulation of~\eqref{eq:MaxUW}.  The
discrete spaces are now set by
\begin{subequations} \label{eq:XYexptUW}
\begin{align}
  X_h
  & = \{ ( E, H, \Et, \Ht) \in L^2(\om)^3 \times L^2(\om)^3 
    \times \Hhcurl{\d\oh} \times \Hhcurl{\d\oh} :\;
  \\ \nonumber 
  & \hspace{1.5cm}
    E|_K, H|_K \in P_{p-1}(K)^3, \;
    \Et|_{\d K}, \Ht|_{\d K}   \in \trpK N_{p}(K)
    \text{ for all } K \in \oh \}, 
  \\ 
  Y_h & = \{ (F, G) \in \Hcurl\oh \times \Hcurl\oh : \; 
        F|_K, G|_K \in N_{p+2}(K) \text{ for all } K \in \oh \}.
\end{align}  
\end{subequations}
Recall that the DPG computations require a specification of the
$Y$-norm. Using the observation (made in the proof of
Theorem~\ref{thm:MaxwellCycles}) that the adjoint graph norm is equivalent
to the natural norm in $\Hcurl\om^2$, we set 
\[
\Vert (E,H) \Vert_Y^2 = 
\sum_{K\in\oh} 
\left(
\| E \|_{L^2(\om)}^2 
+ \| H \|_{L^2(\om)}^2 
+ \| \ii \og \mu H - \curl E \|_{L^2(\om)}^2 
+ \|  \ii \og \eps E + \curl H \|_{L^2(\om)}^2
\right)
\]
in all computations involving the ultraweak formulation.  The results
reported in Figure~\ref{fig:smooth_uw_rates} again show optimal
convergence rates. Note that only the errors in the interior variables
$E$ and $H$ (in $L^2(\om)$-norm) are reported in the figure. To
compute errors in the interface variables, we must compute
approximations to fractional norms carefully
(see~\cite{CarstDemkoGopal14} for such computations in two
dimensions). Since the code does not yet have this capability in three
dimensions, we have not reported the errors in interface
variables.\hfill\qeg
\end{example}

\begin{example}[Singular solution]

  \begin{figure}
    \centering
    \includegraphics[width=.2\textwidth,angle=0]{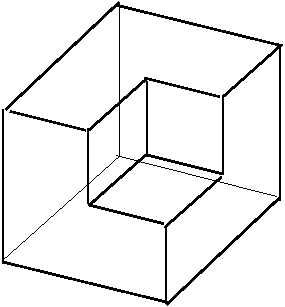}
    \hspace*{.1\textwidth}
    \includegraphics[width=.2\textwidth,angle=0]{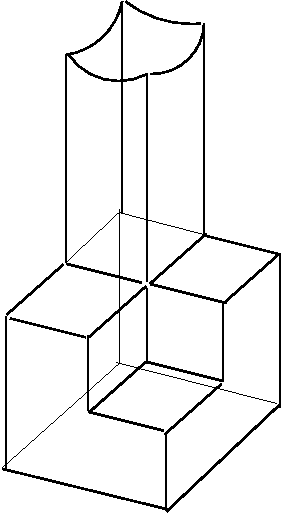}
    \caption{Construction of Fichera oven.}
    \label{fig:Fichera_microwave}
  \end{figure}
\begin{figure}
\centering
\includegraphics[width=0.18\textwidth]{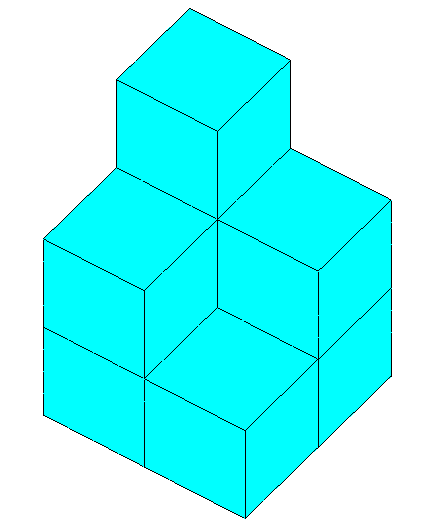}
\includegraphics[width=0.18\textwidth]{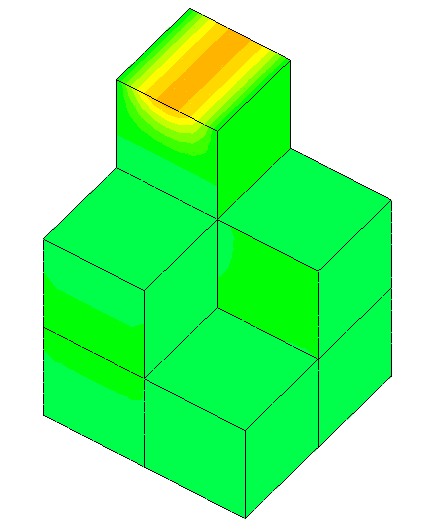}
\includegraphics[width=0.18\textwidth]{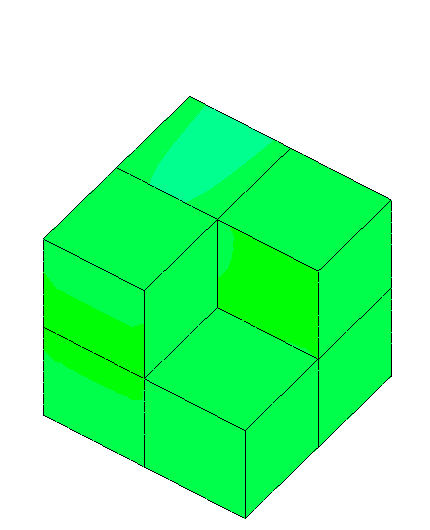}
\includegraphics[width=0.18\textwidth]{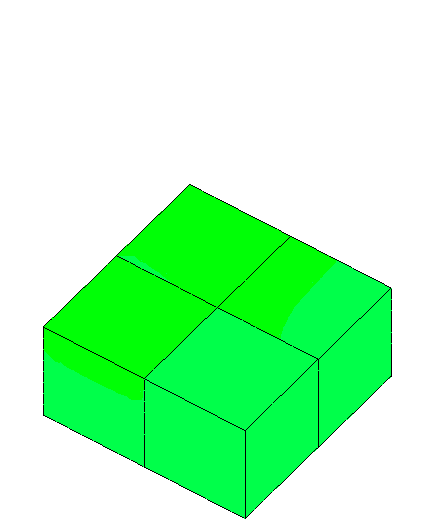}
\\
\includegraphics[width=0.18\textwidth]{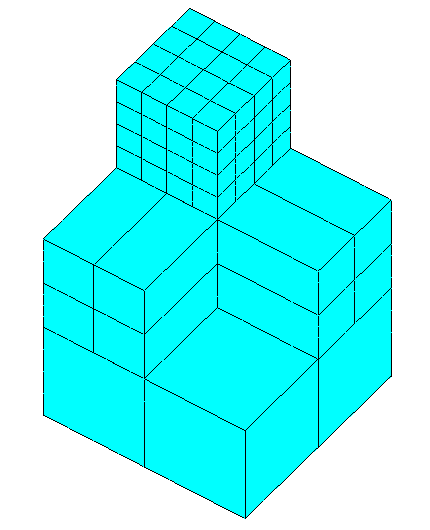}
\includegraphics[width=0.18\textwidth]{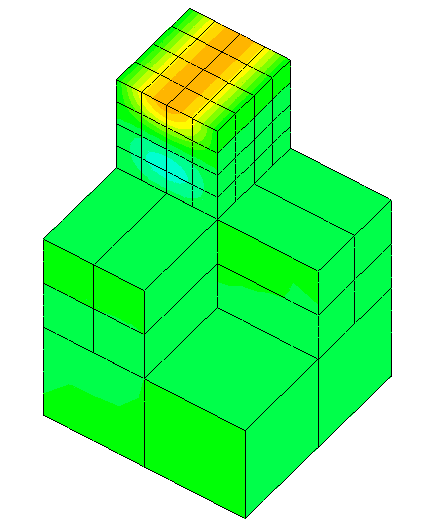}
\includegraphics[width=0.18\textwidth]{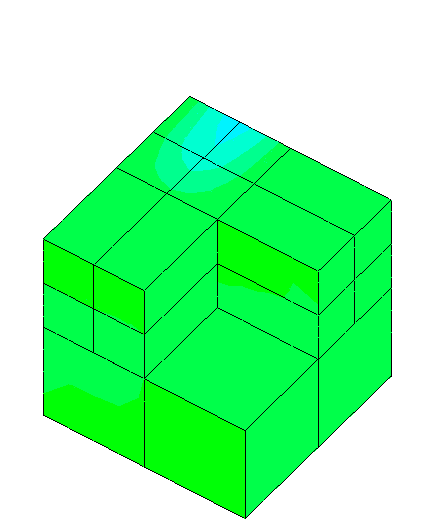}
\includegraphics[width=0.18\textwidth]{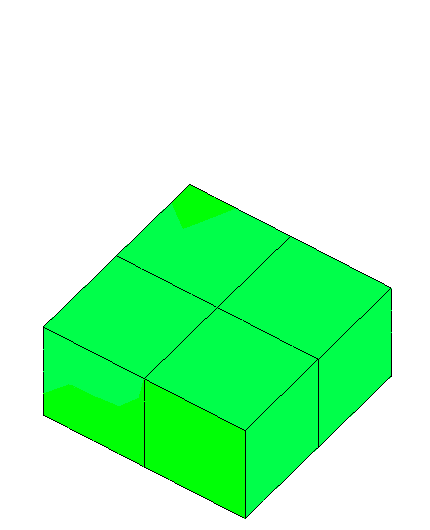}
\\
\includegraphics[width=0.18\textwidth]{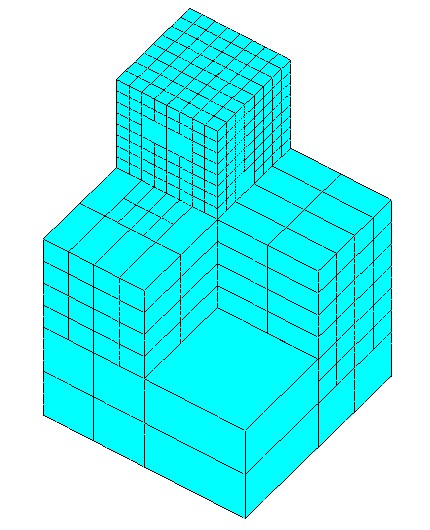}
\includegraphics[width=0.18\textwidth]{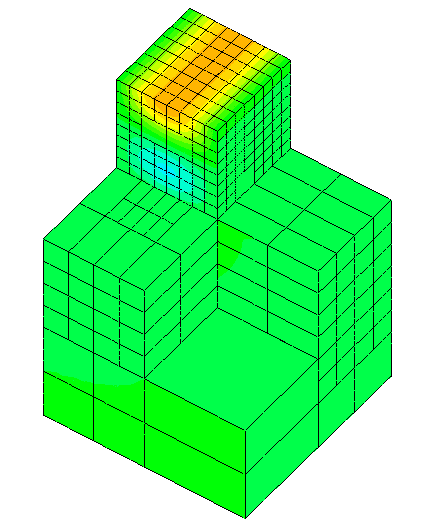}
\includegraphics[width=0.18\textwidth]{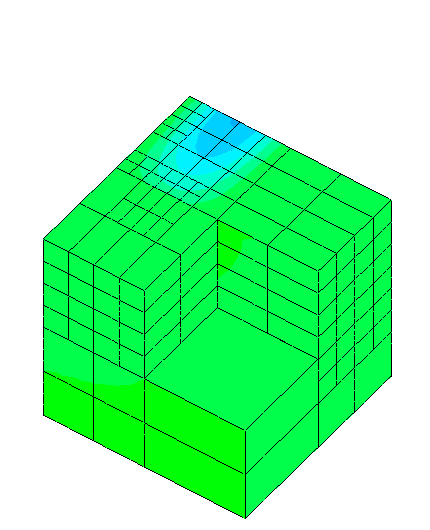}
\includegraphics[width=0.18\textwidth]{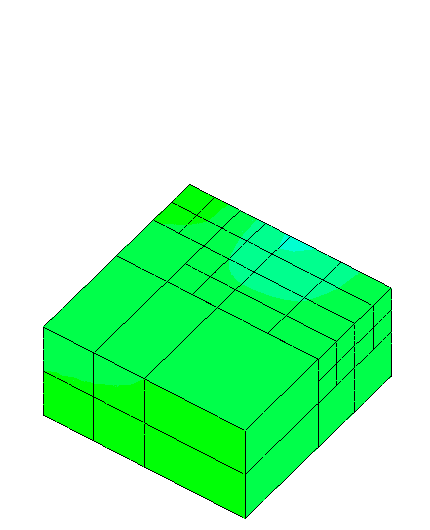}
\\
\includegraphics[width=0.18\textwidth]{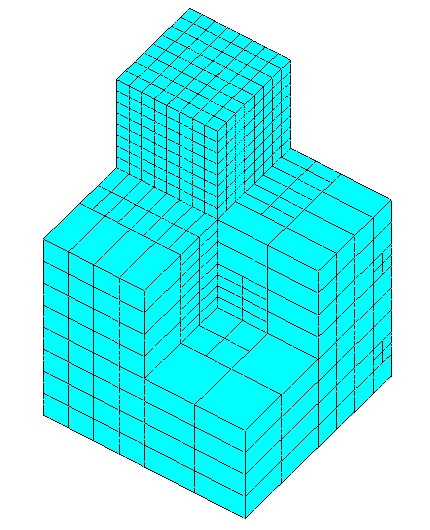}
\includegraphics[width=0.18\textwidth]{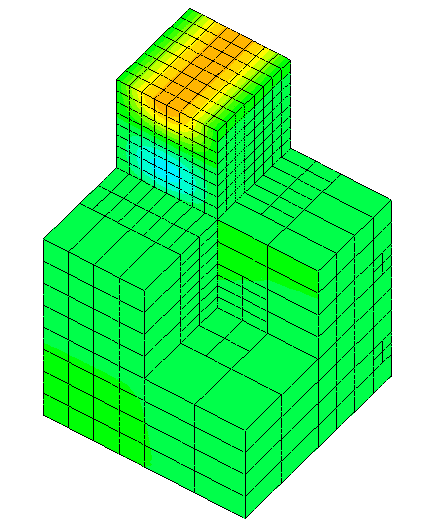}
\includegraphics[width=0.18\textwidth]{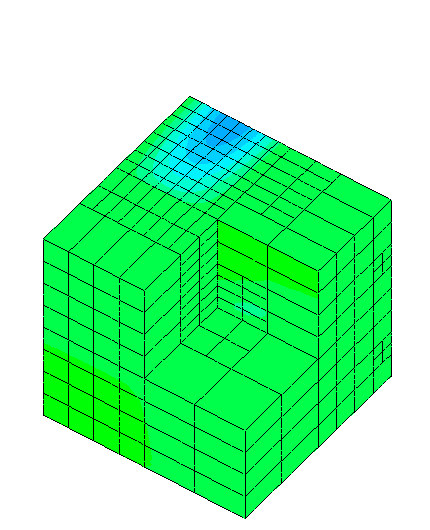}
\includegraphics[width=0.18\textwidth]{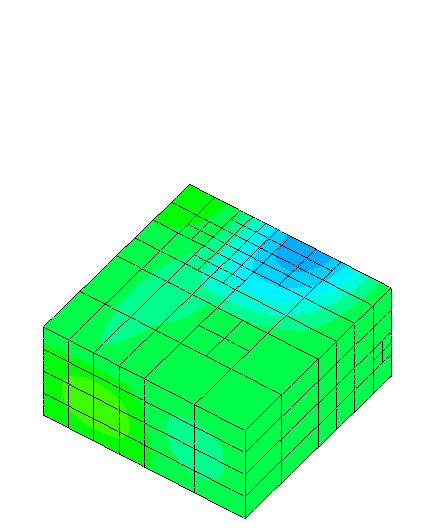}
\\
\includegraphics[width=0.18\textwidth]{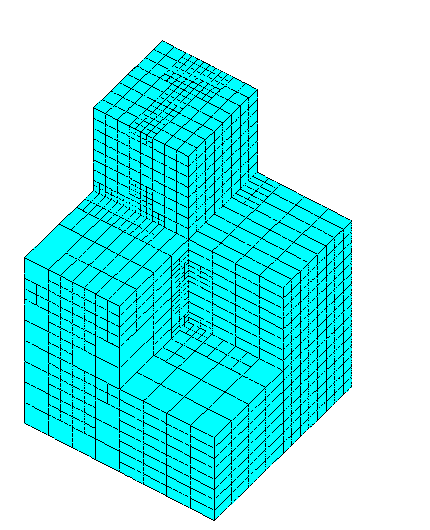}
\includegraphics[width=0.18\textwidth]{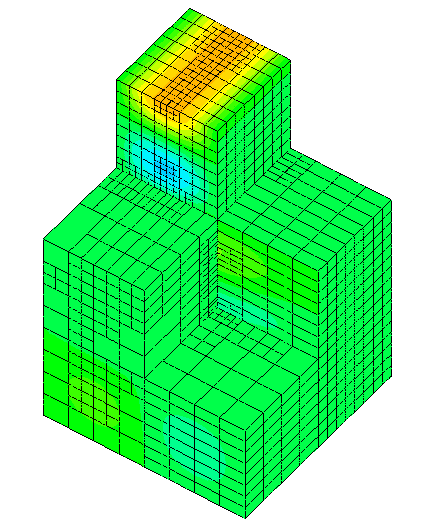}
\includegraphics[width=0.18\textwidth]{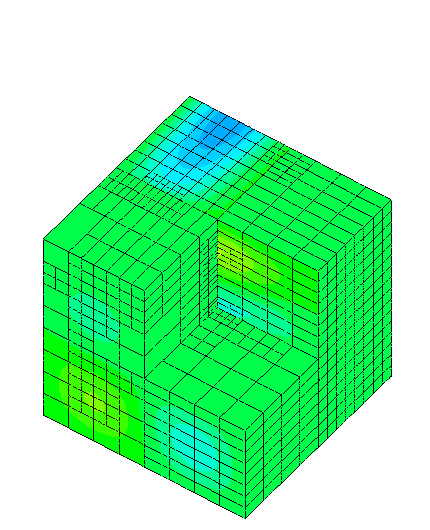}
\includegraphics[width=0.18\textwidth]{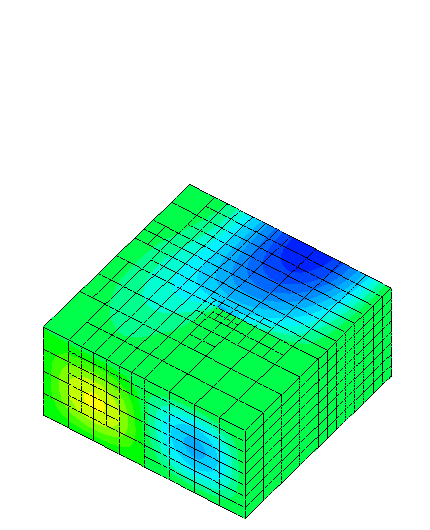}
\\
\includegraphics[width=0.18\textwidth]{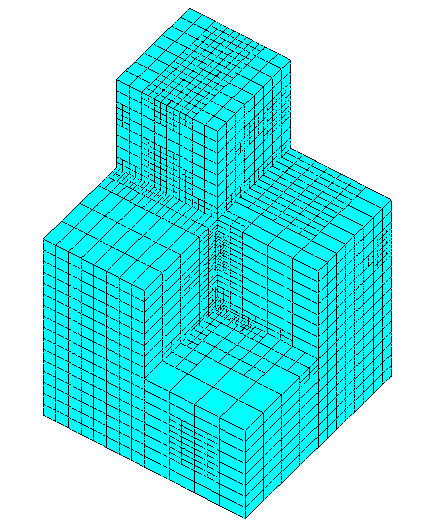}
\includegraphics[width=0.18\textwidth]{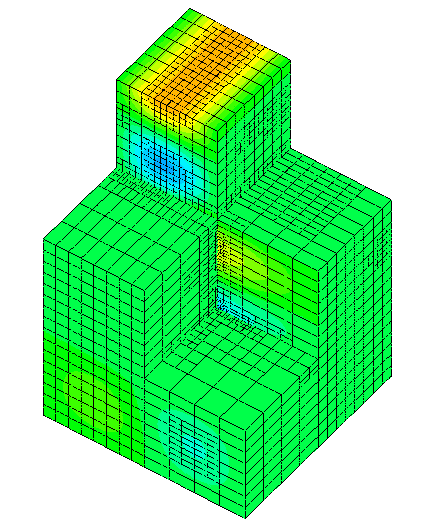}
\includegraphics[width=0.18\textwidth]{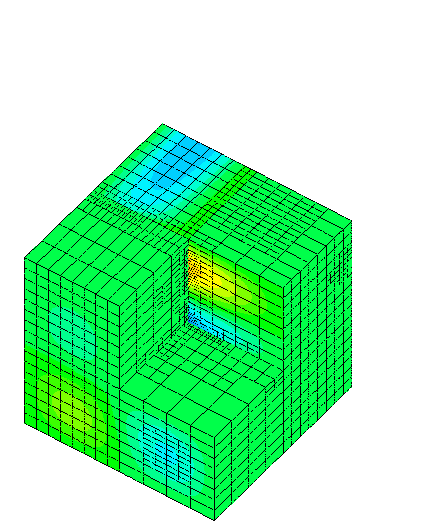}
\includegraphics[width=0.18\textwidth]{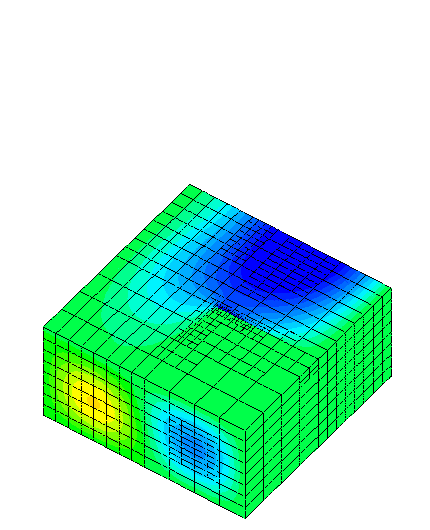}
\caption{Every other iterate (iterates 1,3,5,7,9 and 11) 
  from adaptive algorithm applied to solve the  Fichera
  oven problem with the primal formulation. Meshes (left) and the
  corresponding real part of $E_1$ are shown.}
\label{fig:microwave_primal}
\end{figure}
  To illustrate adaptive possibilities of DPG method and the
  difference between different variational formulations, we now
  present results from a ``Fichera oven'' problem. We start with with
  the standard domain with a Fichera corner obtained by refining a
  cube $(0,2)^3$ into eight congruent cubes and removing one of
  them. We then attach an infinite waveguide to the top of the oven
  and truncate it at a unit distance from the Fichera corner, as shown
  in Figure~\ref{fig:Fichera_microwave}.  Setting
$$
\epsilon = \mu = 1,\quad \omega = 5,
$$
we drive the problem with the first propagating waveguide mode,
$$
E^D_1 = \sin \pi x_2,\quad
E^D_2=E^D_3 = 0
$$
which is used for non-homogenous electric boundary condition
$n \times E = n \times E^D$ across the waveguide section.  Analogous
to a microwave oven model, we set the homogeneous perfect electric
boundary condition $n \times E = 0$ everywhere else on the boundary.
The above material data correspond to about 0.8 wavelengths per unit
domain.  In all the reported computations, we start with a uniform
mesh of eight quadratic elements that clearly does not even meet the
Nyquist criterion.  We expect the solution to develop strong
singularities at the reentrant corner and edges, but we do not know
the exact solution.

\begin{figure}
  \centering
  \begin{subfigure}{0.45\textwidth}
    \begin{tikzpicture}
      \begin{semilogxaxis}[
        xlabel=Degrees of freedom,
        ylabel=Residual norm $\eta$
        ]

        \addplot plot coordinates {
          (   1224 ,   0.16645E+01 )
          (   2054 ,   0.84174E+00 )
          (  10208 ,   0.63574E+00 )
          (  39348 ,   0.55354E+00 )
          (  77332 ,   0.50959E+00 )
          ( 110142 ,   0.47632E+00 )
          ( 150660 ,   0.44076E+00 )
          ( 245786 ,   0.39261E+00 )
          ( 358210 ,   0.36662E+00 )
          ( 708786 ,   0.33661E+00 )

        };


      \end{semilogxaxis}
    \end{tikzpicture}
    \caption{Adaptivity for primal formulation.}
    \label{fig:microwave_primal_residual}
  \end{subfigure}
  \qquad
  \begin{subfigure}{0.45\textwidth}
    \begin{tikzpicture}
      \begin{semilogxaxis}[
        xlabel=Degrees of freedom,
        ylabel=Residual norm $\eta$
        ]

        \addplot plot coordinates {

          (   1224  ,   0.58556E+00 ) 
          (   2054  ,   0.48856E+00 ) 
          (   10208 ,   0.38958E+00 ) 
          (   17778 ,   0.35059E+00 ) 
          (   35374 ,   0.31035E+00 ) 
          (   90568  ,  0.27091E+00 ) 
          (   168242 ,  0.24020E+00 ) 
          (   225176 ,  0.22554E+00 ) 
          (   339240 ,  0.20098E+00 ) 

        };


      \end{semilogxaxis}
    \end{tikzpicture}
    \caption{Adaptivity for ultraweak formulation.}
    \label{fig:microwave_ultraweak_residual}
  \end{subfigure}
  \caption{Convergence of residual during adaptive iterations for the Fichera oven.}
\end{figure}
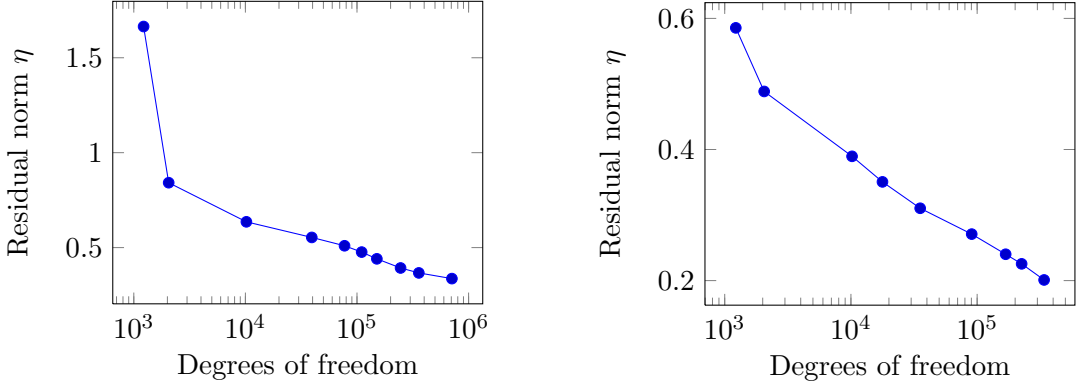

First, we report the results from the electric primal formulation,
choosing spaces again as in~\eqref{eq:XYexpt}.
Figure~\ref{fig:microwave_primal} presents the evolution of the mesh
along with the corresponding real part of the first component of
electric field $E_1$.  Since we do not have the exact solution for
this problem, we display convergence history using a plot of the
evolution of the computed residual $\eta$ in
Figure~\ref{fig:microwave_primal_residual}.  (Recall that theoretical
guidance on the similarity of behaviors of error estimator $\eta$ and
the error is provided by Theorem~\ref{thm:dpg}.) Clearly, the figure
shows the residual is being driven to zero during the adaptive
iteration.

Next, we solve the same problem using the ultraweak formulation with
the spaces set as in~\eqref{eq:XYexptUW} and the $Y$-norm set to the
adjoint graph norm as in the previous example.  The convergence
history of the residual norm $\eta$ is displayed in
Figure~\ref{fig:microwave_ultraweak_residual}. The evolution of the
mesh along with the real part of $E_1$ is illustrated in
Figure~\ref{fig:microwave_ultraweak}.

\begin{figure}
\centering
\includegraphics[width=0.19\textwidth]{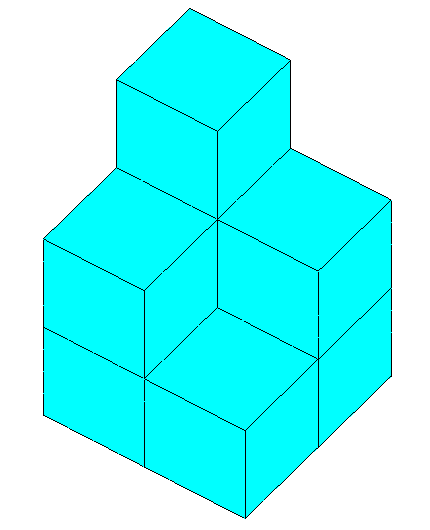}
\includegraphics[width=0.19\textwidth]{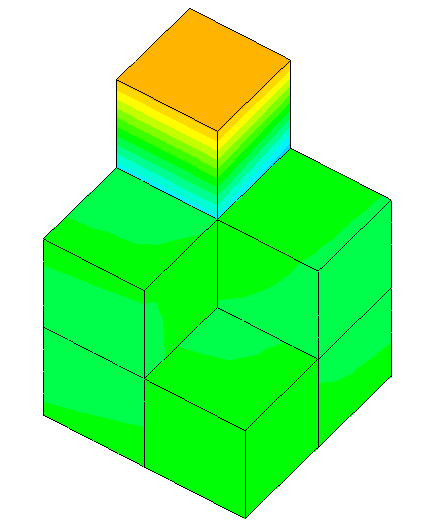}
\includegraphics[width=0.19\textwidth]{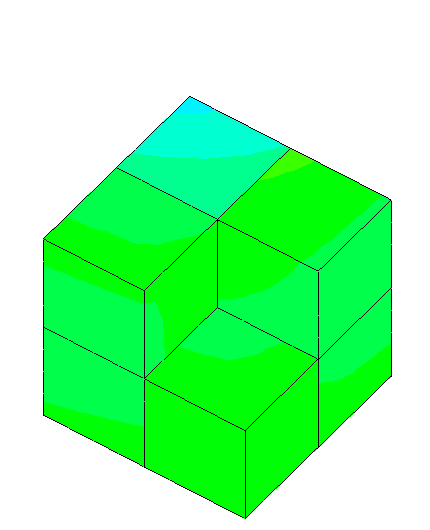}
\includegraphics[width=0.19\textwidth]{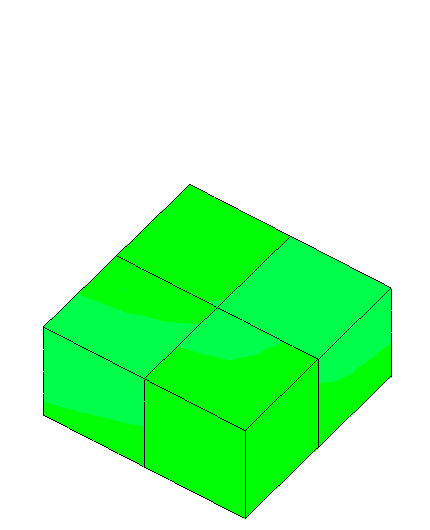}
\\
\includegraphics[width=0.19\textwidth]{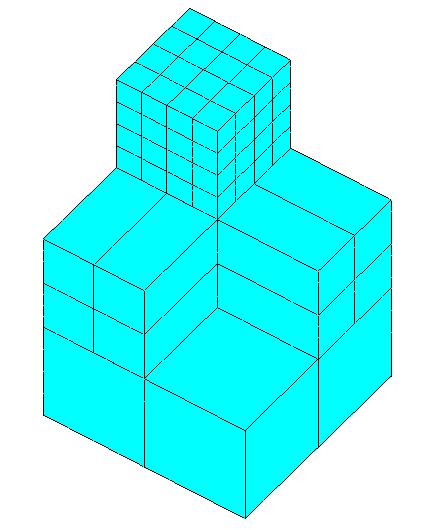}
\includegraphics[width=0.19\textwidth]{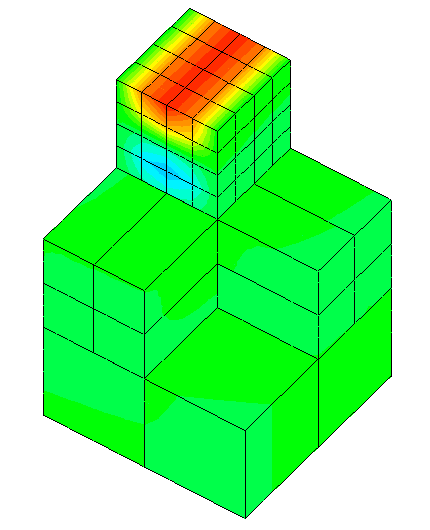}
\includegraphics[width=0.19\textwidth]{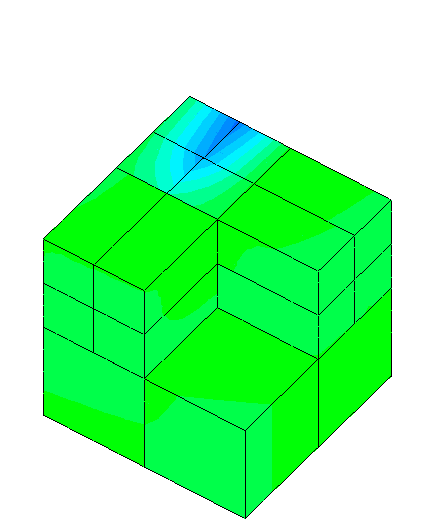}
\includegraphics[width=0.19\textwidth]{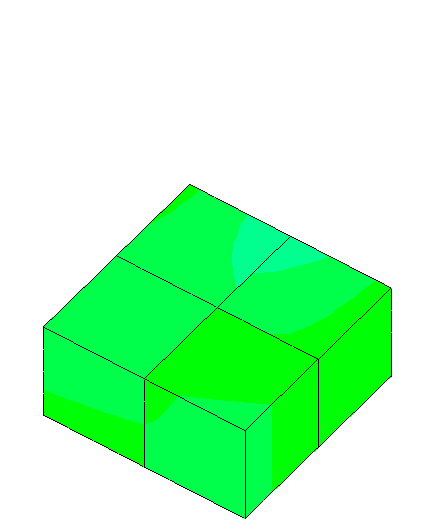}
\\
\includegraphics[width=0.19\textwidth]{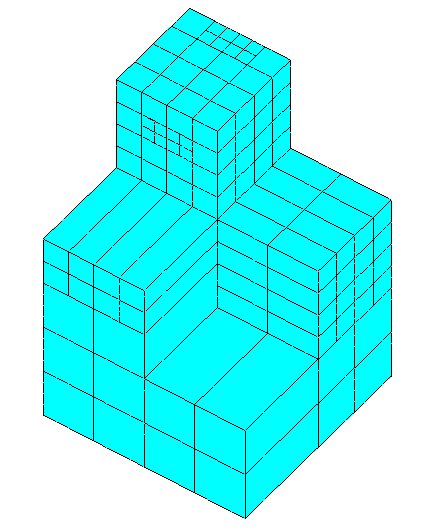}
\includegraphics[width=0.19\textwidth]{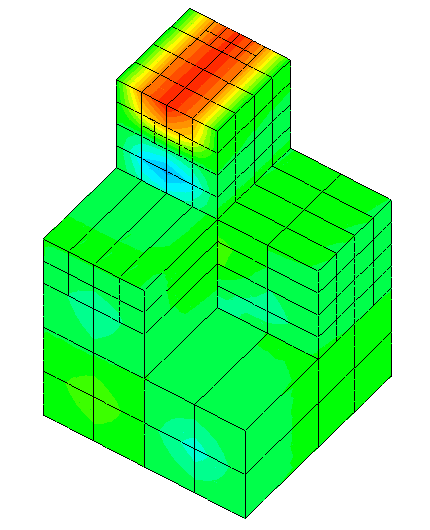}
\includegraphics[width=0.19\textwidth]{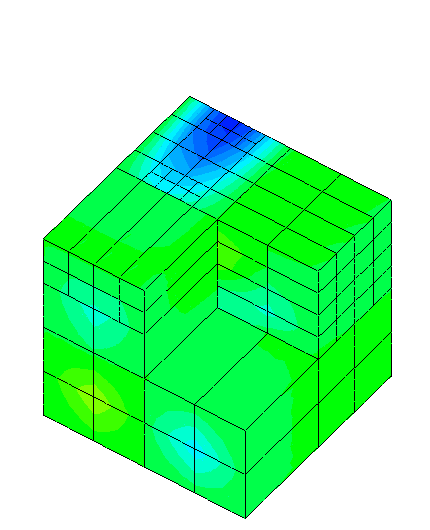}
\includegraphics[width=0.19\textwidth]{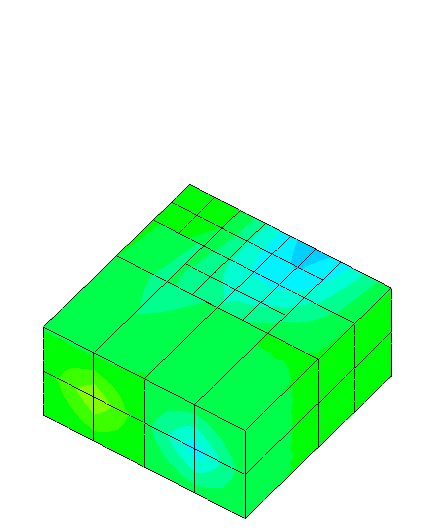}
\\
\includegraphics[width=0.19\textwidth]{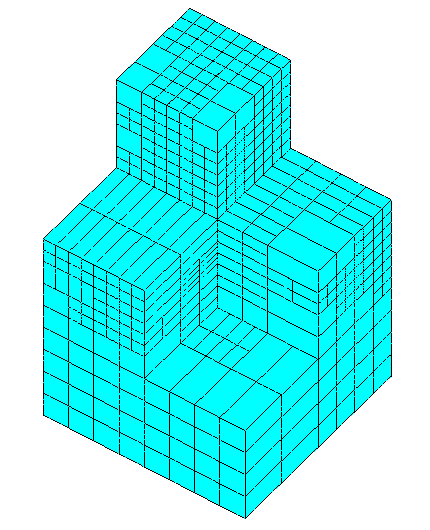}
\includegraphics[width=0.19\textwidth]{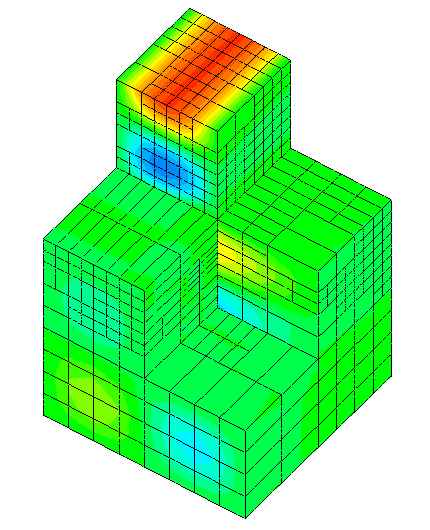}
\includegraphics[width=0.19\textwidth]{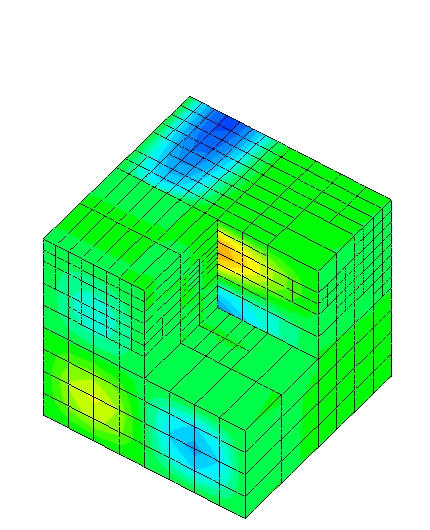}
\includegraphics[width=0.19\textwidth]{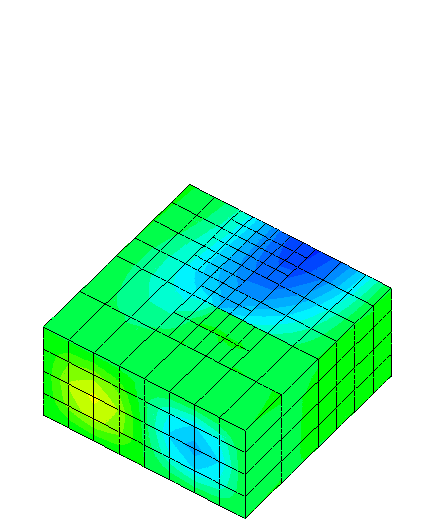}
\\
\includegraphics[width=0.19\textwidth]{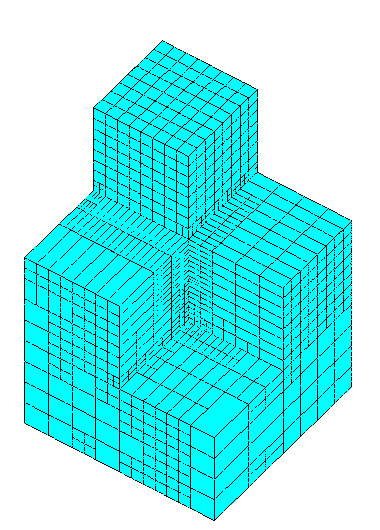}
\includegraphics[width=0.19\textwidth]{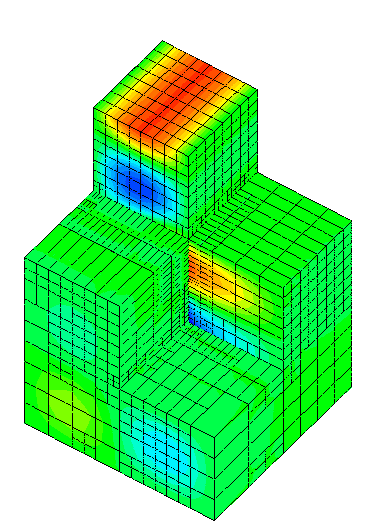}
\includegraphics[width=0.19\textwidth]{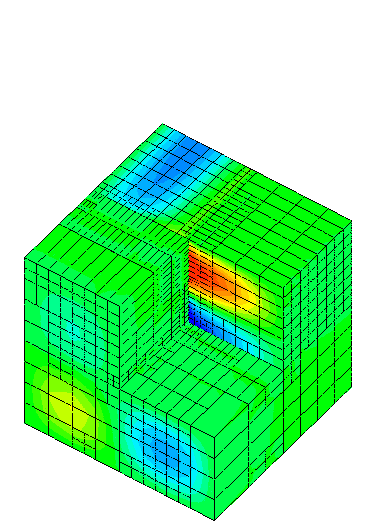}
\includegraphics[width=0.19\textwidth]{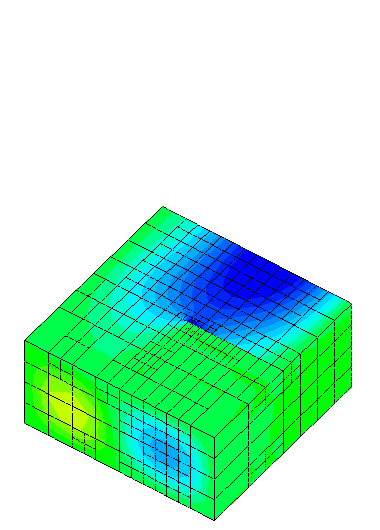}
\caption{Iterates 1,3,5,7, and 9 from adaptive algorithm for Fichera
  oven problem with the ultraweak formulation. Meshes (left) and the
  corresponding real part of $E_1$ are shown.}
\label{fig:microwave_ultraweak}
\end{figure}

\begin{figure}
\centering
\begin{tabular}{c|c}
\begin{subfigure}{0.49\textwidth}
  \includegraphics[width=0.45\textwidth]{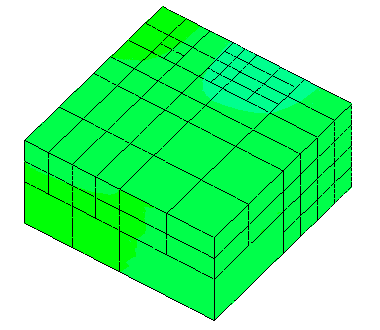}
  \includegraphics[width=0.45\textwidth]{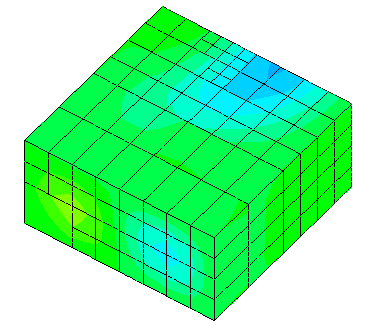}
  \caption{Primal (left) and ultraweak (right) iterate~6.}
\end{subfigure}
&
\begin{subfigure}{0.49\textwidth}
  \includegraphics[width=0.45\textwidth]{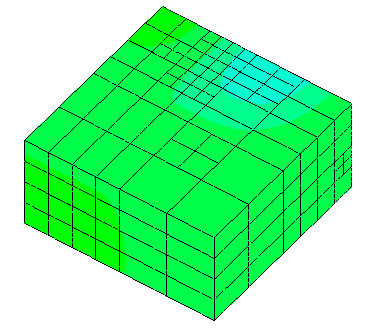}
  \includegraphics[width=0.45\textwidth]{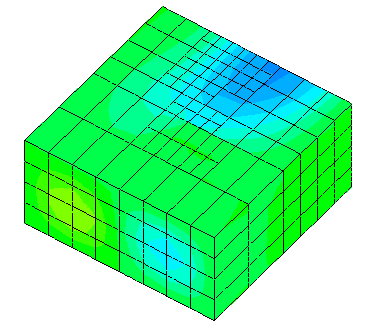}
  \caption{Primal (left) and ultraweak (right) iterate~7.}
\end{subfigure}
\\
& 
\\
\hline
\\ 
& 
\\
\begin{subfigure}{0.49\textwidth}
  \includegraphics[width=0.45\textwidth]{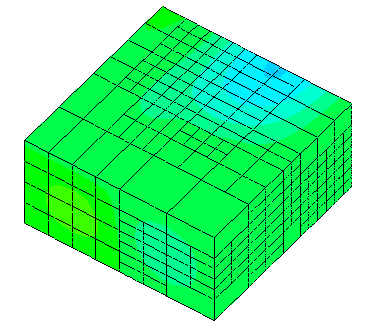}
  \includegraphics[width=0.45\textwidth]{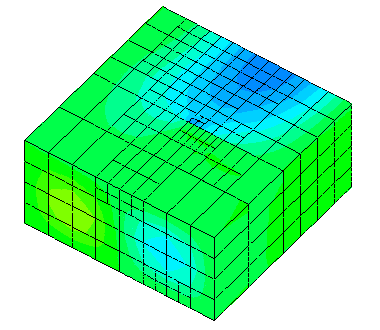}
  \caption{Primal (left) and ultraweak (right) iterate~8.}
\end{subfigure}
&
\begin{subfigure}{0.49\textwidth}
  \includegraphics[width=0.45\textwidth]{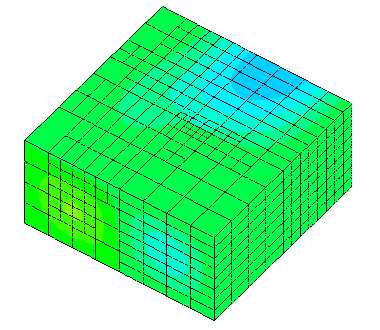}
  \includegraphics[width=0.45\textwidth]{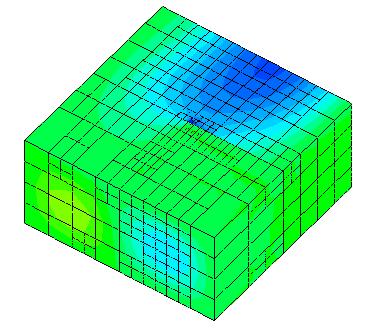}
  \caption{Primal (left) and ultraweak (right) iterate~9.}
\end{subfigure}
\end{tabular}
\caption{Same scale comparison of real part of $E_1$ computed by the
  primal and the ultraweak formulations. A middle slice (through
  singular vertex) of adaptive iterates 6,7,8 and 9 are shown.  }
\label{fig:comparison}
\end{figure}

It is illustrative to visualize the difference between the two
different DPG formulations and the accompanying convergence in
different norms. Figure~\ref{fig:comparison} presents a side-by-side
comparison of the real part of the electric field component $E_1$
obtained using the primal (left) and ultraweak (right) formulations.
The same color scale (min $= -1$, max $= 1$) is applied to both solutions
in this figure (whereas the scales of
Figures~\ref{fig:microwave_primal} and~\ref{fig:microwave_ultraweak}
are not identical).  Obviously, the meshes are different, but they are
of comparable size, so we believe the comparison is fair.  The primal
method, which delivers solution converging in the stronger
$\Hcurl\om$-norm, ``grows'' the unknown solution slower, whereas the
ultraweak formulation converging in the weaker $L^2(\om)$-norm seems
to capture the same solution features faster. Both methods ultimately
approximate the same solution but at different speeds and the ultraweak
formulation seems to be a winner. Recall that the number of interface
unknowns for both formulations is identical, but the total number of
unknowns for the ultraweak formulation is higher, i.e., the ultraweak
formulation requires a larger number of local (element-by-element)
computations.\hfill\qeg
\end{example}

\section{Conclusion}

\jrev{In addition to presenting the first analysis of DPG methods for
  Maxwell equations, we have presented a technique that considerably
  simplifies analysis of various DPG methods. The idea is to inherit
  the stability of broken formulations using the known stability of
  unbroken standard formulations and is the content of
  Theorem~\ref{thm:hybrid}. To obtain discrete stability for the
  Maxwell discretization, a new Fortin operator was constructed in
  Theorem~\ref{thm:Fortin}.}

\jrev{We have shown how certain duality identities (proved in
  Theorem~\ref{thm:duality}) can be used to verify a critical
  assumption (Assumption~\ref{asm:hybrid}) involving an interface
  inf-sup condition.  During this process, we have provided a 
  simple technique
  to prove duality identities 
(see \eqref{eq:sup2a} and Lemma~\ref{lem:duality})
  like
\[
    \Vert \hat{E}_{\dv} \Vert_{\Hhdiv{\ptl K}}
    = 
    \Vert \hat{E}_{\dv} \Vert_{ [\Hhcurl{\d K} ]^*},
\]
(where the norm on the right hand side is the norm  in 
the space dual to $\Hhcurl{\d K}$).} 

\jrev{Finally, the connection between the stability of weak and strong
  formulations was made precise in Theorem~\ref{thm:MaxwellCycles}:
  the wellposedness of one of the displayed 
  six formulations imply the
  wellposedness of all others. Notwithstanding this result, the
  numerical experiments clearly showed the practical differences in
  convergences among the formulations.}

\jrev{Before concluding, we mention a few limitations of our analysis
  and open issues. Convergence results explicit in the polynomial
  degree $p$ are not obtained by the currently known Fortin
  operators. While construction of local Fortin operators provides one
  way to prove discrete stability, other avenues to reach the same
  goal (such as the analysis of~\cite{HeuerKarkuSayas14} assuming
  higher regularity, or the analysis of~\cite{BoumaGopalHarb14}
  extending the Strang lemma) may prove important.  Our analysis did
  not track the dependence on the wavenumber $\og$. 
  More complex techniques are likely to be needed for such
  parameter tracking~\cite{DemkoGopalMuga11a}, including in the
  unrelated important examples of advective singular perturbation
  problems~\cite{ChanHeuerBui-T14,CohenDahmeWelpe12,DemkoHeuer13}.}

\end{document}